\documentclass[12pt]{article}

\usepackage{times}
\usepackage{amsthm}
\usepackage{amsfonts}
\usepackage{bbm}

\usepackage{amsmath,amssymb}
\usepackage[pdftex]{color, graphicx} 

\usepackage{mathrsfs}
\usepackage{plain}
\usepackage{makeidx}
\theoremstyle{plain}
\newtheorem{thm}{Theorem}
\numberwithin{thm}{section}
\newtheorem{cor}[thm]{Corollary}
\newtheorem{lem}[thm]{Lemma}

\theoremstyle{definition}
\newtheorem{defn}[thm]{Definition}
\newtheorem{assumption}[thm]{Assumption}

\headsep \headheight

\newcommand*{\medcap}{\mathbin{\scalebox{1.5}{\ensuremath{\cap}}}}%

\theoremstyle{remark}
\newtheorem{oss}[thm]{Remark}

\usepackage[english]{babel} 
\usepackage[T1]{fontenc}
\usepackage{indentfirst}
\usepackage{fancyhdr}
\usepackage{afterpage}

\usepackage{cite}
\DeclareMathOperator{\sgn}{sgn}
\usepackage{mathtools}

\DeclarePairedDelimiter\abs{\lvert}{\rvert}%
\DeclarePairedDelimiter\norm{\lVert}{\rVert}%
\makeatletter
\let\oldabs\abs
\def\abs{\@ifstar{\oldabs}{\oldabs*}}
\let\oldnorm\norm
\def\norm{\@ifstar{\oldnorm}{\oldnorm*}}
\makeatother
\usepackage{wasysym}
\DeclareMathOperator*{\esssup}{ess\,sup}

\usepackage{scalerel}

\usepackage{amssymb,fge}

\usepackage{calrsfs}
\DeclareMathAlphabet{\pazocal}{OMS}{zplm}{m}{n}
\newcommand{\Ma}{\mathcal{M}}

\usepackage[titletoc,toc,title]{appendix}

\providecommand{\keywords}[1]
{
	\textbf{\textit{Keywords:}} #1
}

\usepackage{calc}

\usepackage{accents}
\newcommand{\dbtilde}[1]{\accentset{\approx}{#1}}

\makeatletter
\newcommand*{\rom}[1]{\expandafter\@slowromancap\romannumeral #1@}
\makeatother

\usepackage{enumitem}

\usepackage{authblk}

\title{Optimal regularity in time and space for stochastic porous medium equations }
\author[1]{Stefano Bruno \thanks{s.bruno@bath.ac.uk}}
\author[2]{Benjamin Gess \thanks{b.gess@mis.mpg.de}}
\author[1]{Hendrik Weber  \thanks{h.weber@bath.ac.uk}}
\affil[1]{\textit{University of Bath,}}
\affil[2]{\textit{Max Planck Institute for Mathematics in the Sciences and Universit{\"a}t Bielefeld}}

\date{}

\begin{document}
	\maketitle
	\maketitle
\begin{abstract}
	We prove optimal regularity estimates in Sobolev spaces in time and space for solutions to stochastic porous medium equations. The noise term considered here is multiplicative, white in time and coloured in space. The coefficients are assumed to be H{\"o}lder continuous and the cases of smooth coefficients of at most linear growth as well as $\sqrt{u}$ are covered by our assumptions. The regularity obtained is consistent with the optimal regularity derived for the deterministic porous medium equation in \cite{gess2017sobolev,gess2019optimal} and the presence of the temporal white noise.  
	The proof relies on a significant adaptation of velocity averaging techniques from their usual $L^1$ context to the natural $L^2$ setting of the stochastic case. 
	We introduce a new mixed kinetic/mild representation of solutions to quasilinear SPDE and use $L^2$ based a priori bounds to treat the stochastic term.  
\end{abstract} 
\tableofcontents
\makeatletter
\@starttoc{toc}
\makeatother


\section{Introduction} \label{Introduction_section}
We establish optimal regularity estimates in time and space for solutions to the following stochastic porous medium equation (SPME): 
\begin{gather} \label{degenerate_parabolic_hyperbolic_SPDE}
	\begin{cases}
		\partial_t u (t,x) = \Delta ( \abs{u}^{m-1} u ) (t,x) + \sum_{k=1}^{\infty} g_k(x,u(t,x))  \dot{\beta_k}(t) & \text{in} \ (0,T) \times \mathbb{T}^d \\ u(0)=u_0  & \text{in} \  \mathbb{T}^d
	\end{cases}
\end{gather}
where $m >1 $, $T>0$ denotes some time horizon, $ d \ge 1$, $\mathbb{T}^d = (\mathbb{R}/ \mathbb{Z})^d$ is the $d$-dimensional torus, and $u_0 \in L^{2 \alpha}  (\mathbb{T}^d)$ for some $ \alpha \in [\frac{1}{2},1]$, $(\beta_k)_{k \ge 1}$ is a sequence of independent real-valued standard Wiener processes defined on some probability space $(\Omega, \mathcal{F}, \mathbb{P})$ and $( g_k)_{k \ge 1} $ is a sequence of H{\"o}lder continuous diffusion coefficients. For the exact conditions see Assumption \ref{Assumption_diffusion_coefficient} below.

Stochastic porous medium equations are well-studied models describing nonlinear diffusion dynamics perturbed by noise \cite{barbu2016stochastic,da2014stochastic,prevot2007concise}.  Although the regularity of their solutions has received considerable attention \cite{kim2006stochastic, rockner2008non,gess2011strong, gess2018well, barbu2014operatorial,gess2012random}, so far,
all known regularity results for SPME 
are  restricted to a degree of spatial differentiability of  order less or equal to one. This is in contrast with the \textit{optimal} spatial regularity of the deterministic version of \eqref{degenerate_parabolic_hyperbolic_SPDE} which has been derived recently by the second author and his coauthors \cite{gess2017sobolev, gess2019optimal}.
Namely, if $g_k=0$, $u_0 \in L^1$ and $m>1$ then
\begin{gather*}
	u \in L^m(0,T; W^{\frac{2}{m}- ,m}(\mathbb{T}^d)) :=  \underset{\varepsilon >0 }{\medcap} \  L^m(0,T; W^{\frac{2}{m}-\varepsilon ,m}(\mathbb{T}^d)) .
\end{gather*}
We derive the same spatial regularity for \eqref{degenerate_parabolic_hyperbolic_SPDE}. To the best of our knowledge, this is the first spatial regularity result which gives more than one derivative for  \textit{any} quasilinear degenerate stochastic PDE. 

Our estimates depend on the interplay between the growth of the diffusion coefficient and the integrability of the initial data. 
\begin{assumption} \label{Assumption_diffusion_coefficient}
	Suppose that $ g_k \in C( \mathbb{T}^d \times \mathbb{R}) $ and that there exists a sequence $(\lambda_k)_{k \ge 1}$ of non-negative numbers satisfying $	D:= \sum_{k \ge 1} \lambda_k^2 < \infty $ such that for some $\alpha \in [\frac{1}{2},1]$
	\begin{equation*}
		\abs{g_k(x,v)} +  \abs{\nabla_x g_k(x,v)}  \le  \lambda_k \left( \abs{v}^{\frac{1}{2}} \mathbbm{1}_{\abs{v} \le 1} + \abs{v}^{\alpha} \mathbbm{1}_{\abs{v} \ge 1}   \right),
	\end{equation*}
	as well as
	
	\begin{equation*}
		\qquad \qquad  \qquad \quad \	\ \abs{\partial_v g_k(x,v)}  \le  \lambda_k \left( \abs{v}^{- \frac{1}{2}} \mathbbm{1}_{\abs{v} \le 1} + \abs{v}^{\alpha - 1} \mathbbm{1}_{\abs{v} \ge 1}   \right),
	\end{equation*}
	for all $ x \in \mathbb{T}^d$ and $v \in \mathbb{R}$. 
\end{assumption}

This is the main result of this paper.
\begin{thm} \label{Main_theorem}
	Let $ \alpha \in [\frac{1}{2}, 1]$,  $m > 1$,  $u_0 \in  L^{2 \alpha} (\mathbb{T}^d) $  and assume that $g_k$ satisfies the Assumption \ref{Assumption_diffusion_coefficient} for this value of $\alpha$. Let $u$ be a solution to \eqref{degenerate_parabolic_hyperbolic_SPDE} in the sense of Definition \ref{kinetic_solution} (Section \ref{Section_kinetic_measure_solution}) below. Then, for all
	\begin{align*} 
		\sigma_x \in \left[0, \frac{2}{m} \right)  , 
	\end{align*}
	we have
	\begin{align*}
		u \in L^{ m}(\Omega; L^{ m}(0,T;W^{\sigma_x,  m}(\mathbb{T}^d))).
	\end{align*}
	Moreover, the following estimate holds:
	\begin{gather} \label{Main_theorem_estimate_statement_small_velocity}
		\| u \|_{L^{ m}(\Omega; L^{ m}(0,T;W^{\sigma_x,  m}(\mathbb{T}^d)))} \lesssim   \| u_0 \|_{L^{2 \alpha}}^{  2 \alpha}    +  1.
	\end{gather}
	Here and in the proof  $ \ \lesssim $ denotes a bound that holds up to a multiplicative constant that only depends on $  \alpha, m,d, D, \sigma_x$ and $T$.
\end{thm}

The \textit{optimal} time regularity for the deterministic porous medium equation (PME) driven by an $L^1$ (in time and space) forcing term $S$ has also been derived in \cite{gess2019optimal}. If $u_0 \in L^1$  and the noise term $\sum_{k=1}^{\infty} g_k(x,u(t,x))  \dot{\beta_k}(t)$ in \eqref{degenerate_parabolic_hyperbolic_SPDE} is replaced by $S$ then  
\begin{gather*}
	u \in W^{1-, 1} (0,T; L^1 (\mathbb{T}^d)) :=  \underset{\varepsilon >0 }{\medcap} \ W^{1-\varepsilon, 1} (0,T; L^1 (\mathbb{T}^d)).
\end{gather*}
In the stochastic case, such strong regularity cannot be expected to hold. Due to the low temporal regularity of the noise, the time regularity of $u$ cannot be more than  $\frac{1}{2}-$. In this case, however, we can get better
temporal integrability. 	


\begin{thm} \label{Main_theorem_time}
	Let $ \alpha \in [\frac{1}{2}, 1]$,  $m > 1$,  $u_0 \in  L^{2 \alpha} (\mathbb{T}^d) $  and assume that $g_k$ satisfies the Assumption \ref{Assumption_diffusion_coefficient} for this value of $\alpha$. Let $u$ be a solution to \eqref{degenerate_parabolic_hyperbolic_SPDE} in the sense of Definition \ref{kinetic_solution} (Section \ref{Section_kinetic_measure_solution}) below. Then, for all   
	\begin{align*}
		\sigma_t \in \left[0, \frac{1}{2} \right) ,
	\end{align*}
	we have
	\begin{align*}
		u \in L^1(\Omega; W^{ \sigma_t,2}(0,T;L^{1}(\mathbb{T}^d))).
	\end{align*}
	Furthermore, the following estimate holds:
	\begin{gather} \label{Main_theorem_estimate_statement_time_small_v}
		\| u \|_{L^1(\Omega; W^{\sigma_t,2}(0,T;L^{1}(\mathbb{T}^d)))} \lesssim    \| u_0 \|_{ L^{2 \alpha}}^{  2 \alpha}  + 1,
	\end{gather}
	where  $  \lesssim $ denotes a bound that holds up to a multiplicative constant that only depends on $  \alpha, m,d, D, \sigma_t$ and $T$.
\end{thm}
\begin{oss}
	A scaling argument in Lemma \ref{Scaling_lemma_time_integrability_appendix} (Appendix \ref{Scaling_appendix}) suggests that the time integrability in Theorem \ref{Main_theorem_time} is optimal for \eqref{degenerate_parabolic_hyperbolic_SPDE} when $g_k=0$ and $u_0 \in L^1(\mathbb{T}^d)$ (i.e. $\alpha = \frac{1}{2}$).
\end{oss}
\begin{oss}
	The results in Theorem \ref{Main_theorem} and Theorem \ref{Main_theorem_time}  still hold for random initial datum under the stochastic integrability condition, $u_0 \in L^{m-1+ 2 \alpha }(\Omega; L^{2 \alpha}(\mathbb{T}^d))$ required by Lemma \ref{Control_higher_powers_Lp} below.
\end{oss}
The following corollary follows from interpolating the results of Theorem \ref{Main_theorem} and Theorem \ref{Main_theorem_time}.
\begin{cor} \label{Space_time_corollary}
 Under the assumptions of Theorem \ref{Main_theorem} and Theorem \ref{Main_theorem_time}, we define, for all $\theta \in (0,1)$
\begin{equation*}
	\frac{1}{p}= \frac{1- \theta}{m} + \theta, \qquad \frac{1}{q}= \frac{1-\theta}{m} + \frac{\theta}{2}.
\end{equation*}
Let $ 	\sigma_t \in \left[0, \frac{1}{2} \right)$ and $\sigma_x \in \left[0, \frac{2}{m} \right)$. Then, we have 
	\begin{equation*}
		u \in L^{ p} \left(\Omega;   W^{\theta \sigma_t, q}(0,T; W^{(1-\theta) \sigma_x,  p}(\mathbb{T}^d))  \right),
	\end{equation*}
with estimate
\begin{gather*}
	\| u \|_{L^{ p} \left(\Omega;   W^{\theta \sigma_t, q}(0,T; W^{(1-\theta) \sigma_x,  p}(\mathbb{T}^d))  \right)} \lesssim    \| u_0 \|_{ L^{2 \alpha}}^{  2 \alpha}  + 1.
\end{gather*}
	Here and in the proof  $ \ \lesssim $ denotes a bound that holds up to a multiplicative constant that only depends on $ \theta, \alpha, m,d, D,  \sigma_t,  \sigma_x$ and $T$.
\end{cor}

\noindent The estimates in Theorem \ref{Main_theorem} and Theorem \ref{Main_theorem_time} cover the case $g_k(x,u) = \mu_k e_k(x) \sqrt{u}$ where $(e_k)_{k \ge 1}$ 
are suitable basis functions (e.g trigonometric polynomials) and $\mu_k \geq 0$ satisfy a decay condition, depending on the choice of $e_k$.
Stochastic equations with $\sqrt{u}$ noise coefficients are known to play an important role in population dynamics. Most prominently, in the linear case $m=1$ and for $d=1$, the stochastic heat equation with $\sqrt{u}$ coefficients driven by  space-time white noise  arises as the scaling limit of independent branching Brownian motions and  describes the evolution in time and space of the density of the so-called super-Brownian motion or Dawson--Watanabe superprocess \cite{perkins2002part,etheridge2000introduction}. 
At least formally, the nonlinear stochastic porous medium equation
\begin{equation} \label{nonlocal_derivation_stochastic_diffusion_introduction}
	\partial_t u = \Delta (u^2) + b \sqrt{u} \xi,
\end{equation}
where $b$ is a constant 
and $\xi$ is  space-time white noise,
has also been derived as a scaling limit for mean-field interacting branching processes,  see \cite{meleard1993interacting,dareiotis2020porous}.
The interpretation of $u$ as the density of a population makes it natural to work with measure valued initial data. We emphasize that for $\alpha=1/2$ the estimates derived in the present work depend only on the $L^1$ norm of the initial data and, therefore, hold uniformly over suitable approximations of measure valued initial data.

The proofs of Theorem \ref{Main_theorem} and Theorem \ref{Main_theorem_time} are based on the kinetic approach originally introduced by Lions, Perthame and Tadmor \cite{lions1994kinetic} in the context of (deterministic) scalar conservation laws. The main idea of this method is to introduce an auxiliary variable $v \in \mathbb{R}$ and to study the so-called kinetic function $ \chi(t,x,v):= \mathbbm{1}_{v < u(t,x)} - \mathbbm{1}_{v < 0} $. This amounts to applying a nonlinear discontinuous function to the solution $u$ which can be recovered by integrating $\chi$ in $v$. The kinetic form for the deterministic part of \eqref{degenerate_parabolic_hyperbolic_SPDE} (i.e. when $ g_k=0 $) is
\begin{equation} \label{kinetic_form_introduction}
	\partial_t \chi  - m \abs{v}^{m-1} \Delta_x \chi =  \partial_v n,
\end{equation}
for some non-negative dissipation (or kinetic) measure $n$. The main advantage of writing the kinetic form is that \eqref{kinetic_form_introduction} can be treated as a linear equation in $\chi$. The linearity of \eqref{kinetic_form_introduction} allows to use Fourier analytic techniques and to derive regularity estimates on $ \int_{v} \chi \ dv$ by means of suitable microlocal decompositions in Fourier space.  This is the basis of the so-called averaging estimates \cite{lions1994kinetic,tadmor2007velocity} and the optimal regularity results in \cite{gess2017sobolev,gess2019optimal} for the PME. With this approach, the analysis of the (spatial) regularity is limited by the presence of the kinetic measure. 
At least formally 
this measure corresponds  to 
\begin{equation} \label{kinetic_measure_expression_introduction}
	n = \delta_{u=v}   |\nabla (|u|^{\frac{m-1}{2}} u)|^2.  	
\end{equation}
In \cite{gess2019regularity,gess2017sobolev} it was shown that it is possible to improve the regularity for hyperbolic and parabolic conservation laws by exploiting the finiteness of  singular moments of the kinetic measure. 
This control was essential to prove the optimal regularity results for the PME in \cite{gess2017sobolev,gess2019optimal}. In particular, the finiteness of these singular moments $\|\abs{v}^{- \gamma} n \|_{L^{1}_{t,x,v}}$ highlights the $L^1$ space as the natural framework for the derivation of the optimal regularity estimates for the PME. This is the setting chosen in \cite{gess2019optimal} and, therefore, the inclusion of a deterministic (time and space) $L^1$  forcing term in \eqref{kinetic_form_introduction} does not complicate the analysis.

%

This is fundamentally different from the inclusion of a stochastic forcing term $\sum_{k=1}^{\infty} g_k(x,u(t,x))  \dot{\beta_k}(t)$ in \eqref{kinetic_form_introduction}. In this case,
the kinetic form of \eqref{degenerate_parabolic_hyperbolic_SPDE} can be written as follows:
\begin{equation} \label{kinetic_form_introduction_stochastic}
	\partial_t \chi  - m \abs{v}^{m-1} \Delta_x \chi =  \sum_{k=1}^{\infty} \delta_{u=v} g_k \dot{\beta}_k + \partial_{v}(-\frac{1}{2}G^{2}\delta_{u=v}+n),
\end{equation}
where $G^2:= G^2(x,v)= \sum_{k \ge 1} \abs{ g_k(x,v)}^2$ comes from applying It{\^o}'s formula to $\chi$. 
Two fundamental problems appear in the stochastic case: Firstly, the gain of spatial regularity from the (scaled) heat semigroup in (\ref{kinetic_form_introduction_stochastic}) is only expected to be of first order \cite{da2014stochastic}, reflecting the $L^2$ nature of  stochastic integrals and the irregularity of temporal white noise.  Secondly, the natural $L^2$-based estimates on stochastic integrals appear  incompatible with the $L^{1}$ setting of the deterministic case. We next comment on both aspects in detail.

The obstacle of gaining only one order of spatial derivative by stochastic convolution has to be addressed by exploiting \textit{a priori} known regularity of the noise coefficients $\delta_{u=v} g_k$ in \eqref{kinetic_form_introduction_stochastic}. 
In the deterministic setting, related obstacles have been overcome using  bootstrapping techniques \cite{lions1994kinetic}.
In order to implement such an argument in the stochastic setting, one could attempt to rewrite the noise coefficient using the following distributional equality\footnote{In the  derivation of this identity, we use the definition of the kinetic function and the condition that $g_k(x,0)=0$ from the Assumption \ref{Assumption_diffusion_coefficient}.}
\begin{equation} \label{ditribution_equality_noise_coefficients_introduction}
	\delta_{u=v} g_k = \chi \partial_v g_k - \partial_v(\chi g_k)  .
\end{equation}
Unfortunately, as $\chi$ is an indicator function, it exhibits much better regularity when measured in $L^1$ than when measured in $L^2$, as required for stochastic integration. 
For instance, if we knew a priori that $\nabla u \in L^{2}_{\omega,t,x}$, then bootstraping  \eqref{ditribution_equality_noise_coefficients_introduction} would lead to expressions   like $\nabla_x \chi = \delta_{u=v} \nabla u \in L^2_{\omega,t,x} L^1_v$. These are $L^1_v$-based expressions (due to $\delta_{u=v}$) and thus incompatible with the $L^2$-type nature of the stochastic integral. Indeed, $\nabla_x \chi = \delta_{u=v} \nabla u \notin L^2_{\omega,t,x} L^2_v$. 
 This point is laid out in more detail in Appendix \ref{averaging_techniques_and_the_stochastic_integral_appendix} below.

The first key idea of the present work is to resolve this issue by avoiding bootstrapping arguments 
and instead directly exploiting   \textit{a priori} dissipation estimates of the type
\begin{equation} \label{eq:apriori_new}
	\nabla (\abs{u}^{\frac{m+\eta}{2}-1} u)\in L^{2}(\Omega \times (0,T) \times \mathbb{T}^d), \qquad \text{for} \  \eta > 0.
\end{equation}
This is an $L^{2}$-type estimate and, therefore, suitable for the stochastic forcing term.  In the linear Dawson--Watanabe case ($m=1$), our argument corresponds to exploiting bounds on $\nabla u^{\frac{1+\eta}{2}}\in L^{2}_{\omega,t,x}$ for non-negative $u$. It seems that this is a new idea even in the linear context.

The proof of averaging estimates and optimal regularity in the deterministic setting relies on the use of velocity parametrized multiplier estimates \cite{gess2017sobolev,gess2019optimal}. These techniques are incompatible with using the a \textit{priori} estimate \eqref{eq:apriori_new}; see Appendix \ref{averaging_techniques_and_the_stochastic_integral_appendix} below for a more detailed discussion. This is in contrast to the linear case of the stochastic heat equation ($m=1$), where the mild formulation can be employed
\begin{equation*}
	u(t,x)=e^{- t \Delta }u_{0}+ \sum_{k=1}^{\infty} \int_{0}^{t}e^{- (t-s) \Delta} g_k(x,u(s,x)) d\beta_k(s),
\end{equation*}
but in the quasilinear case \eqref{degenerate_parabolic_hyperbolic_SPDE} no mild formulation is known. The second main idea of the present work is the identification of a (partially) mild formulation for quasilinear stochastic PDEs through the kinetic formulation.  More precisely, we use the kinetic form \eqref{kinetic_form_introduction_stochastic}  to obtain the following mixed mild-Fourier representation
\begin{equation} \label{eq:mixed_representation_new}
	\begin{split}
		u(t,x)  & =\int_{v}\chi(t,x,v) \ dv \\  & = \int_v e^{- m \abs{v}^{m-1} t \Delta} \chi(0,x,v) \ dv  \\ &+\sum_{k=1}^{\infty}\int_{0}^{t}(e^{-m|v|^{m-1}(t-s)\Delta}g_{k}(x,u(s,x)))\Big|{}_{v=u(s,x)} \ d\beta_{k}(s) \\
		& +\int_{v}   \frac{1}{ \mathcal{L}(\partial_t, \nabla_x,v)}   \partial_{v}(-\frac{1}{2}G^{2}\delta_{v=u}+n) \ dv, 
	\end{split}
\end{equation}
where $ \mathcal{L}(\partial_t, \nabla_x,v)$ is the differential operator with symbol $	\mathcal{L}(i \tau, \xi, v ) := (2\pi)( i \tau + m \abs{v}^{m-1} (2\pi) \abs{\xi}^2)$. We emphasize that the velocity integration in the stochastic integral in  \eqref{eq:mixed_representation_new} has already been performed  to avoid the use of velocity parametrized estimates.
In other words, this representation allows to write the stochastic forcing in terms of a heat kernel with diffusivity frozen at the value of the solution $u$ itself. We believe this to be an interesting and useful observation, since it allows to ``freeze the coefficients'' of the quasilinear drift without relying on any a \textit{priori} continuity of the solution, which would lead to dimensional restrictions, nor on maximal regularity arguments. In a sense, this representation allows to consider a concept of ``mild solutions'' in the setting of quasilinear stochastic PDEs. This is the second main idea of the present work.

Starting from the mixed representation \eqref{eq:mixed_representation_new}, we can then combine Fourier analytic methods to treat the deterministic parts of \eqref{eq:mixed_representation_new} with the real-analytic heat kernel estimates of the stochastic part in \eqref{eq:mixed_representation_new}. The combination of these techniques into consistent and optimal estimates is technically delicate and constitutes the third main contribution of this work.

\begin{oss}
Assumption \ref{Assumption_diffusion_coefficient} implies in particular that $g(x,u)$ goes to zero as $u$ goes to zero. This is necessary to obtain the required control on singular moments of the kinetic measure.
Indeed, at least formally, the estimate on these singular moments is obtained by applying
 It{\^o}'s formula 
to the function
\begin{equation*}
	u \mapsto \int_{\mathbb{T}^d} |u|^{1+ \varepsilon} \ dx,  \quad \quad \text{for} \  \varepsilon \in (0,1).
\end{equation*}
With an integration by parts one obtains
\begin{equation} \label{Ito_formula_additive_case}
	\begin{split}
	&	\int_0^T \int_{\mathbb{T}^d} |u|^{1+ \varepsilon} \ dx \ dt \\ & + \frac{4m \varepsilon( 1+ \varepsilon)  }{(m+ \varepsilon)^2} \int_0^T \int_{\mathbb{T}^d} |\nabla (|u|^{\frac{m+ \varepsilon}{2}-1}u)|^2 \ dx \ dt
		\\ & = \int_{\mathbb{T}^d} |u_0|^{1+ \varepsilon} \ dx 
		\\ & +  ( 1+ \varepsilon) \sum_{k=1}^{\infty} \int_0^T  \int_{\mathbb{T}^d}  |u|^{ \varepsilon} \sgn(u)  g_k(x,u) \ dx \  d\beta_k(t)
		\\ & + \frac{\varepsilon( 1+ \varepsilon)}{2} \int_0^T \int_{\mathbb{T}^d}  |u|^{ \varepsilon -1} \sum_{k=1}^{\infty} g^2_k(x,u) \ dx  \ dt.
	\end{split} 
\end{equation}
The singular moments of the kinetic measure correspond to the second term on the left-hand side of \eqref{Ito_formula_additive_case}, i.e. $|\nabla (|u|^{\frac{m+ \varepsilon}{2}-1}u)|^2$ for $\varepsilon$ small. The decay of $g_k$ as the solution $u$ goes to zero in Assumption \ref{Assumption_diffusion_coefficient} is required to control the It{\^o}'s correction term $\int_0^T \int_{\mathbb{T}^d}  |u|^{ \varepsilon -1} \sum_{k=1}^{\infty} g^2_k(x,u) \ dx  \ dt$ on the right-hand side of \eqref{Ito_formula_additive_case}.
This problem does not appear in the deterministic case \cite{gess2019optimal}, where the forcing term is a deterministic $L^1$ (time and space) function. 
\end{oss}

	\subsection{Structure of the paper}
	 In Section  \ref{Preliminaries} we introduce the spaces employed in this work and the notion of kinetic solution.  Two averaging lemmata (Lemma \ref{Isotropic_Averaging_Lemma} and Lemma \ref{Time_Averaging_Lemma}) are derived in Section \ref{Isotropic_Averaging_Lemmas_section}. The two main regularity results (Theorem \ref{Main_theorem} and Theorem \ref{Main_theorem_time}) are then obtained by applying the two lemmata to \eqref{degenerate_parabolic_hyperbolic_SPDE} in Section \ref{Application}.

	\section{Preliminaries} \label{Preliminaries}
	
	\subsection{Notation and Spaces} \label{Notation}
	We fix a filtered probability space $ \left(\Omega, \mathcal{F},  \left(\mathcal{F}_t \right)_{t \in [0,T]} ,  \mathbb{P} \right)$, carrying an infinite sequence of independent $\left(\mathcal{F}_t \right)_{t \in [0,T]}$-Wiener processes $(\beta_k(t))_{k \in \mathbb{N}, t \in [0,T]}$. Let $\mathcal{P}$ be the predictable $\sigma$-algebra on $\Omega \times [0,T]$ associated to $\left(\mathcal{F}_t \right)_{t \in [0,T]}$.
	
 The short-hand $ \int_{t,x,v}$ is used for the integration over $  \mathbb{R}_t \times \mathbb{T}^d_x \times \mathbb{R}_v$, unless stated otherwise. The Fourier transform in time and space $ \mathcal{F}_{t,x}$ over $ \mathbb{R}_t \times \mathbb{T}^d_x $ for $f(t,x)$ is defined as follows
	\begin{equation*}
		\hat{f}(\tau,\xi) =  \mathcal{F}_{t,x} f(\tau,\xi) = \int_{\mathbb{R}} \int_{\mathbb{T}^d}  f(t,x) e^{-2 \pi i \xi \cdot x}  e^{-2 \pi i \tau  t}  \ dx \ dt, \quad \xi \in  \mathbb{Z}^d, \ \tau \in   \mathbb{R}.
	\end{equation*}
The Fourier transform in the space variable $x$ over $ \mathbb{T}^d $ (resp. time variable $t$ over $\mathbb{R}$) is denoted by $ \mathcal{F}_{x}$ (resp. $ \mathcal{F}_{t}$).

	 For $u \in \mathbb{R}$ we set $u^{[m]} := \abs{u}^{m-1}u$.
	We denote by $ \mathcal{M}_{TV}$ the space of signed measures with finite total variation. 
 For a Banach space $X$ and $ q \in [1, \infty]$, we endow the Bochner--Lebesgue space $L^q(\Omega; X)$ with the usual norm 
	\begin{equation*}
	\| f \|_{L^q(\Omega ; X)} :=	\left(\mathbb{E} \| f(\omega)\|_X^q \right)^{\frac{1}{q}},
	\end{equation*}
with the usual interpretation as essential supremum for $q= \infty$.
	We derive regularity results in an inhomogeneous Chemin--Lerner space and a vector-valued Besov space that are suitable to treat the space and time regularity with Fourier analytic methods. We briefly recall their definitions from \cite{schmeisser1987topics,gess2019optimal,amann1997operator}. 
	Let $ \mu : \mathbb{R} \rightarrow \mathbb{R}$ be a  function whose Fourier transform is smooth and compactly supported in the annulus $ \left \{ \tau \in \mathbb{R}: 2^{-1} \le \abs{\tau} \le 2 \right \}$ with 
	\begin{equation*}
		\sum_{l \in \mathbb{Z}} \hat{\mu}_l (\tau):= \sum_{l \in \mathbb{Z}} \hat{\mu} (2^{-l}\tau)=1, \quad \forall \tau \in \mathbb{R}  \fgebackslash \{ 0 \},
	\end{equation*}
	and let $ \zeta: \mathbb{R}^d \rightarrow \mathbb{R}$ be a function whose Fourier transform is smooth and compactly supported in the annulus $ \left \{ \xi \in \mathbb{R}^d: 2^{-1} \le \abs{\xi} \le 2 \right \}$ such that
	\begin{equation*}
		\sum_{j \in \mathbb{Z} } \hat{\zeta}_j (\xi):= \sum_{j \in \mathbb{Z}} \hat{\zeta} (2^{-j}\xi)=1, \quad \forall \xi \in \mathbb{R}^d \fgebackslash \{0 \}.
	\end{equation*}
	Let $ \hat{\eta}_l := \hat{\mu}_l $ for $ l \ge 1$ and $ \hat{\eta}_0 := 1 - \sum_{l \ge 1} \hat{\eta}_l $. Similarly, let $ \hat{\varphi}_j := \hat{\zeta}_j $ for $ j \ge 1$ and $ \hat{\varphi}_0 := 1 - \sum_{j \ge 1} \hat{\varphi}_j $. We use the notation $ \abs{\tau} \sim 2^l$ for $ 2^l \lesssim  \abs{\tau} \lesssim 2^l$.
	\begin{defn}
		Let $ \mathcal{S}^{'}$ be the space of tempered distributions on $ \mathbb{R}_t \times \mathbb{T}^d_x$ and $  - \infty < \kappa_t , \kappa_x < \infty $ and $ 1 \le p,q \le \infty $.
	\begin{enumerate}
	\item		 The inhomogeneous Chemin--Lerner space $ \tilde{L}^{p}_t B^{\kappa_x}_{p, \infty} $ is defined by
			\begin{equation*}
				\tilde{L}^{p}_t B^{\kappa_x}_{p, \infty} := \tilde{L}^{p}_t B^{\kappa_x}_{p, \infty}(\mathbb{R}_t \times \mathbb{T}^d_x ) := \left \{ f \in \mathcal{S}^{'} \ | \  \| f \|_{\tilde{L}^{p}_t B^{\kappa_x}_{p, \infty}} <  \infty \right \},
			\end{equation*}
			with the norm 
			\begin{equation*}
				\| f \|_{\tilde{L}^{p}_t B^{\kappa_x}_{p, \infty}} := \sup_{j \ge 0} 2^{\kappa_x j} \| \mathcal{F}_x^{-1} \hat{\varphi}_j \mathcal{F}_x f \|_{L^p (\mathbb{R}_t \times \mathbb{T}^d_x)}.
			\end{equation*}
				where
		\begin{equation*}
			\mathcal{F}^{-1}_{x} \hat{f}  (t,x) :=   \sum_{\xi \in   \mathbb{Z}^d}	\hat{f}(t,\xi) e^{2 \pi i \xi \cdot x} .  
		\end{equation*}
		\item		 The vector-valued Besov space $ B^{\kappa_t}_{q, \infty} (\mathbb{R}_t; L^p(\mathbb{T}^d_x))$ is defined by
	\begin{equation*}
	B^{\kappa_t}_{q, \infty} (L^p_x) :=	B^{\kappa_t}_{q, \infty} (\mathbb{R}_t; L^p(\mathbb{T}^d_x))  := \left \{ f \in \mathcal{S}^{'} \ | \  \| f \|_{B^{\kappa_t}_{q, \infty} (L^p_x)} <  \infty \right \},
	\end{equation*}
	with the norm 
	\begin{equation*}
		\| f \|_{B^{\kappa_t}_{q, \infty} (L^p_x) } := \sup_{l \ge 0} 2^{\kappa_t l} \| \mathcal{F}_t^{-1} \hat{\eta}_l \mathcal{F}_t f \|_{  L^q( \mathbb{R}_t; L^p( \mathbb{T}^d_x) )}.
	\end{equation*}
	where
	\begin{equation*}
		\mathcal{F}^{-1}_{t} \hat{f} (t,x) :=  \int_{\mathbb{R}} 	\hat{f}(\tau,x)   e^{2 \pi i \tau  t} dt .  
	\end{equation*}
			\end{enumerate}
	\end{defn}
\begin{oss}
The spaces $ \tilde{L}^{p}_t B^{\kappa_x}_{p, \infty} $ and $ B^{\kappa_t}_{q, \infty} (L^p_x) $ are   Banach spaces  \cite{triebel1977generalIII,amann1997operator}.
\end{oss}
	The following embedding results are used to derive the two main results (Theorem \ref{Main_theorem} and Theorem \ref{Main_theorem_time}) from the regularity estimates obtained in the averaging lemmata (Lemma \ref{Isotropic_Averaging_Lemma} and Lemma \ref{Time_Averaging_Lemma}) in Section \ref{Isotropic_Averaging_Lemmas_section}.
	\begin{lem} \label{Embedding_time_spaces} 
		Let $ \kappa_x, \kappa_t > 0$ and $ p,q \in [1, \infty]$. Then 
		\begin{equation*}
			\begin{split}
		&	\tilde{L}_t^p B^{\kappa_x}_{p, \infty} (\mathbb{R}_t \times \mathbb{T}^d_x ) \subset L^{p}(\mathbb{R}_t; W^{\sigma_x, p}(\mathbb{T}^d_x)),
		\\ & B^{\kappa_t}_{q, \infty} (\mathbb{R}_t ;  L^p (\mathbb{T}^d_x )) \subset W^{\sigma_t, q}(\mathbb{R}_t; L^{p}(\mathbb{T}^d_x)),
			\end{split}
		\end{equation*}
for all	$ \sigma_x \in [0, \kappa_x )$ and $ \sigma_t \in [0, \kappa_t )$.
	\end{lem}
\begin{proof}
The first embedding follows from \cite[page 98]{bahouri2011fourier}.  The second embedding is derived using \cite[Theorem 2.2.2]{amann2019linear} and \cite[Section 5]{amann1997operator}. 
\end{proof}


	\subsection{Kinetic solution} \label{Section_kinetic_measure_solution}
	Throughout this paper we work with the following definition of kinetic measure and kinetic solution.
		\begin{defn}[Kinetic measure]
		A mapping $n: \Omega \rightarrow \mathcal{M}_{+}( [0,T] \times \mathbb{T}^d \times \mathbb{R})$, the set of nonnegative Radon measures over $ [0,T] \times \mathbb{T}^d \times \mathbb{R}$, is said to be a kinetic measure provided that for all $ \varrho \in C_c([0,T) \times \mathbb{T}^d \times \mathbb{R})$, the process 
		\begin{equation*}
		 \int_0^t \int_{\mathbb{T}^d} \int_{\mathbb{R}} \varrho(s,x,v) \ dn(s,x,v)
		\end{equation*}
		 is $\mathcal{F}_t$-predictable.
	\end{defn}
	
	\begin{defn}[Kinetic solution] \label{kinetic_solution}
		A function $u \in L^1(\Omega \times [0,T], \mathcal{P}, d\mathbb{P} \otimes dt; L^1(\mathbb{T}^d ))$ is called a kinetic solution to \eqref{degenerate_parabolic_hyperbolic_SPDE} with initial datum $u_0$ if the following conditions are satisfied:
		\begin{itemize}
		\item $\nabla u^{[\frac{m+1}{2}]}  \in L^2(\Omega \times [0,T] \times \mathbb{T}^d)$.
		\item Let $n_1 : \Omega \rightarrow \mathcal{M}_{+}([0,T] \times \mathbb{T}^d \times \mathbb{R})$ be defined as follows: for all $ \phi \in C^{\infty}_c([0,T] \times \mathbb{T}^d \times \mathbb{R}) $,  \begin{equation} \label{definition_n1_parabolic_dissipation_measure}
			n_1 (\phi) = \int_0^T \int_{\mathbb{T}^d} \int_{\mathbb{R}} \phi(t,x,v) \delta_{u(t,x)=v} \frac{4m}{(m+1)^2} \abs{\nabla u^{\left[\frac{m+1}{2} \right]}(t,x)}^2 \ dv \ dx \ dt.
		\end{equation} 
There exists a kinetic measure $n$ such that, for all $\phi \in C^{\infty}_c([0,T] \times \mathbb{T}^d \times \mathbb{R})$, $\phi \ge 0$, it holds  $n(\phi) \ge n_1(\phi)$, $\mathbb{P}$-a.s. Let $ \chi =\chi(\omega,t,x,v) := \mathbbm{1}_{v < u(\omega,t,x)} -\mathbbm{1}_{v < 0}$. Then the pair $ ( \chi, n)$ satisfies,  for all $ \varphi \in C^{\infty}_c([0,T) \times \mathbb{T}^d \times \mathbb{R})$,  $ \mathbb{P}$-a.s., the following
	\begin{equation} \label{kinetic_formulation_distribution_PME}
				\begin{split}
				& \int_0^T \int_{\mathbb{T}^d} \int_{\mathbb{R}} \chi(t,x,v) \partial_t \varphi(t,x,v)    \ dv \ dx \ dt    \\ & + 	 \int_{\mathbb{T}^d} \int_{\mathbb{R}}  \chi(0,x,v) \varphi(0,x,v)  \ dv \ dx
				 \\   & + m \int_0^T \int_{\mathbb{T}^d} \int_{\mathbb{R}}  \chi(t,x,v)  \abs{v}^{m-1} \Delta \varphi(t,x,v)  \ dv  \ dx \ dt 
		   \\ &  = - \sum_{k = 1}^{\infty} \int_0^T \int_{\mathbb{T}^d} g_k(x,u(t,x)) \varphi(t,x,u(t,x)) \ dx \  d\beta_k(t) 		
		\\ & +  \int_0^T \int_{\mathbb{T}^d} \int_{\mathbb{R}}  \partial_v \varphi(t,x,v) \ dn(t,x,v)    \\  &  - \frac{1}{2} \int_0^T \int_{\mathbb{T}^d} G^2(x,u(t,x)) \partial_v \varphi(t,x,u(t,x)) \ dx \ dt. 
			\end{split}
			\end{equation}
		\end{itemize}
	\end{defn}
The definition of a kinetic solution given in Definition \ref{kinetic_solution} includes only the conditions required to prove the regularity results in Theorem \ref{Main_theorem} and Theorem \ref{Main_theorem_time}. Additional assumptions on the kinetic measure are needed in order to prove uniqueness of solutions \cite{gess2021fehrman,gess2018well}.
	Under these additional assumptions, the well-posedness of non-negative kinetic solutions of \eqref{degenerate_parabolic_hyperbolic_SPDE} for the case $\alpha =\frac{1}{2}$ in Assumption \ref{Assumption_diffusion_coefficient} has been proven in \cite{gess2021fehrman} for locally $\frac{1}{2}$-H{\"o}lder continuous $ g_k$. The existence and uniqueness for signed kinetic solutions of \eqref{degenerate_parabolic_hyperbolic_SPDE} with the same type of  $ g_k$ is still an open problem. For the case $\alpha > \frac{1}{2}$, 
	 the well-posedness for signed entropy/kinetic solutions of \eqref{degenerate_parabolic_hyperbolic_SPDE} with $ (\frac{1}{2}+ \delta)$-H{\"o}lder continuous $g_k$ for some $\delta \in (0, \frac{1}{2}]$ has been proven in \cite{dareiotis2019entropy}. 

	\section{Averaging Lemmata} \label{Isotropic_Averaging_Lemmas_section}
     In this section, we derive two averaging lemmata that we will use in Section \ref{Application} to prove our main results (Theorem \ref{Main_theorem} and Theorem \ref{Main_theorem_time}). Here, we work with the following notion of solution.
     \begin{defn}[Quasi-solution] \label{distributional_solution_Averaging_lemmata}
  	Let $\tilde{u}$ be a $\mathcal{F}_t$-predictable function defined on $ \Omega \times [0,T] \times \mathbb{T}_x^d $ and let $\mathring{g}_k$ be a function defined on $ \mathbb{R}_t \times \mathbb{T}_x^d \times \mathbb{R}_v $ compactly supported in $(0,T)$ such that, for all $ \varphi \in C^{\infty}_c( \mathbb{R}_t \times \mathbb{T}_x^d \times \mathbb{R}_v)$, the following holds
     	\begin{equation*}
     		\mathbb{P} \left[ \int_0^T \sum_{k=1}^{\infty} \left( \int_{\mathbb{T}^d } \mathring{g}_k(s,x,\tilde{u}) \varphi(s,x,\tilde{u})  \ dx  \right)^2 ds < \infty \right] = 1.
     	\end{equation*}
      Let $h, \tilde{h}$ be Radon measures on $ \Omega \times \mathbb{R}_t \times \mathbb{T}_x^d \times \mathbb{R}_v$. 
 Let $ \chi $ be a $\mathcal{F}_t$-predictable function defined on $ \Omega \times \mathbb{R}_t \times \mathbb{T}_x^d \times \mathbb{R}_v $ compactly supported in $ (0,T)$, a.s. locally integrable and solution, for all $ \varphi \in C^{\infty}_c( \mathbb{R}_t \times \mathbb{T}_x^d \times \mathbb{R}_v)$,  $ \mathbb{P}$-a.s., to the following 
     	\begin{equation} \label{notion_solution_kinetic_torus}
     		\begin{split}
&	\int_{\mathbb{R}_t} \int_{\mathbb{T}_x^d} \int_{\mathbb{R}_v}  \chi(t,x,v) \partial_t \varphi(t,x,v)  \ dv \ dx \ dt    \\ & + m	\int_{\mathbb{R}_t} \int_{\mathbb{T}_x^d} \int_{\mathbb{R}_v}  \chi(t,x,v)  \abs{v}^{m-1} \Delta \varphi(t,x,v)  \ dv  \ dx \ dt    \\  &  = - \sum_{k = 1}^{\infty} \int_{\mathbb{R}_t} \int_{\mathbb{T}_x^d} \mathring{g}_k(t,x, \tilde{u}(t,x)) \varphi(t,x,\tilde{u}(t,x))  \ dx \ d\beta_k(t) 
\\  & - 	\int_{\mathbb{R}_t} \int_{\mathbb{T}_x^d} \int_{\mathbb{R}_v}  \varphi(t,x,v) \ dh(t,x,v)  
	\\  & + 	\int_{\mathbb{R}_t} \int_{\mathbb{T}_x^d} \int_{\mathbb{R}_v}  \partial_v \varphi(t,x,v) \ d \tilde{h}(t,x,v)  
     		\end{split} 
     	\end{equation}
     \end{defn}
\begin{oss}
We introduce two cut-off functions in the time and velocity variable in the kinetic formulation \eqref{kinetic_formulation_distribution_PME} to prove Theorem \ref{Main_theorem} and Theorem \ref{Main_theorem_time} in Section \ref{Application}.	The Definition  \ref{distributional_solution_Averaging_lemmata} is a modification of Definition \ref{kinetic_solution} to take into account the presence of these cut-off functions. We use the Definition \ref{distributional_solution_Averaging_lemmata} in Lemma \ref{Isotropic_Averaging_Lemma} and Lemma \ref{Time_Averaging_Lemma} below.
	
	
\end{oss}
The following averaging lemma is used to prove Theorem \ref{Main_theorem}.
	\begin{lem} \label{Isotropic_Averaging_Lemma}
	Assume $m \in (1,\infty) $, $ \beta > 1$ and $ \rho= 1-  \frac{1}{\beta}$. Let $ \nu \in ( \frac{1}{2}, \frac{m}{2})$, $\epsilon \in (0,1)$, $ \gamma \in (0,\min(1+\epsilon,m))$, $ \varpi >1 $ and $ \varkappa^{-1} =1 - \varpi^{-1} $ such that 
	\begin{equation} \label{definition_gamma_star_statement}
\gamma^{\star}  :=   \gamma(1- \epsilon \varpi) + \epsilon \varpi(m+ \frac{1}{2} - \nu) < \frac{m+1}{2}.
	\end{equation}
 Let $ \chi \in L^{\beta}_{\omega,t,x,v} $ be a solution in the sense of Definition \ref{distributional_solution_Averaging_lemmata}. 
	Suppose that $  \mathring{g}_k(\tilde{u}) \nabla  \tilde{u}^{\frac{m+1}{2}- \gamma^{\star} }$, $ \tilde{u}^{\frac{m+1}{2}- \gamma^{\star} } \nabla \mathring{g}_k(\tilde{u})     \in L^2_{\omega,t,x}$, $ \mathring{g}_k(\tilde{u})  \in L^{\infty}_t L^{2}_{\omega,x}$, $ \tilde{u}^{\varkappa(2 \nu- 1)} \mathring{g}^2_k(\tilde{u})  \in L^1_{\omega,t,x}$ and  $h, \tilde{h}$ satisfy the condition
	\begin{equation} \label{condition_singular_moments_statement}
		\abs{h}(\omega,t,x,v) \abs{v}^{1- \gamma}  + | \tilde{h} | (\omega,t,x,v) \abs{v}^{- \gamma} \in L^1_{\omega} \Ma_{\text{TV}}( \mathbb{R}_t \times \mathbb{T}_x^d \times \mathbb{R}_v).
	\end{equation}
	If $\bar{\chi} := \int_v \chi \ dv \in L^1_{\omega,t,x}$, then $ \bar{\chi}  \in L^{p}_{\omega} \tilde{L}^{p}_t B_{p, \infty}^{\kappa_x}   $ where 
	\begin{equation}  \label{reg_time_results_before_bootstrap}
		\begin{split} 
			\kappa_x &  :=  \frac{2 \rho }{m - \gamma  + \rho} - \epsilon,
			\\    p &  :=   \frac{m- \gamma  + \rho}{\rho +(1-\rho)(m- \gamma )} - \epsilon, 
		\end{split}
	\end{equation}
	 with estimate  
	\begin{equation} \label{Bounds_Averaging_lemma_statement}
		\begin{split}
			\| \bar{\chi}   \|_{ L^{p}_{\omega} \tilde{L}^{p}_t B_{p, \infty}^{\kappa_x}}  & \lesssim   \| \chi   \|_{ L^{\beta}_{\omega,t,x,v}} +   \left(  \sum_{k=1}^{\infty}     \|  \mathring{g}_k   \nabla  \tilde{u}^{\frac{m+1}{2}- \gamma^{\star} }   \|_{L^2_{\omega,t,x}}^2 \right)^{\frac{1}{2}} 
	\\	& + \left(  \sum_{k=1}^{\infty}     \| \tilde{u}^{\frac{m+1}{2}- \gamma^{\star} }  \nabla \mathring{g}_k      \|_{L^2_{\omega,t,x}}^2 \right)^{\frac{1}{2}}     \\ & + \left( \sum_{k=1}^{\infty}  \esssup_{t \in [0,T]}    \left \|    \mathring{g}_k^2(t,\cdot,\tilde{u})  \right \|_{L^1_{\omega,x}}^{\frac{1}{\varpi}}   \left \|    \tilde{u}^{\varkappa(2 \nu - 1)}  \mathring{g}_k^2   \right \|_{L^1_{\omega,t,x}}^{\frac{1}{\varkappa}} \right)^{\frac{1}{2}} 	
	     \\ & + \| \abs{v}^{1- \gamma}  h  \|_{L^1_{\omega} \mathcal{M}_{TV}}  +  \| \abs{v}^{- \gamma } \tilde{h}   \|_{L^1_{\omega} \mathcal{M}_{\text{TV}}} +  \| \bar{\chi}   \|_{ L^{1}_{\omega,t,x}}.
		\end{split}
\end{equation}
	Here and in the proof  $ \ \lesssim $ denotes a bound that holds up to a multiplicative constant that only depends on $ m, \rho, \nu, \epsilon, \gamma, \varpi $ and $T$.
\end{lem}
	\begin{oss}
		Note that if $ \rho, \gamma$ are chosen close to one and $\epsilon$ chosen small enough, the order of the differentiability and the integrability exponents in \eqref{reg_time_results_before_bootstrap} correspond to $\kappa_x$ close to $\frac{2}{m}$ and the stochastic, time and space integrability equal to $m$ as in Theorem \ref{Main_theorem}.
	\end{oss}
From now on, we omit the $ \omega$-dependence in $\chi$ and $ \tilde{u}$ for notational simplicity.
	\begin{proof} 
We first assume that $ \chi$ is compactly supported with respect to $v$ and then remove this qualitative condition at the end of the proof.

Let $ \mathcal{L}(\partial_t, \nabla_x,v)$ be the differential operator with symbol $	\mathcal{L}(i \tau, \xi, v ) := (2\pi)(i \tau + m \abs{v}^{m-1}(2\pi) \abs{\xi}^2)$.  We decompose $\chi  $ into Littlewood--Paley blocks with respect to the space variable. Let $ \{ \hat{\varphi}_j \}_{j \ge 0}$ be defined as in Section \ref{Notation}. For $j \ge 0$, we define the Littlewood--Paley block of $\chi $ as follows:
		\begin{equation*}
			\chi_{j}= \mathcal{F}_{x}^{-1} [  \hat{\varphi}_j \mathcal{F}_{x}  \chi  ].
		\end{equation*}
		By definition $ \mathcal{F}_{x} \chi_{j} (t,\xi,v)$ is supported on space frequency $ \abs{\xi} \sim 2^j$ for $j > 0$ and on  $ \abs{\xi} \lesssim 1$ for $j=0$.
	In the calculations below, we assume that  $j>0$ unless stated otherwise.
		
Then $\chi_j$ solves, in the sense of distributions,
		\begin{equation} \label{definition_solution_LP_space_distribution}
				 \mathcal{L}(\partial_t, \nabla_x,v) \chi_j   =    e_{j}  +  h_{j}    + \partial_v \tilde{h}_{j},
		\end{equation}
		where the block for the stochastic integral is defined as follows 
		\begin{equation*} 
			e_{j}=  \mathcal{F}_{x}^{-1} \left[  \hat{\varphi}_j \mathcal{F}_{x} \left( \sum_{k=1}^{\infty}  \delta_{\tilde{u}=v}   \mathring{g}_k  \dot{\beta}_k \right) \right],
		\end{equation*}
		and the blocks for $h  $ and $\tilde{h} $ are defined as follows
		\begin{equation} \label{Littlewood_Paley_blocks_for_large_velocities}
			h_{j}= \mathcal{F}_{x}^{-1} [  \hat{\varphi}_j \mathcal{F}_{x}  h], \quad \tilde{h}_{j} = \mathcal{F}_{x}^{-1} [  \hat{\varphi}_j \mathcal{F}_{x} \tilde{h}   ], 
		\end{equation}
		respectively.
		
Let $\psi_0$ be a smooth function with compact support in the ball $B_2(0)$ such that $\psi_0=1$ in  $B_1(0)$ and $\psi_1:=1- \psi_0$. We consider a microlocal decomposition of $\chi_{j}$ with regard to the degeneracy of the operator $ \mathcal{L}(\partial_t, \nabla_x,v)$ on the Fourier block $\abs{\xi}^2 \sim 2^{2j}$. For $ \delta >0$ to be specified later, we write
\begin{equation} \label{decomposition_LP_space_space_averaging}
	\chi_{j}=  \psi_0 \left(\frac{2^{2j} \abs{v} }{\delta} \right)  \chi_{j} + \psi_1 \left(\frac{2^{2j} \abs{v}}{\delta} \right)  \chi_{j} =: \chi^{0}_{j}+  \chi^{1}_{j}.
\end{equation}
From \eqref{definition_solution_LP_space_distribution}, we have 
\begin{equation} \label{decomposition_LP_blocks_space_beginning}
	\begin{split}
		\chi^{1}_{j}  & = \mathcal{F}_{t,x}^{-1} \psi_1 \left(\frac{2^{2j} \abs{v} }{\delta} \right) \frac{1}{ \mathcal{L}(i\tau, \xi,v)} \mathcal{F}_{t,x} e_{j}  + \mathcal{F}_{t,x}^{-1} \psi_1 \left(\frac{2^{2j} \abs{v} }{\delta} \right) \frac{1}{ \mathcal{L}(i\tau, \xi,v)} \mathcal{F}_{t,x}  h_{
			j}   \\  & + \mathcal{F}_{t,x}^{-1} \psi_1 \left(\frac{2^{2j} \abs{v} }{\delta} \right) \frac{1}{ \mathcal{L}(i\tau, \xi,v)} \mathcal{F}_{t,x}  \partial_v  \tilde{h}_{j}  \\  & =: \chi^{1,1}_{j}  + \chi^{1,2}_{j} +  \chi^{1,3}_{j}  .
	\end{split}
\end{equation}
We have obtained the following decomposition
\begin{equation*}
\begin{split}
\bar{\chi}_{j} = \int_v \chi_{j}  \ dv  = \int_{v} \chi_{j}^0  \ dv +  \int_{v} \chi_{j}^{1}   \ dv  = \int_{v} \chi_{j}^0   \ dv   +  \int_{v} \chi_{j}^{1,1} dv +  \int_{v} \chi_{j}^{1,2} dv + \int_{v} \chi_{j}^{1,3} dv .
\end{split}
\end{equation*}
		We proceed in dividing the proof into five steps. In the first four we estimate the velocity averages above and in the final step we interpolate all the estimates. The main novelty relative to \cite[Proof of Lemma 4.2]{gess2019optimal} is the estimate of $\int_{v} \chi_{j}^{1,1} dv$  in Step 2 below which is based on the integral representation of the stochastic term and the control on \eqref{eq:apriori_new} as outlined in the Introduction. Steps 1, 3 and 4 below are identical to \cite[Proof of Lemma 4.2 Steps 1, 2 and 3]{gess2019optimal}, respectively with the difference that here the estimates are also $\omega$-dependent. 
		By Definition \ref{distributional_solution_Averaging_lemmata} and \eqref{decomposition_LP_space_space_averaging}, $\chi^{0}_{j}$ and $ \chi^{1}_{j}$ are compactly supported in $(0,T)$. Therefore, we have
		\begin{equation} \label{Localization_in_time_degenerate_spatial_regularity}
 \left \| \int_v \chi_j^0 \ dv \right \|_{L^{\beta}(\Omega \times \mathbb{R}_t \times \mathbb{T}^d)}  = 	\left \| \int_v \chi_j^0 \ dv  \right \|_{L^{\beta}(\Omega \times [0,T] \times \mathbb{T}^d)},
		\end{equation}
even though each $  \chi^{1,i}_{j} $ for $i=1,2,3$ does not need to be compactly supported in $(0,T)$, we have
		\begin{equation} \label{Localization_in_time_nondegenerate_spatial_regularity}
			\begin{split}
		 \left \| \int_v \chi_j^1 \ dv \right \|_{L^{1}(\Omega \times \mathbb{R}_t \times \mathbb{T}^d)}  & = 	 \left \| \int_v \chi_j^1 \ dv \right \|_{L^{1}(\Omega \times [0,T] \times \mathbb{T}^d)} \\ &  \le \sum_{i=1}^{3}  \left \| \int_v  \chi_j^{1,i} \ dv \right \|_{L^{1}(\Omega \times [0,T] \times \mathbb{T}^d)}.
		 \end{split}
		\end{equation}
For this reason, the norms in Steps 1, 2, 3 and 4 will be evaluated in $t \in [0,T]$.
	 \paragraph{Step 1}  Let $j > 0 $ be arbitrary and fixed. We follow the same arguments as in  \cite [Proof of Lemma 4.2 Step 1]{gess2019optimal}. Since $ \abs{v} \le \delta 2^{-2j}$ on the support of $  \psi_0 \left(\frac{2^{2j} \abs{v}}{\delta} \right)$, we may use Minkowski's and H{\"o}lder's inequality to get
	 
		\begin{equation} \label{velocity_degenerate}
			\begin{split}
		\left\Vert \int_{v} \chi_{j}^0  \ dv \right\Vert_{L^{\beta}_{\omega,t,x}} & = \left\Vert \int_{v}    \psi_0 \left(\frac{2^{2j} \abs{v} }{\delta} \right)  \chi_{j}  \ dv  \right\Vert_{L^{\beta}_{\omega, t,x}} 
		\\ 	& \le   \int_{v}   \abs{ \psi_0} \left(\frac{2^{2j} \abs{v} }{\delta} \right) \| \chi_{j}  \|_{L^{\beta}_{\omega, t,x}} dv  
		\\	&  \lesssim     \| \chi   \|_{L^{\beta}_{\omega,t,x,v}} \left(  \int_{v}   \abs{ \psi_0} \left(\frac{2^{2j} \abs{v} }{\delta} \right)^{\frac{1}{\rho}}  dv  \right)^{\rho}  
		\\ &	\lesssim \frac{\delta^{\rho}}{2^{2j\rho}}    \| \chi   \|_{ L^{\beta}_{\omega,t,x,v}}. 
						\end{split}
\end{equation}
\paragraph{Step 2}  Let $ j > 0 $ be arbitrary and fixed. 
We define 
\begin{equation} \label{definition_av_space_regularity}
	a(v)= m \abs{v}^{m-1},
\end{equation}  
and the periodic heat kernel
\begin{equation*}
	\Phi(t,x) = \begin{cases}\frac{1}{(4 \pi t)^{\frac{d}{2}}} \sum_{n \in \mathbb{Z}^d}  \exp \left(- \frac{\abs{x+n}^2}{4t}\right) & t >0,
		\\ 0  & t < 0. 
	\end{cases}
\end{equation*}
Note that $\chi_j^{1,1}$ in \eqref{decomposition_LP_blocks_space_beginning} is given by

\begin{equation} \label{equivalent_definition_chi11}
	\begin{split}
		\chi_j^{1,1}(t,x,v)  &  = \sum_{k=1}^{\infty}  \int_{s, z} \Phi(a(v)(t-s),x-z)   \psi_1 \left(\frac{2^{2j} \abs{v} }{\delta} \right)  \\ & \times  \int_{y}  \delta_{\tilde{u}(s,y)=v}   \mathring{g}_k(s,y,v)   \varphi_j(z-y)   \ dy  \ d\beta_k(s)  \ dz. 
	\end{split}
\end{equation}
Let 
\begin{equation} \label{Definition_Fk}
 F_{k}(t,x)  :=	 \psi_1 \left(\frac{2^{2j} \abs{\tilde{u}(t,x)} }{\delta} \right)   \mathring{g}_k(t,x,\tilde{u}(t,x)).
\end{equation}
Recall that $\tilde{u}$ is defined on $ [0,T]$. Using \eqref{equivalent_definition_chi11} and \eqref{Definition_Fk}, we have for $\eta>0$ to be determined later 
\begin{equation} \label{Expression_stoch_integral_dec_LP_space}
	\begin{split}
		& \int_{v} \chi_{j}^{1,1}(t,x,v) \ dv  
			\\  & = \sum_{k=1}^{\infty}  \int_{v, s, z} \Phi(a(v)(t-s),x-z)   \psi_1 \left(\frac{2^{2j} \abs{v} }{\delta} \right) 
\\  & \times \int_{y}  \delta_{\tilde{u}(s,y)=v}   \mathring{g}_k(s,y,v)   \varphi_j(z-y)   \ dy  \ d\beta_k(s)  \ dz   \ dv 
	\\ & =  \sum_{k=1}^{\infty}  \int_{z}   \varphi_j(x-z)   \int_{s,y} \Phi(a(\tilde{u}(s,y))(t-s),z-y)    F_{k}(s,y) \  dy \  d\beta_k(s)  \   dz
	\\ & = B(0,t-\eta) + B(t-\eta,t),
	\end{split}
\end{equation}
where 
\begin{equation*}
	\begin{split}
			B(0,t-\eta) &  := \sum_{k=1}^{\infty}  \int_{z}   \varphi_j(x-z)    \int_{0}^{t-\eta} \int_{y} \Phi(a(\tilde{u}(s,y))(t-s),z-y)  F_{k}(s,y)   \  dy \ d\beta_k(s)   \ dz, \\
	 B(t-\eta,t) &  := \sum_{k=1}^{\infty}  \int_{z}   \varphi_j(x-z) \int_{t-\eta}^{t}  \int_{y} \Phi(a(\tilde{u}(s,y))(t-s),z-y)    F_{k}(s,y)  \  dy  \ d\beta_k(s)  \ dz.
	\end{split}
\end{equation*}
In \eqref{Expression_stoch_integral_dec_LP_space}, we have splitted the $s$-integral into $\int_{s \in [0,t- \eta]}$ and $\int_{s \in [t- \eta,t]}$ to avoid dealing with the divergence in $\int_0^t (t-s)^{-1} ds$ which will appear in the estimates \eqref{estimate_Ia}, \eqref{estimate_Ib}, \eqref{estimate_II_first_term} and \eqref{estimate_II_second_term} below. 

 We provide a bound on the first term on the right-hand side of \eqref{Expression_stoch_integral_dec_LP_space}. Suppose that $\tilde{\psi}_1$ is a smooth function such that $\tilde{\psi}_1=1$ on $(B_{1}(0))^c$ and $\tilde{\psi}_1=0$ on $B_{\frac{1}{2}}(0)$.
We have  using Bernstein's lemma\footnote{The proof on $ \mathbb{R}^d $ can be found e.g. in \cite[Lemma 2.1]{bahouri2011fourier}. The proof on $ \mathbb{T}^d$ follows along the same lines.} and It{\^o}'s isometry 
	\begin{equation} \label{Expression_stoch_integral_dec_LP_space_2}
	\begin{split}
			& \| B(0,t-\eta)  \|_{L^2_{\omega, t,x}}^2
\\	& = \Bigg \| \sum_{k=1}^{\infty}  \int_{z}   \varphi_j(x-z)    \int_{0}^{t-\eta} \int_{y} \Phi(a(\tilde{u}(s,y))(t-s),z-y)    F_{k}(s,y)   \  dy \ d\beta_k(s)   \ dz \Bigg \|_{L^2_{\omega, t,x}}^2 	
\\ 		& \lesssim  2^{-4j} \   \int_{t,x}  \sum_{k=1}^{\infty} \mathbb{E} \Bigg|   \int_{0}^{t- \eta} \Delta_x \int_y \Phi(a(\tilde{u}(s,y))(t-s),x-y) F_{k}(s,y)   \ dy \ d\beta_k(s)  \Bigg|^2  dx \ dt 
	\\ & = 2^{-4j} \  \mathbb{E} \int_{t,x} \sum_{k=1}^{\infty}    \int_{0}^{t- \eta} \Bigg| \Delta_x
				\int_{y} \Phi(a(\tilde{u}(s,y))(t-s),x-y) F_{k}(s,y)   \ dy \Bigg|^2 ds \   dx \ dt	
		 	 \\	 & \lesssim 2^{-4j} \Bigg(  \mathbb{E} \int_{t,x} \sum_{k=1}^{\infty}     \int_{0}^{t- \eta} \abs{ \nabla_x \cdot  \rom{1} }^2 \ ds  \  dx \ dt
	 	+    \mathbb{E} \int_{t,x} \sum_{k=1}^{\infty}    \int_{0}^{t- \eta}   \abs{  \nabla_x \cdot  \rom{2}  }^2  \ ds  \  dx \  dt \Bigg),
					\end{split}
\end{equation}
where 
\begin{equation*}
	\begin{split}
 \rom{1}  & := \int_{y}   \nabla_x \Phi(a(\tilde{u}(s,x-y))(t-s),y)   F_{k}(s,x-y) \  dy,  
\\  \rom{2} &   := \int_{y} \Phi(a(\tilde{u}(s,x-y))(t-s),y)   \tilde{\psi}_1 \left(\frac{2^{2j} \abs{\tilde{u}(s,x-y)} }{\delta} \right)   \nabla_x 	F_{k}(s,x-y)  \   dy.
\end{split}
\end{equation*}
In term $\rom{2}$, we have used that $\tilde{\psi}_1=1$ on the support of $\nabla	F_{k}$.

We denote the first and second derivative with respect to the first argument of the periodic heat kernel $ \Phi$ by $ \Phi_1$ and $ \Phi_{11}$, respectively. Let
\begin{equation*}
	\begin{split}
		\tilde{\Phi} (t,x)   := \Phi_1 (t,x) t, \qquad  \qquad \dbtilde{\Phi} (t,x)   := \Phi_{11}(t,x) t^2  + \tilde{\Phi} (t,x).
	\end{split}
\end{equation*} 
Note that 
\begin{equation} \label{_no_grad_heat_kernel_estimate_tilde_before}
	\begin{split}
		\|  \tilde{\Phi} ( t,\cdot) \|_{L^1_x} \lesssim 1, 
	\end{split}
\end{equation}
and
\begin{equation} \label{heat_kernel_estimate_tilde_before}
	\begin{split}
	\| \nabla_x \Phi( t,\cdot) \|_{L^1_x},		\| \nabla_x \tilde{\Phi} ( t,\cdot) \|_{L^1_x}, \| \nabla_x \dbtilde{\Phi} ( t,\cdot) \|_{L^1_x} \lesssim t^{- \frac{1}{2}}. 
	\end{split}
\end{equation}
  Let
\begin{equation} \label{Definition_Fka}
	\begin{split}
		F_{k,a}(t,x)  & :=  \abs{\tilde{u}(t,x)}^{-(\frac{m+1}{2} - \gamma^{\star} )} \tilde{\psi}_1 \left(\frac{2^{2j} \abs{\tilde{u}(t,x)} }{\delta} \right)   \frac{\nabla_x a(\tilde{u}(t,x))}{a(\tilde{u}(t,x))}  \abs{\tilde{u}(t,x)}^{\frac{m+1}{2} - \gamma^{\star} } \\ & \times  \mathring{g}_k(t,x,\tilde{u}(t,x)).
	\end{split}
\end{equation}
We can use \eqref{Definition_Fka} to treat the term $ \nabla_x \cdot	 \rom{1}$ on the right-hand side of  \eqref{Expression_stoch_integral_dec_LP_space_2} as follows
\begin{equation} \label{divergence_I_space}
	\begin{split}
& \nabla_x \cdot	 \rom{1} \\ &  =  \nabla_x \cdot	 \int_{y}   \nabla_x \Phi(a(\tilde{u}(s,x-y))(t-s),y)  F_{k}(s,x-y)   \ dy
\\ & =  \nabla_x \cdot	 \int_{y}  	\tilde{\Phi} (a(\tilde{u}(s,x-y))(t-s),y)  \ \psi_1 \left(\frac{2^{2j} \abs{\tilde{u}(s,x-y)} }{\delta} \right) 	F_{k,a}(s,x-y) \ dy
\\ & = \nabla_x \cdot	  \int_{y}  \int_{0}^{\tilde{u}(s,y)}	 \partial_v \left(	\tilde{\Phi} (a(v)(t-s),x-y)   \psi_1 \left(\frac{2^{2j} \abs{v} }{\delta} \right) \right) 	  \ dv \  F_{k,a}(s,y) \ dy 
\\    &  =   \int_{v,y} \abs{v}^{-\nu} \mathbbm{1}_{[0,\tilde{u}(s,y)]}(v)     \nabla_x \dbtilde{\Phi} (a(v)(t-s),x-y)   \psi_1 \left(\frac{2^{2j} \abs{v} }{\delta} \right) \frac{a^{'}(v)}{a(v)} \abs{v}^{\nu} \cdot  F_{k,a}(s,y)   \ dy \ dv
\\  &  + \int_{v,y} \abs{v}^{-\nu} \mathbbm{1}_{[0,\tilde{u}(s,y)]}(v)    \nabla_x \tilde{\Phi} (a(v)(t-s),x-y) \abs{v}^{-1+ \nu }  \\ & \times  \psi_1^{'} \left(\frac{2^{2j} \abs{v} }{\delta} \right) \frac{2^{2j} \abs{v} \sgn(v)}{\delta} \cdot  F_{k,a}(s,y)  \ dy \ dv.
\end{split}
\end{equation}		
Recall that $\nu \in (\frac{1}{2}, \frac{m}{2})$. Using \eqref{definition_av_space_regularity}, we have
\begin{equation} \label{bound_a_prime_v_space}
	\begin{split}
	\int_{ \{\abs{v} \gtrsim 2^{- 2j} \delta \}}   \frac{(a^{'}(v))^2 }{   (a(v))^3 }  \abs{v}^{2 \nu}      \ dv \lesssim ( 2^{- 2j} \delta)^{-m + 2 \nu }.
	\end{split}
\end{equation}
Using \eqref{Definition_Fka}, observe that
\begin{equation} \label{estimate_Fka}
	\begin{split}
		\|   F_{k,a}(t,\cdot)  \|_{L^2_{\omega,x}} \lesssim    (2^{- 2j} \delta)^{ -(\frac{m+1}{2}- \gamma^{\star})}    \| \mathring{g}_k(t,\cdot,\tilde{u})   \nabla \tilde{u}^{\frac{m+1}{2} - \gamma^{\star} }    \|_{L^2_{\omega,x}}.
	\end{split}
\end{equation}
 We estimate the first term on the right-hand side of \eqref{Expression_stoch_integral_dec_LP_space_2} using the expression \eqref{divergence_I_space}. For this derivation, we mainly rely on the heat kernel estimates \eqref{heat_kernel_estimate_tilde_before} and the assumption that $ \mathring{g}_k(\tilde{u}) \nabla  \tilde{u}^{\frac{m+1}{2}- \gamma^{\star} }  \in L^2_{\omega,t,x}$.
We begin estimating the first term on the right-hand side of \eqref{divergence_I_space}. Using Cauchy--Schwarz inequality, Young's convolution inequality, estimates \eqref{heat_kernel_estimate_tilde_before}, \eqref{bound_a_prime_v_space} and \eqref{estimate_Fka}, we have
\begin{equation*} 
		\begin{split}
& \mathbb{E} \int_{t,x}   \sum_{k=1}^{\infty}   \int_{0}^{t- \eta} \Bigg|  \int_{v,y} \abs{v}^{- \nu } \mathbbm{1}_{[0,\tilde{u}(s,y)]}(v)      \nabla_x \dbtilde{\Phi} (a(v)(t-s),x-y)  \\ & \times  \psi_1 \left(\frac{2^{2j} \abs{v} }{\delta} \right) \frac{a^{'}(v)}{a(v)} \abs{v}^{\nu }  \cdot  F_{k,a}(s,y) \  dy \ dv \Bigg|^2  ds  \  dx  \ dt 
\\ & \le  \mathbb{E} \int_{t} \sum_{k=1}^{\infty}  \int_{0}^{t- \eta}  \int_{ \{\abs{v} \ge 2^{- 2j} \delta \} }  \abs{v}^{-2 \nu} dv  \   \int_{ \{\abs{v} \ge 2^{- 2j} \delta \} }   \left( \frac{a^{'}(v) }{   a(v) } \right)^2 \abs{v}^{2 \nu} \\ & \times \int_x \Bigg|  \int_{y} \nabla_x \dbtilde{\Phi} (a(v)(t-s),x-y)    \cdot  F_{k,a}(s,y) \  dy    \Bigg|^2 \ dx \ dv \ ds    \ dt
\end{split}
\end{equation*}
\begin{equation} \label{estimate_Ia}
	\begin{split}
	& \lesssim   (2^{- 2j} \delta)^{-2 \nu +1 }  \int_{t}  \sum_{k=1}^{\infty} \int_{0}^{t- \eta}    \int_{ \{\abs{v} \ge 2^{- 2j} \delta \} } \left( \frac{a^{'}(v) }{   a(v) } \right)^2 \abs{v}^{2 \nu}  \\ & \times   \| \nabla_x \dbtilde{\Phi} (a(v)(t-s),\cdot) \|_{L^1_x}^2  \mathbb{E}   \|   F_{k,a}(s,\cdot)  \|^2_{L^2_x} \ dv \  ds \   dt
\\  & \lesssim  (2^{- 2j} \delta)^{-2 \nu +1 }  \int_{ \{\abs{v} \ge 2^{- 2j} \delta \}}   \frac{(a^{'}(v))^2 }{   (a(v))^3 }  \abs{v}^{2 \nu}      \ dv   \int_t \sum_{k=1}^{\infty} \int_{0}^{t- \eta}   \frac{\mathbb{E}   \|   F_{k,a}(s,\cdot)  \|^2_{L^2_x}}{t-s}  ds    \ dt
  \\ & \lesssim  (2^{- 2j} \delta)^{-2 \nu +1 } ( 2^{- 2j} \delta)^{-m + 2 \nu }   \int_t \sum_{k=1}^{\infty} \int_{0}^{t- \eta}   \frac{\mathbb{E}   \|   F_{k,a}(s,\cdot)  \|^2_{L^2_x}}{t-s}  ds    \ dt  
\\ & \lesssim \abs{ \log(\eta)}  (2^{- 2j} \delta)^{1 -m -2(\frac{m+1}{2}- \gamma^{\star})}  \sum_{k=1}^{\infty}        \| \mathring{g}_k   \nabla \tilde{u}^{\frac{m+1}{2} - \gamma^{\star} }    \|^2_{L^2_{\omega,t,x}}.
\end{split}
\end{equation}
The second term on the right-hand side of \eqref{divergence_I_space} is estimated similarly since $\abs{v} \sim 2^{-2j} \delta$ on the support of $\psi_1^{'} \left( \frac{2^{2j} \abs{v}}{\delta} \right)$. We have
	\begin{equation} \label{estimate_Ib}
	\begin{split}
		&    \mathbb{E} \int_{t,x} \sum_{k=1}^{\infty} \int_{0}^{t- \eta}  \Bigg|  \int_{v,y} \abs{v}^{- \nu } \mathbbm{1}_{[0,\tilde{u}(s,y)]}(v) \   \nabla_x \tilde{\Phi} (a(v)(t-s),x-y) \abs{v}^{-1+ \nu} \\ & \times   \psi_1^{'} \left(\frac{2^{2j} \abs{v} }{\delta} \right) \frac{2^{2j} \abs{v} \sgn(v)}{\delta}  \cdot  F_{k,a}(s,y) \ dy \ dv  \Bigg|^2  ds   \  dx  \ dt
\\ 	 & \lesssim \abs{\log(\eta)}  (2^{-2j} \delta)^{-(m-1) -2(\frac{m+1}{2} - \gamma^{\star})} \sum_{k=1}^{\infty}       \| \mathring{g}_k   \nabla \tilde{u}^{\frac{m+1}{2} - \gamma^{\star} }  \|^2_{L^2_{\omega,t,x}}     .
	\end{split}
\end{equation}
The first term on the right-hand side of \eqref{Expression_stoch_integral_dec_LP_space_2}  is estimated using \eqref{estimate_Ia} and \eqref{estimate_Ib}
\begin{equation} \label{estimate_I_conclusion_before_Bernstein_term}
	\begin{split}
		&  \mathbb{E} \int_{t,x} \sum_{k=1}^{\infty}    \int_{0}^{t- \eta}   \abs{ \nabla_x \cdot  \rom{1}  }^2 \  ds  \  dx \ dt \\ &  \lesssim \abs{\log(\eta)}  (2^{-2j} \delta)^{-(m-1) -2(\frac{m+1}{2} - \gamma^{\star})} \sum_{k=1}^{\infty}       \| \mathring{g}_k   \nabla \tilde{u}^{\frac{m+1}{2} - \gamma^{\star} }  \|^2_{L^2_{\omega,t,x}} .
	\end{split}
\end{equation}
The term $ \nabla_x \cdot	 \rom{2}$ on the right-hand side of  \eqref{Expression_stoch_integral_dec_LP_space_2} is treated as follows
\begin{equation*} 
	\begin{split}
		&  \nabla_x \cdot  \rom{2}   \\ & =   \nabla_x \cdot		\int_{y} \Phi(a(\tilde{u}(s,x-y))(t-s),y)    \tilde{\psi}_1 \left(\frac{2^{2j} \abs{\tilde{u}(s,x-y)} }{\delta} \right)    \nabla_x F_{k}(s,x-y) \ dy 
	\\	 & =    \nabla_x \cdot		\int_{y,v}  \mathbbm{1}_{[0,\tilde{u}(s,y)]}(v)   \ \partial_v \left( \Phi(a(v)(t-s),x-y) \tilde{\psi}_1 \left(\frac{2^{2j} \abs{v} }{\delta} \right)   \right)  \ dv \   \nabla_y F_{k}(s,y)  \ dy 
	\\	& =      		\int_{v,y} \abs{v}^{- \nu} \mathbbm{1}_{[0,\tilde{u}(s,y)]}(v)  \nabla_x \tilde{\Phi}(a(v)(t-s),x-y)   \frac{a^{'}(v)}{a(v)} \abs{v}^{\nu} \tilde{\psi}_1 \left(\frac{2^{2j} \abs{v} }{\delta} \right)  \   \nabla_y F_{k}(s,y)  \ dy \ dv 
		\end{split}
	\end{equation*}
\begin{equation} \label{Treat_II_term}
	\begin{split}
		 & + 	\int_{v,y}   \abs{v}^{- \nu} \mathbbm{1}_{[0,\tilde{u}(s,y)]}(v)   \nabla_x \Phi(a(v)(t-s),x-y) \abs{v}^{-1+ \nu} \\ & \times  \tilde{\psi}^{'}_1 \left(\frac{2^{2j} \abs{v} }{\delta} \right) \frac{2^{2j} \abs{v} \sgn(v)}{\delta}     \   \nabla_y F_{k}(s,y)  \ dy \ dv.
	\end{split}
\end{equation}
Recall that  $\gamma^{\star} <  \frac{m+1}{2}$. Using \eqref{Definition_Fk}, we have  
\begin{equation} \label{Gradient_Fkjdelta}
	\begin{split}	
		&	\nabla_x  F_{k}(t,x) 
		\\ & =  \abs{\tilde{u}(t,x)}^{-(\frac{m+1}{2} - \gamma^{\star} )}     \psi_1^{'} \left(\frac{2^{2j} \abs{\tilde{u}(t,x)} }{\delta} \right) \frac{2^{2j}\abs{\tilde{u}(t,x)}}{\delta } \sgn(\tilde{u}(t,x))\\ & \times  \nabla_x \tilde{u}(t,x)  \abs{\tilde{u}(t,x)}^{\frac{m+1}{2} - \gamma^{\star} -1 }   \mathring{g}_k(t,x,\tilde{u}(t,x)) \\ & +  \abs{\tilde{u}(t,x)}^{-(\frac{m+1}{2} - \gamma^{\star} )}  \psi_1 \left(\frac{2^{2j} \abs{\tilde{u}(t,x)} }{\delta} \right) \abs{\tilde{u}(t,x)}^{\frac{m+1}{2} - \gamma^{\star} }  \nabla_x \mathring{g}_k(t,x,\tilde{u}(t,x)).
	\end{split}
\end{equation}
Note that
\begin{equation} \label{estimate_gradient_Fk}
	\begin{split}
		& \|  	\nabla_x  F_{k}(t,\cdot)   \|_{L^2_{\omega,x}} \\ & \lesssim    (2^{- 2j} \delta)^{ -(\frac{m+1}{2}- \gamma^{\star})} \left(   \| \mathring{g}_k(t,\cdot,\tilde{u})   \nabla \tilde{u}^{\frac{m+1}{2} - \gamma^{\star} }    \|_{L^2_{\omega,x}} +  \|      \tilde{u}^{\frac{m+1}{2} - \gamma^{\star} }  \nabla \mathring{g}_k(t,\cdot,\tilde{u})  \|_{L^2_{\omega,x}} \right).
	\end{split}
\end{equation}
 We estimate the second term on the right-hand side of \eqref{Expression_stoch_integral_dec_LP_space_2} using the expression \eqref{Treat_II_term}. For this derivation, we mainly rely on the heat kernel estimates \eqref{heat_kernel_estimate_tilde_before} and the assumption that $ \mathring{g}_k(\tilde{u}) \nabla  \tilde{u}^{\frac{m+1}{2}- \gamma^{\star} }, \tilde{u}^{\frac{m+1}{2}- \gamma^{\star} } \nabla \mathring{g}_k(\tilde{u}) \in L^2_{\omega,t,x}$. We begin estimating the first term on the right-hand side of \eqref{Treat_II_term}. Using  Cauchy--Schwarz inequality, Young's convolution inequality, estimates \eqref{heat_kernel_estimate_tilde_before}, \eqref{bound_a_prime_v_space} and \eqref{estimate_gradient_Fk} we have
\begin{equation} \label{estimate_II_first_term}
	\begin{split}
&   \mathbb{E} \int_{t,x} \sum_{k=1}^{\infty} \int_{0}^{t- \eta}  \Bigg| 	\int_{v,y}  \abs{v}^{- \nu} \mathbbm{1}_{[0,\tilde{u}(s,y)]}(v)  \nabla_x \tilde{\Phi}(a(v)(t-s),x-y)   \frac{a^{'}(v)}{a(v)} \abs{v}^{\nu}  \\ & \times  \tilde{\psi}_1 \left(\frac{2^{2j} \abs{v} }{\delta} \right)\   \nabla_y F_{k}(s,y)  \ dy  \ dv   \Bigg|^2 ds  \  dx \ dt
\\ & \le    \mathbb{E} \int_{t} \sum_{k=1}^{\infty} \int_0^{t- \eta}        \int_{\{\abs{v} \gtrsim 2^{- 2j} \delta  \} } \abs{v}^{-2 \nu} dv   \int_{ \{\abs{v} \gtrsim 2^{- 2j} \delta  \} }  \left(\frac{ a^{'}(v) }{ a(v)}  \right)^2 \abs{v}^{2 \nu}
\\			 & \times \int_x	\Bigg|  \int_{y} \nabla_x \tilde{\Phi}(a(v)(t-s),x-y)  \cdot  \nabla_y F_{k}(s,y)  \  dy \Bigg|^2  dx \ dv \ ds  \ dt
	\\	& \lesssim (2^{- 2j} \delta)^{-2 \nu +1 } \int_{ \{\abs{v} \gtrsim 2^{- 2j} \delta  \} }  \frac{(a^{'}(v))^2 }{   (a(v))^3 }  \abs{v}^{2 \nu}	dv  \  \mathbb{E} \int_{t}   \sum_{k=1}^{\infty}  \int_0^{t- \eta}     \frac{\| \nabla_x  F_{k}(s,\cdot) \|^2_{L^2_x} }{t-s}       \ ds   \ dt \\
& \lesssim   \abs{\log(\eta)} (2^{- 2j} \delta)^{1 -m -2(\frac{m+1}{2} - \gamma^{\star})} \left(\sum_{k=1}^{\infty}       \| \mathring{g}_k   \nabla  \tilde{u}^{\frac{m+1}{2}- \gamma^{\star}}   \|^2_{L^2_{\omega,t,x}}  +  \sum_{k=1}^{\infty}    \|      \tilde{u}^{\frac{m+1}{2} - \gamma^{\star} }  \nabla \mathring{g}_k  \|^2_{L^2_{\omega,t,x}} \right).
			\end{split}
\end{equation}
The second term on the right-hand side of \eqref{Treat_II_term} is estimated similarly since $\abs{v} \sim 2^{-2j} \delta$ on the support of $\tilde{\psi}_1^{'} \left( \frac{2^{2j} \abs{v}}{\delta} \right)$. 
We have 
\begin{equation} \label{estimate_II_second_term}
	\begin{split}
		&   \mathbb{E} \int_{t,x} \sum_{k=1}^{\infty} \int_{0}^{t- \eta}  \Bigg| 	\int_{v,y} \abs{v}^{- \nu}  \mathbbm{1}_{[0,\tilde{u}(s,y)]}(v)   \nabla_x \Phi(a(v)(t-s),x-y) \abs{v}^{-1+ \nu} \\ & \times \tilde{\psi}^{'}_1 \left(\frac{2^{2j} \abs{v} }{\delta} \right) \frac{2^{2j} \abs{v} \sgn(v)}{\delta}     \   \nabla_y F_{k}(s,y)  \ dy \ dv \Bigg|^2 ds  \  dx \ dt
		\\	 & \lesssim \abs{\log(\eta)}  (2^{-2j} \delta)^{-(m-1) -2(\frac{m+1}{2} - \gamma^{\star})} \Bigg(\sum_{k=1}^{\infty}       \| \mathring{g}_k   \nabla  \tilde{u}^{\frac{m+1}{2}- \gamma^{\star}}   \|^2_{L^2_{\omega,t,x}}  \\ &     +  \sum_{k=1}^{\infty}    \|      \tilde{u}^{\frac{m+1}{2} - \gamma^{\star} }  \nabla \mathring{g}_k  \|^2_{L^2_{\omega,t,x}} \Bigg) .
	\end{split}
\end{equation}
The second term on the right-hand side of \eqref{Expression_stoch_integral_dec_LP_space_2}  is estimated using \eqref{estimate_II_first_term} and \eqref{estimate_II_second_term}
\begin{equation} \label{estimate_II_conclusion_before_Bernstein_term}
	\begin{split}
	&  \mathbb{E} \int_{t,x} \sum_{k=1}^{\infty}    \int_{0}^{t- \eta}   \abs{ \nabla_x \cdot  \rom{2}  }^2 \  ds  \  dx \ dt \\ &  \lesssim \abs{\log(\eta)}  (2^{-2j} \delta)^{-(m-1) -2(\frac{m+1}{2} - \gamma^{\star})}  \Bigg(\sum_{k=1}^{\infty}       \| \mathring{g}_k   \nabla  \tilde{u}^{\frac{m+1}{2}- \gamma^{\star}}   \|^2_{L^2_{\omega,t,x}} 
	 \\ &    +  \sum_{k=1}^{\infty}    \|      \tilde{u}^{\frac{m+1}{2} - \gamma^{\star} }  \nabla \mathring{g}_k  \|^2_{L^2_{\omega,t,x}} \Bigg) .
	\end{split}
\end{equation}
Combining \eqref{estimate_I_conclusion_before_Bernstein_term} and \eqref{estimate_II_conclusion_before_Bernstein_term}, we have obtained the following estimate for \eqref{Expression_stoch_integral_dec_LP_space_2} 
\begin{equation} \label{estimate_t_small_second_step}
	\begin{split}
  \| B(0,t-\eta)  \|_{L^2_{\omega, t,x}}^2 
 & \lesssim  \abs{\log(\eta)} 2^{-4j}   (2^{- 2j} \delta)^{1 -m -2(\frac{m+1}{2} - \gamma^{\star})}       \Bigg(\sum_{k=1}^{\infty}       \| \mathring{g}_k   \nabla  \tilde{u}^{\frac{m+1}{2}- \gamma^{\star}}   \|^2_{L^2_{\omega,t,x}}    \\ &   +  \sum_{k=1}^{\infty}    \|      \tilde{u}^{\frac{m+1}{2} - \gamma^{\star} }  \nabla \mathring{g}_k  \|^2_{L^2_{\omega,t,x}} \Bigg)
\\ & \lesssim \eta^{\epsilon} \ 2^{-4j}   (2^{- 2j} \delta)^{1 -m -2(\frac{m+1}{2} - \gamma^{\star})}    \Bigg(\sum_{k=1}^{\infty}       \| \mathring{g}_k   \nabla  \tilde{u}^{\frac{m+1}{2}- \gamma^{\star}}   \|^2_{L^2_{\omega,t,x}}    \\ &   +  \sum_{k=1}^{\infty}    \|      \tilde{u}^{\frac{m+1}{2} - \gamma^{\star} }  \nabla \mathring{g}_k  \|^2_{L^2_{\omega,t,x}} \Bigg).
	\end{split}
\end{equation}
	We then provide a bound on the second term on the right-hand side of \eqref{Expression_stoch_integral_dec_LP_space}. Using Young's convolution inequality and It{\^o}'s isometry, we have 
	
\begin{equation} \label{Expression_stoch_integral_dec_LP_space_2_B}
	\begin{split}
		 	& \| B(t-\eta,t)  \|_{L^2_{\omega, t,x}}^2 \\ & =  \Bigg \|  \sum_{k=1}^{\infty}  \int_{z}   \varphi_j(x-z) \int_{t-\eta}^{t}  \int_{y} \Phi(a(\tilde{u}(s,y))(t-s),z-y) \\ & \times  F_{k}(s,y)  \  dy  \ d\beta_k(s)  \ dz \Bigg \|_{L^2_{\omega, t,x}}^2 
		 	\\ 	& \lesssim \Bigg \|  \sum_{k=1}^{\infty}   \int_{t-\eta}^{t}  \int_{y} \Phi(a(\tilde{u}(s,y))(t-s),x-y) F_{k}(s,y) \  dy  \ d\beta_k(s)  \Bigg \|_{L^2_{\omega, t,x}}^2 
	\\	& =   \int_{t,x}    \mathbb{E} \int_{t-\eta}^{t} \sum_{k=1}^{\infty}   \Bigg| 
		\int_{y} \Phi(a(\tilde{u}(s,y))(t-s),x-y)   F_{k}(s,y)    \ dy \Bigg|^2 ds \   dx \ dt	 
\\	& =    \int_{t,x}     \mathbb{E} \int_{t-\eta}^{t}  \sum_{k=1}^{\infty}   \Bigg| 
	\int_{y,v}   \mathbbm{1}_{[0,\tilde{u}(s,y)]}(v) \ \partial_v \left( \Phi(a(v)(t-s),x-y) \psi_1 \left(\frac{2^{2j} \abs{v} }{\delta} \right) \right)   \ dv \\ & \times  \mathring{g}_k(s,y,\tilde{u}(s,y)) \ dy \Bigg|^2 ds \   dx \ dt
			\\	& \lesssim     \int_{t,x}     \mathbb{E} \int_{t-\eta}^{t}  \sum_{k=1}^{\infty}  | \rom{3} |^2   \ ds \   dx \ dt +       \int_{t,x}     \mathbb{E} \int_{t-\eta}^{t}  \sum_{k=1}^{\infty}  | \rom{4} |^2   \ ds \   dx \ dt,
\end{split}
\end{equation}
where
\begin{equation*}
	\begin{split}
		\rom{3} & := \int_{v,y}  \abs{v}^{- \nu}  \psi_1 \left(\frac{2^{2j} \abs{v} }{\delta} \right) \mathbbm{1}_{[0,\tilde{u}(s,y)]}(v)   \tilde{\Phi}(a(v)(t-s),x-y)  \frac{a^{'}(v)}{a(v)} \abs{v}^{\nu} \\ &  \times  \mathring{g}_k(s,y,\tilde{u}(s,y)) \ dy  \ dv ,
		\\  \rom{4} & := \int_{v,y}  \abs{v}^{- \nu}   \mathbbm{1}_{[0,\tilde{u}(s,y)]}(v)  \Phi(a(v)(t-s),x-y)  \abs{v}^{-1+\nu} \psi_1^{'} \left(\frac{2^{2j} \abs{v} }{\delta} \right) \frac{2^{2j} \abs{v}\sgn(v)}{\delta}   \\ &  \times  \mathring{g}_k(s,y,\tilde{u}(s,y))  \ dy  \ dv.
	\end{split}
\end{equation*}
Recall that $\varpi^{-1}+ \varkappa^{-1} =1 $. We estimate the first term on the right-hand side of \eqref{Expression_stoch_integral_dec_LP_space_2_B} using Cauchy--Schwarz inequality, Young's convolution inequality, estimate \eqref{_no_grad_heat_kernel_estimate_tilde_before}   and H{\"o}lder's inequality
\begin{equation*} 
	\begin{split}
		 &	  \int_{t,x}     \mathbb{E} \int_{t-\eta}^{t}  \sum_{k=1}^{\infty}  | \rom{3} |^2   \ ds \   dx \ dt 
		 	\\ &	=  \int_{t,x}     \mathbb{E} \int_{t-\eta}^{t}  \sum_{k=1}^{\infty}  \Bigg| \int_{v}  \abs{v}^{- \nu}  \psi_1 \left(\frac{2^{2j} \abs{v} }{\delta} \right) \int_y  \mathbbm{1}_{[0,\tilde{u}(s,y)]}(v)  \tilde{\Phi}(a(v)(t-s),x-y)   \frac{a^{'}(v)}{a(v)}  \abs{v}^{ \nu} 
		 							\end{split}
	 	\end{equation*} 
 	\begin{equation} \label{Estimate_III_space_regularity}
 \begin{split}	
		  & \times    \mathring{g}_k(s,y,\tilde{u}(s,y))  \ dy  \ dv  \Bigg|^2   \ ds \   dx \ dt 
		 \\	&	\lesssim  \int_{t}     \mathbb{E} \int_{t-\eta}^{t}  \sum_{k=1}^{\infty}  \int_{ \{\abs{v} \ge  2^{-2j} \delta  \}}   \abs{v}^{- 2 \nu}  \ dv \int_v  \abs{v}^{2( \nu-1)} \int_x \Bigg|  \int_y \mathbbm{1}_{[0,\tilde{u}(s,y)]}(v)   \tilde{\Phi}(a(v)(t-s),x-y)      \\ & \times  \mathring{g}_k(s,y,\tilde{u}(s,y))  \ dy    \Bigg|^2     dx \ dv    \ ds \ dt 	
		 \\	 & \lesssim  (2^{-2j} \delta)^{- 2 \nu +1}  \int_{t}     \mathbb{E} \int_{t-\eta}^{t}  \sum_{k=1}^{\infty}  \int_{v}    \abs{v}^{2( \nu-1)}      \| \tilde{\Phi}(a(v)(t-s),\cdot) \|_{L^1_x}^2 \\ & \times  \int_{x}        \mathbbm{1}_{[0,\tilde{u}(s,x)]}(v)   | \mathring{g}_k(s,x,\tilde{u}(s,x))  |^2       \ dx     \ dv    \ ds \ dt
		 \\ & \lesssim  (2^{-2j} \delta)^{- 2 \nu +1}  \int_{t}     \mathbb{E} \int_{t-\eta}^{t}  \sum_{k=1}^{\infty}   \int_{x,v}  \abs{v}^{2( \nu-1)}        \mathbbm{1}_{[0,\tilde{u}(s,x)]}(v)   | \mathring{g}_k(s,x,\tilde{u}(s,x))  |^2     \ dv  \ dx   \ ds \ dt 
\\		 & =  (2^{-2j} \delta)^{- 2 \nu +1}  \int_{t}  \mathbb{E}    \int_{t-\eta}^{t}    \sum_{k=1}^{\infty}   \int_x   \abs{\tilde{u}(s,x)}^{2 \nu-1}           | \mathring{g}_k(s,x,\tilde{u}(s,x)) |^{2(\frac{1}{\varpi} + \frac{1}{\varkappa})}       \ dx   \ ds \ dt 
\\ &   \le (2^{-2j} \delta)^{- 2 \nu +1}  \int_t  \mathbb{E} \int_{t-\eta}^t \sum_{k=1}^{\infty}   \left( \int_x | \mathring{g}_k(s,x,\tilde{u}(s,x))     |^{2 } dx \right)^{\frac{1}{\varpi}} \\ & \times  \left( \int_x | \tilde{u}(s,x) |^{\varkappa(2 \nu - 1)}     | \mathring{g}_k(s,x,\tilde{u}(s,x))     |^{2} \  dx \right)^{\frac{1}{\varkappa}}    \  ds \ dt		
\\		 & \le  (2^{-2j} \delta)^{- 2 \nu +1}   \eta^{\frac{1}{\varpi}}  \int_t \sum_{k=1}^{\infty}   \esssup_{s \in [t-\eta,t]}   \left \|  \mathring{g}_k^2 (s,\cdot,\tilde{u})  \right \|_{L^1_{\omega,x}}^{\frac{1}{\varpi}}  \left( \int_{t-\eta}^t  \left \|   \tilde{u}(s,\cdot)^{\varkappa(2 \nu - 1)}   \mathring{g}_k^2 (s,\cdot,\tilde{u})   \right \|_{L^1_{\omega, x} }      \   ds \right)^{\frac{1}{\varkappa}} dt	 
\\  & \lesssim (2^{-2j} \delta)^{- 2 \nu +1}  \eta^{\frac{1}{\varpi}}        \sum_{k=1}^{\infty}   \esssup_{s \in [0,T]}   \left \|  \mathring{g}_k^2(s,\cdot,\tilde{u})   \right \|_{L^1_{\omega,x}}^{\frac{1}{\varpi}}  	
 \left \|    \tilde{u}^{\varkappa(2 \nu - 1)}  \mathring{g}_k^2   \right \|_{L^1_{\omega,t,x}}^{\frac{1}{\varkappa}}.
	\end{split}
\end{equation}
The second term on the right-hand side of \eqref{Expression_stoch_integral_dec_LP_space_2_B} is estimated similarly since $\abs{v} \sim 2^{-2j} \delta$ on the support of $\psi_1^{'} \left( \frac{2^{2j} \abs{v}}{\delta} \right)$. 
We have 
\begin{equation} \label{Estimate_iV_space_regularity}
	\begin{split}
		&	  \int_{t,x}     \mathbb{E} \int_{t-\eta}^{t}  \sum_{k=1}^{\infty}  | \rom{4} |^2   \ ds \   dx \ dt \\ & \lesssim (2^{-2j} \delta)^{- 2 \nu +1}  \eta^{\frac{1}{\varpi}}        \sum_{k=1}^{\infty}   \esssup_{s \in [0,T]}   \left \|  \mathring{g}_k^2(s,\cdot,\tilde{u})   \right \|_{L^1_{\omega,x}}^{\frac{1}{\varpi}}    \left \|    \tilde{u}^{\varkappa(2 \nu - 1)}  \mathring{g}_k^2   \right \|_{L^1_{\omega,t,x}}^{\frac{1}{\varkappa}}.	
	\end{split}
\end{equation}
Combining \eqref{Estimate_III_space_regularity} and \eqref{Estimate_iV_space_regularity}, we have obtained the following estimate for \eqref{Expression_stoch_integral_dec_LP_space_2_B} \begin{equation*}  \label{Estimate_B_t_small_spatial_regularity}
		\| B(t-\eta,t)  \|_{L^2_{\omega, t,x}}^2 
	\end{equation*}
 \begin{equation}  \label{Estimate_B_t_small_spatial_regularity}
		 \lesssim (2^{-2j} \delta)^{- 2 \nu +1}  \eta^{\frac{1}{\varpi}}        \sum_{k=1}^{\infty}   \esssup_{s \in [0,T]}   \left \|  \mathring{g}_k^2 (s,\cdot,\tilde{u}) \right \|_{L^1_{\omega,x}}^{\frac{1}{\varpi}}  \left \|    \tilde{u}^{\varkappa(2 \nu - 1)}  \mathring{g}_k^2   \right \|_{L^1_{\omega,t,x}}^{\frac{1}{\varkappa}}. 	
\end{equation}
We denote
\begin{equation*}
	\begin{split}
	\mathcal{K}^1 & =  \left(  \sum_{k=1}^{\infty}     \|  \mathring{g}_k   \nabla  \tilde{u}^{\frac{m+1}{2}- \gamma^{\star} }   \|_{L^2_{\omega,t,x}}^2 \right)^{\frac{1}{2}} + \left(  \sum_{k=1}^{\infty}     \| \tilde{u}^{\frac{m+1}{2}- \gamma^{\star} }  \nabla \mathring{g}_k      \|_{L^2_{\omega,t,x}}^2 \right)^{\frac{1}{2}}  \\ &    +   \left( \sum_{k=1}^{\infty}  \esssup_{t \in [0,T]}    \left \|    \mathring{g}_k^2(t,\cdot,\tilde{u})  \right \|_{L^1_{\omega,x}}^{\frac{1}{\varpi}}   \left \|    \tilde{u}^{\varkappa(2 \nu - 1)}  \mathring{g}_k^2   \right \|_{L^1_{\omega,t,x}}^{\frac{1}{\varkappa}} \right)^{\frac{1}{2}} .
	\end{split}
\end{equation*}
From \eqref{estimate_t_small_second_step} and \eqref{Estimate_B_t_small_spatial_regularity}, we have the estimate of \eqref{Expression_stoch_integral_dec_LP_space} 
\begin{equation} \label{Estimate_A_B}
		\begin{split}
 \left\Vert \int_{v} \chi_{j}^{1,1} dv \right\Vert_{L^2_{\omega, t,x}} & 	\le \| B(0,t-\eta)  \|_{L^2_{\omega, t,x}} + \| B(t-\eta,t)  \|_{L^2_{\omega, t,x}} \\ &
\lesssim \left(\eta^{\frac{\epsilon}{2}} 2^{-2j}    (2^{- 2j} \delta)^{\frac{1 -m}{2} -(\frac{m+1}{2} - \gamma^{\star})} + \eta^{\frac{1}{2 \varpi}} ( 2^{-2j} \delta)^{\frac{1}{2} - \nu}  \right) 	\mathcal{K}^1.
			 \end{split}
\end{equation}
Equilibrating the terms on the right-hand side above
\begin{equation*}
\eta^{\frac{\epsilon}{2}} 2^{-2j}    (2^{- 2j} \delta)^{\frac{1 -m}{2} -(\frac{m+1}{2} - \gamma^{\star})}=  \eta^{\frac{1}{2 \varpi}} ( 2^{-2j} \delta)^{\frac{1}{2} - \nu},
\end{equation*}
we find the optimal value of $\eta$ to be
\begin{equation*}
\eta =    \left(2^{- 2j(\gamma^{\star} + \nu + \frac{1}{2} -m )} \delta^{\gamma^{\star} + \nu -\frac{1}{2} -m } \right)^{\frac{2 \varpi}{1- \epsilon \varpi}}.
\end{equation*}
Plugging this value of $\eta$ into \eqref{Estimate_A_B}, we have
\begin{equation} \label{final_stochastic_part_space}
	\begin{split}
		 \left\Vert \int_{v} \chi_{j}^{1,1} dv \right\Vert_{L^1_{\omega, t,x}} 	
 & \lesssim  \frac{2^{2j((m-\frac{1}{2} - \nu - \gamma^{\star})\frac{1}{1 - \epsilon \varpi} - \frac{1}{2}+ \nu) }}{ \delta^{(m+ \frac{1}{2} - \nu - \gamma^{\star})\frac{1}{1 - \epsilon \varpi}  - \frac{1}{2} + \nu}} \mathcal{K}^1.
	\end{split}
\end{equation}	
		
\paragraph{Step 3}  Let $j >0 $ be arbitrary and fixed. Recall that $\gamma \in (0,\min(1+\epsilon,m))$. Following the same arguments as in \cite [Proof of Lemma 4.2 Step 2]{gess2019optimal}, we obtain 
\begin{equation} \label{Estimate_cut_off_time_reg_small_v_space_new}
	\begin{split}
		&	\left\Vert  \int_{v} \chi_{j}^{1,2} dv  \right\Vert_{L^1_{\omega,t,x}} \\ &  = \left\Vert   \int_{v}  \mathcal{F}_{t,x}^{-1}     \psi_1 \left(\frac{2^{2j} \abs{v} }{\delta} \right) \frac{ \abs{v}^{\gamma-1} }{ \mathcal{L}(i\tau, \xi,v) }  \mathcal{F}_{t,x} \abs{v}^{1- \gamma} h_{j}   \ dv \right\Vert_{L^1_{\omega,t,x}}  \\ 
		&	\lesssim \frac{2^{2j(m-1- \gamma)}}{\delta^{m- \gamma} }  \| \abs{v}^{1- \gamma}  h  \|_{L^1_{\omega} \mathcal{M}_{\text{TV}}} .
	\end{split}
\end{equation}


\paragraph{Step 4} Let $ j > 0 $ be arbitrary and fixed. We have 
\begin{equation} \label{Expression_kinetic_time_small_v_space_new}
	\begin{split}
		\int_{v} \chi_{j}^{1,3} dv & = -  \int_{v}  \mathcal{F}_{t,x}^{-1}  \hat{\varphi}_j  \psi_1^{'} \left(\frac{2^{2j} \abs{v} }{\delta} \right) \frac{2^{2j} \abs{v} }{\delta}  \frac{ \sgn(v) \abs{v}^{ \gamma-1}}{ \mathcal{L}(i\tau, \xi,v)} \mathcal{F}_{t,x} \abs{v}^{- \gamma} \tilde{h}   \ dv 
		 \\  & + m (m -1)  \int_{v}  \mathcal{F}_{t,x}^{-1}  \hat{\varphi}_j \psi_1 \left(\frac{2^{2j} \abs{v} }{\delta} \right)  \frac{ 4 \pi^2 \abs{v}^{m-2 + \gamma} \sgn(v) \abs{\xi}^2}{ \mathcal{L}(i\tau, \xi,v)^2} \mathcal{F}_{t,x}    \abs{v}^{- \gamma} \tilde{h}  \ dv. 
			\end{split}
\end{equation}
Following the same arguments as in \cite [Proof of Lemma 4.2 Step 3]{gess2019optimal}, we obtain 
\begin{equation} \label{Kinetic_First_Averaging_small_v_space_new}
	\begin{split}
		 \left\Vert \int_{v} \chi_{j}^{1,3} dv \right\Vert_{L^1_{\omega,t,x}}  \lesssim \frac{2^{2j(m-1- \gamma)}}{\delta^{m- \gamma} }  \| \abs{v}^{- \gamma} \tilde{h}  \|_{L^1_{\omega} \mathcal{M}_{\text{TV}}} . 
	\end{split} 
\end{equation}

		\paragraph{Step 5} We denote
		\begin{equation*}
			\begin{split}
			\mathcal{K}^2 & = \| \abs{v}^{1- \gamma}  h  \|_{L^1_{\omega} \mathcal{M}_{TV}}  +  \| \abs{v}^{- \gamma } \tilde{h}   \|_{L^1_{\omega} \mathcal{M}_{\text{TV}}}. 
			\end{split}
		\end{equation*}
		We aim to conclude by real interpolation. For $z>0$, let
		\begin{equation*}
			K(z, \bar{\chi}_{j} ):= \inf  \left \{ \left \| \int_v \chi^1_{j} \ dv   \right \|_{ L^{1}_{\omega, t,x}}  + z \left \| \int_v \chi^0_{j} \ dv   \right \|_{ L^{\beta}_{\omega,t,x}}   : \ \bar{\chi}_{j} = \int_v \chi^0_{j} \ dv + \int_v \chi^1_{j} \ dv  \right \}.
		\end{equation*}
		By the estimates \eqref{velocity_degenerate},  \eqref{final_stochastic_part_space},  \eqref{Estimate_cut_off_time_reg_small_v_space_new} and \eqref{Kinetic_First_Averaging_small_v_space_new}, we have
		\begin{equation} \label{K_functional_bootstrap}
			\begin{split}
			 K(z, \bar{\chi}_{j} )  & \lesssim  \frac{2^{2j((m-\frac{1}{2} - \nu - \gamma^{\star})\frac{1}{1 - \epsilon \varpi} - \frac{1}{2}+ \nu) }}{ \delta^{(m+ \frac{1}{2} - \nu - \gamma^{\star})\frac{1}{1 - \epsilon \varpi}  - \frac{1}{2}+ \nu  }} \mathcal{K}^1 +       \frac{2^{2j(m-1- \gamma)}}{\delta^{m - \gamma} }  \mathcal{K}^2   + z \frac{\delta^{\rho }}{2^{2j\rho}}   \| \chi   \|_{ L^{\beta}_{\omega,t,x,v}}.
			\end{split}  
		\end{equation}
We denote $ \mathcal{K}=  \mathcal{K}^1 +  \mathcal{K}^2$ and recall that $ \gamma^{\star}=  \gamma(1- \epsilon \varpi) + \epsilon \varpi(m+ \frac{1}{2} - \nu)$. Since $j>0$, we have for \eqref{K_functional_bootstrap}
		\begin{equation} \label{estimate_K_functional_after_gamma_star}
		\begin{split}
			K(z, \bar{\chi}_{j} )  & \lesssim  \frac{2^{2j( m- \frac{1}{1- \epsilon \varpi} - \gamma)}}{ \delta^{m- \gamma} } \mathcal{K}^1 +       \frac{2^{2j(m-1- \gamma)}}{\delta^{m - \gamma} }  \mathcal{K}^2   + z \frac{\delta^{\rho }}{2^{2j\rho}}   \| \chi   \|_{ L^{\beta}_{\omega,t,x,v}}
			\\ &  \lesssim        \frac{2^{2j(m-1- \gamma)}}{\delta^{m - \gamma} }  \mathcal{K}   + z \frac{\delta^{\rho }}{2^{2j\rho}}   \| \chi   \|_{ L^{\beta}_{\omega,t,x,v}}.
		\end{split}  
	\end{equation}
	We equilibrate the estimates on the right-hand side above as follows 
		\begin{gather*} 
			\delta^{-a} c^{- a +1  } =  z \delta^{b} c^{b},
		\end{gather*}
		where $ a:= m - \gamma $, $b:= \rho $ and $ c:= 2^{-2j}$. This allows us to derive the value of $ \delta$ from the  expression above, namely $ \delta = z^{- \frac{1}{a+b}}  c^{ \frac{1-a-b}{a+b}} $ and plug this into the right-hand side of \eqref{estimate_K_functional_after_gamma_star} that becomes 
		\begin{equation} \label{z_small}
			z^{- \theta} K(z, \bar{\chi}_{j} ) \lesssim  2^{-j \tilde{\kappa}_x }  \left(  \mathcal{K} + \| \chi    \|_{ L^{\beta}_{\omega,t,x,v}} \right),  
		\end{equation} 
		where $ \theta := \frac{a}{a+b}= \frac{ m - \gamma  }{m - \gamma  + \rho} $ and
		\begin{equation*} 
			\begin{split}
			\tilde{\kappa}_x & := \frac{2b}{a+b}= \frac{2 \rho }{m - \gamma  + \rho}.
			\end{split}
		\end{equation*}
		By \cite[Theorem 5.2.1]{bergh2012interpolation}, we have
		\begin{equation*}
			\left(L^1_{\omega,t,x}, L^{\beta}_{\omega,t,x} \right)_{\theta, \infty}= L^{\tilde{p}, \infty}_{\omega,t,x}, 
		\end{equation*}
		with
		\begin{equation} \label{integrability_parameter}
			\tilde{p}= \frac{1}{1- \theta \rho}  = \frac{a+b}{a(1- \rho)+b}= \frac{m - \gamma  + \rho}{\rho +(1-\rho)(m - \gamma )}.
		\end{equation}
		Taking the supremum over $z>0$ on both sides of \eqref{z_small}, we have
		\begin{equation} \label{switch_inequality} 
			\| \bar{\chi}_{j}  \|_{L^{\tilde{p}, \infty}_{\omega,t,x}}    \lesssim  2^{-j \tilde{\kappa}_x} \left(  \mathcal{K} + \| \chi    \|_{ L^{\beta}_{\omega,t,x,v}} \right).  
		\end{equation}
		Note that the integrands in the above norms are supported on $[0,T]$ due to \eqref{Localization_in_time_degenerate_spatial_regularity} and \eqref{Localization_in_time_nondegenerate_spatial_regularity}. Since $\Omega \times [0,T] \times \mathbb{T}^d$ has finite measure, we may use $L^{\tilde{p}, \infty}_{\omega,t,x} \hookrightarrow L^{p}_{\omega,t,x}$ for $ p =  \tilde{p} - \epsilon$ (see \cite[Exercise 1.1.11]{grafakos2008classical}) in \eqref{switch_inequality}
		 	\begin{equation*} 
		 	\| \bar{\chi}_{j}  \|_{L^{p}_{\omega,t,x}}    \lesssim  2^{-j \tilde{\kappa}_x} \left(  \mathcal{K} + \| \chi    \|_{ L^{\beta}_{\omega,t,x,v}} \right).  
		 \end{equation*}
		Multiplying the above expression by $  2^{j \tilde{\kappa}_x} $ and taking the supremum over $j > 0$, we arrive at
		\begin{gather} \label{before_losing_eps}
			\sup_{j > 0}   2^{j \tilde{\kappa}_x}     \| \bar{\chi}_{j}  \|_{L^{p }_{\omega,t,x} }  \lesssim     \mathcal{K} + \| \chi   \|_{ L^{\beta}_{\omega,t,x,v}} .  
		\end{gather}
		Losing a small $ \epsilon $ in  $\tilde{\kappa}_x$, we have
		\begin{equation} \label{Losing_epsilon_space_regularity_interpolation_step}
			\begin{split}
		&	\left\Vert \sup_{j > 0} 2^{j( \tilde{\kappa}_x- \epsilon)}  \| \bar{\chi}_{j}  \|_{L^{p }_{t,x}} \right\Vert_{L^{p}_{\omega}} \le \left\Vert \sum_{j > 0}  2^{j( \tilde{\kappa}_x- \epsilon)}  \| \bar{\chi}_{j}  \|_{L^{p}_{t,x}} \right\Vert_{L^{p}_{\omega}}  \\ & \le \sum_{ j > 0}  2^{-j \epsilon }  2^{j \tilde{\kappa}_x}    \| \bar{\chi}_{j}  \|_{L^{p }_{\omega,t,x}}  \lesssim \sup_{j > 0}  2^{j \tilde{\kappa}_x}   \| \bar{\chi}_{j}  \|_{L^{p }_{\omega,t,x}},
			\end{split}
		\end{equation}
		where here $\lesssim$ denotes a bound that holds up to a multiplicative constant that only depends on $\epsilon$.
		
		\noindent For $j=0$, we can use Bernstein's lemma to get
		\begin{equation} \label{case_j0_space}
			 \|   \bar{\chi}_{j}   \|_{L^{p }_{\omega, t,x}} \lesssim  \|   \bar{\chi}_{j}   \|_{L^{1}_{\omega, t,x}} \lesssim \|   \bar{\chi}   \|_{L^{1}_{\omega, t,x}}. 
		\end{equation}
		 Let $ \kappa_x := \tilde{\kappa}_x- \epsilon$. From \eqref{before_losing_eps} and \eqref{case_j0_space}, we can arrive at \eqref{Bounds_Averaging_lemma_statement} as follows
		 
		\begin{equation} \label{final_removal_cutoff_space}
			\begin{split}
			\| \bar{\chi}   \|_{L^{p}_{\omega} \tilde{L}^{p}_t B_{p, \infty}^{\kappa_x}} & = \left\Vert \sup_{j \ge 0}  2^{j \kappa_x}  \| \bar{\chi}_{j}  \|_{L^{p }_{t,x}} \right\Vert_{L^{p}_{\omega}} \\ &  \lesssim     \left(  \sum_{k=1}^{\infty}     \|  \mathring{g}_k   \nabla  \tilde{u}^{\frac{m+1}{2}- \gamma^{\star} }   \|_{L^2_{\omega,t,x}}^2 \right)^{\frac{1}{2}} + \left(  \sum_{k=1}^{\infty}     \| \tilde{u}^{\frac{m+1}{2}- \gamma^{\star} }  \nabla \mathring{g}_k      \|_{L^2_{\omega,t,x}}^2 \right)^{\frac{1}{2}}  \\ &    +  \left( \sum_{k=1}^{\infty}  \esssup_{t \in [0,T]}    \left \|    \mathring{g}_k^2(t,\cdot,\tilde{u})  \right \|_{L^1_{\omega,x}}^{\frac{1}{\varpi}}   \left \|    \tilde{u}^{\varkappa(2 \nu - 1)}  \mathring{g}_k^2   \right \|_{L^1_{\omega,t,x}}^{\frac{1}{\varkappa}} \right)^{\frac{1}{2}} \\ &   +   \| \abs{v}^{1- \gamma}  h  \|_{L^1_{\omega} \mathcal{M}_{TV}} +  \| \abs{v}^{- \gamma } \tilde{h}  \|_{L^1_{\omega} \mathcal{M}_{\text{TV}}}     + \| \chi   \|_{ L^{\beta}_{\omega,t,x,v}} + \|   \bar{\chi}   \|_{L^{1}_{\omega, t,x}}.
			\end{split}  
		\end{equation} 
We remove the assumption that $ \chi$ is localized in $v$ similarly as done in \cite[Proof of Lemma 4.2]{gess2019optimal}. Let $ \Upsilon \in C^{\infty}_c ( \mathbb{R}_v)$ and  $\chi^{\Upsilon}(t,x,v) :=  \chi(t,x,v)  \Upsilon(v)$.
Then $ \chi^{\Upsilon}$ is a solution to
\begin{equation} \label{localization_v_space}
	\begin{split}
		\mathcal{L}(\partial_t, \nabla_x , v) \chi^{\Upsilon}(t,x,v)   & = \sum_{k=1}^{\infty}  \delta_{\tilde{u}(t,x)=v} \mathring{g}_k^{\Upsilon}(s,x,v)  \dot{\beta}_k \\ & + h^{\Upsilon}(t,x,v) - \tilde{h}^{\Upsilon^{'}}(t,x,v) + \partial_v \tilde{h}^{\Upsilon}(t,x,v),
	\end{split}
\end{equation}
where
\begin{equation*}
	\begin{split}
		 \mathring{g}_k^{\Upsilon}(t,x,v) & := \mathring{g}_k(t,x,v) \Upsilon(v), \quad \ \qquad \quad  h^{\Upsilon} (t,x,v)  := h(t,x,v) \Upsilon(v), \\  \tilde{h}^{\Upsilon^{'}}(t,x,v) & := \tilde{h}(t,x,v)  \Upsilon^{'}(v),  \qquad \qquad \  \tilde{h}^{\Upsilon}(t,x,v) := \tilde{h}(t,x,v) \Upsilon(v).
	\end{split}
\end{equation*}
Then the estimate \eqref{final_removal_cutoff_space} reads as follows:
\begin{equation} \label{Smooth_cutoff_time_reg_small_v_space}
	\begin{split}
		\| \bar{\chi}^{\Upsilon}   \|_{L^{p}_{\omega} \tilde{L}^{p}_t B_{p, \infty}^{\kappa_x}}  & \lesssim     \left(  \sum_{k=1}^{\infty}     \|  \mathring{g}_k^{\Upsilon} \nabla  \tilde{u}^{\frac{m+1}{2}- \gamma^{\star} }   \|_{L^2_{\omega,t,x}}^2 \right)^{\frac{1}{2}}   + \left(  \sum_{k=1}^{\infty}     \| \tilde{u}^{\frac{m+1}{2}- \gamma^{\star} }  \nabla \mathring{g}_k^{\Upsilon}	    \|_{L^2_{\omega,t,x}}^2 \right)^{\frac{1}{2}}    
\\		&    + \left( \sum_{k=1}^{\infty}    \esssup_{t \in [0,T]}   \left \|   (\mathring{g}^{\Upsilon}_k)^2(t, \cdot,\tilde{u})  \right \|_{L^1_{\omega,x}}^{\frac{1}{\varpi}}       \left \|   \tilde{u}^{ \varkappa(2 \nu - 1)}  (\mathring{g}^{\Upsilon}_k)^2   \right \|_{L^1_{\omega,t,x}}^{\frac{1}{\varkappa}}  \right)^{\frac{1}{2}}  \\ & + \| \abs{v}^{1- \gamma} ( h^{\Upsilon} - \tilde{h}^{\Upsilon^{'}}) \|_{L^1_{\omega} \mathcal{M}_{\text{TV}}}    + \| \abs{v}^{- \gamma} \tilde{h}^{\Upsilon} \|_{L^1_{\omega} \mathcal{M}_{\text{TV}}}  + \|  \chi^{\Upsilon}  \|_{ L^{\beta}_{\omega,t,x,v}} \\ & +  \|   \bar{\chi}^{\Upsilon}   \|_{L^{1}_{\omega, t,x}}.    
	\end{split}
\end{equation}
Since $ \abs{v}^{- \gamma} \tilde{h} \in L_{\omega}^1 \mathcal{M}_{\text{TV}}$ by assumption, there exists to $ \varepsilon_n \downarrow 0$ a sequence $ i_n \uparrow \infty $ such that
\begin{equation*}
	\mathbb{E} \left( \int_{t,x,v} \mathbbm{1}_{i_n \le \abs{v}}   \abs{v}^{- \gamma} \tilde{h}  \ dv \ dx \ dt \right)  \le \varepsilon_n, \quad \forall n \in \mathbb{N}.
\end{equation*}
For $ n \in \mathbb{Z}_{+}$ and a smooth cut-off function $ \Upsilon \in C_c^{\infty} (\mathbb{R})$ with $ \Upsilon=1$ on the ball $B_1(0)$ and $ \text{supp} \ \Upsilon \in B_2(0)$, we define $ \Upsilon_n$ via $ \Upsilon_n(v) := \Upsilon \left( \frac{v}{n} \right)$. Thus, $ \Upsilon^{'}_n$ is supported on $ i_n \le \abs{v} \le 2 i_n$ and takes values in $ \left[0, \frac{1}{i_n} \right]$, so that 
\begin{equation*}
	\begin{split}
		\| \abs{v}^{1- \gamma}  \tilde{h}^{\Upsilon^{'}_n} \|_{L^1_{\omega} \mathcal{M}_{\text{TV}}} & =  \mathbb{E} \left( \int_{t,x,v}  | \Upsilon^{'}_n(v) | |v | ( \abs{v}^{- \gamma} \tilde{h} ) \ dv \ dx \ dt \right)  \\  & =  \mathbb{E} \left( \int_{t,x,v} \mathbbm{1}_{ i_n \le \abs{v} \le 2 i_n}  | \Upsilon^{'}_n(v)| |v| ( \abs{v}^{- \gamma} \tilde{h} ) \ dv \ dx \ dt \right)  \\ & \lesssim \mathbb{E} \left( \int_{t,x,v} \mathbbm{1}_{ i_n \le \abs{v} \le 2 i_n}   \abs{v}^{- \gamma} \tilde{h}   \ dv \ dx \ dt \right)  \le \varepsilon_n.
	\end{split}
\end{equation*}
With the same choices of $ \Upsilon$, $\Upsilon_n$ and  $\Upsilon_n^{'}$, we notice 
\begin{equation*}
	\begin{split}
	\left(  \sum_{k=1}^{\infty}     \|  \mathring{g}_k^{\Upsilon_n} \nabla  \tilde{u}^{\frac{m+1}{2}- \gamma^{\star} }   \|_{L^2_{\omega,t,x}}^2 \right)^{\frac{1}{2}}    & \le  \left(  \sum_{k=1}^{\infty}  \mathbb{E} \int_{t,x} \mathbbm{1}_{i_n \le \abs{\tilde{u}} } |  \mathring{g}_k  \Upsilon_n(\tilde{u}) \nabla \tilde{u}^{\frac{m+1}{2}- \gamma^{\star}} |^2    dx \  dt   \right)^{\frac{1}{2}} \\ & \lesssim \left(  \sum_{k=1}^{\infty}  \mathbb{E} \int_{t,x} \mathbbm{1}_{i_n \le \abs{\tilde{u}} } |  \mathring{g}_k   \nabla \tilde{u}^{\frac{m+1}{2}- \gamma^{\star}} |^2    dx \  dt   \right)^{\frac{1}{2}},
\end{split}
\end{equation*}	
\begin{equation*}
	\begin{split}
 \left(  \sum_{k=1}^{\infty}     \|  \tilde{u}^{\frac{m+1}{2}- \gamma^{\star} } \nabla \mathring{g}_k^{\Upsilon_n}      \|_{L^2_{\omega,t,x}}^2 \right)^{\frac{1}{2}}    & \le  \left( \sum_{k=1}^{\infty}   \mathbb{E} \int_{t,x} \mathbbm{1}_{i_n \le \abs{\tilde{u}} } | \tilde{u}^{\frac{m+1}{2}- \gamma^{\star}}  \nabla \mathring{g}_k  \Upsilon_n(\tilde{u})  |^2    dx \  dt   \right)^{\frac{1}{2}} \\ & + \left(  \sum_{k=1}^{\infty}  \mathbb{E} \int_{t,x} \mathbbm{1}_{i_n \le \abs{
\tilde{u}} \le 2 i_n }    \abs{ \frac{  \tilde{u}^{\frac{m+1}{2}- \gamma^{\star}}}{n} \mathring{g}_k   \Upsilon^{'}_n(\tilde{u}) \nabla  \tilde{u}}^2  dx \  dt   \right)^{\frac{1}{2}}	\\ & \lesssim   \left( \sum_{k=1}^{\infty}  \mathbb{E} \int_{t,x} \mathbbm{1}_{i_n \le \abs{\tilde{u}} } | \tilde{u}^{\frac{m+1}{2}- \gamma^{\star}}  \nabla \mathring{g}_k    |^2    dx \  dt   \right)^{\frac{1}{2}}	\\ & +  \left(  \sum_{k=1}^{\infty}  \mathbb{E} \int_{t,x} \mathbbm{1}_{i_n \le \abs{\tilde{u}} \le 2 i_n }    | \mathring{g}_k     \nabla \tilde{u}^{\frac{m+1}{2}- \gamma^{\star}}   |^2   dx \  dt   \right)^{\frac{1}{2}},
\end{split}
\end{equation*}
and 
		\begin{equation*}
	\begin{split}
		& \Bigg[ \sum_{k=1}^{\infty}  \esssup_{t \in [0,T]}    \left( \mathbb{E} \int_x  \abs{\mathring{g}^{\Upsilon_n}_k(t,x,\tilde{u}(t,x))}^2   dx \right)^{\frac{1}{\varpi}}  \\ & \times \left(   \mathbb{E} \int_{t,x}  |\tilde{u}(t,x)|^{ \varkappa(2 \nu - 1)} | \mathring{g}^{\Upsilon_n}_k(t,x,\tilde{u}(t,x))     |^2  dx \ dt \right)^{\frac{1}{\varkappa}} \Bigg]^{\frac{1}{2}}  
			\end{split}
	\end{equation*}
		\begin{equation*}
	\begin{split}
		 & \lesssim \Bigg[ \sum_{k=1}^{\infty}   \esssup_{t \in [0,T]} \left(  \mathbb{E}  \int_x   \mathbbm{1}_{  i_n \le  \abs{\tilde{u}} }  \abs{\mathring{g}_k(t,x,\tilde{u}(t,x))}^2  dx \right)^{\frac{1}{\varpi}}   \\ & \times \left( \mathbb{E} \int_{t,x}  \mathbbm{1}_{  i_n \le  \abs{\tilde{u}} }  |\tilde{u}(t,x)|^{ \varkappa(2 \nu - 1)}  | \mathring{g}_k(t,x,\tilde{u}(t,x))     |^2   dx \  dt  \right)^{\frac{1}{\varkappa}} \Bigg]^{\frac{1}{2}}.
	\end{split}
\end{equation*}
With these choices of  $\Upsilon$, $\Upsilon_n$ and  $\Upsilon_n^{'}$ in \eqref{Smooth_cutoff_time_reg_small_v_space}, we may take the limit $ n \rightarrow \infty$ and use Fatou's lemma to obtain \eqref{final_removal_cutoff_space} also for general $ \chi$. 
	\end{proof}
 The following averaging lemma is used to prove Theorem \ref{Main_theorem_time}.  
	\begin{lem} \label{Time_Averaging_Lemma}
		Let $m$, $ \beta$, $\rho$,  $ \epsilon$, $\chi$, $\bar{\chi} $ be as in Lemma \ref{Isotropic_Averaging_Lemma} and assume $\zeta \in (0, \frac{1}{2})$. Let $ \varpi >1 $ and $ \varkappa^{-1} =1 - \varpi^{-1}$ such that $\varpi^{-1}> 2 \zeta$ . Assume that $\mathring{g}^2_k(\tilde{u})  \in L^{\infty}_t L^{1}_{\omega,x}$. 
		\begin{enumerate} [label=(\roman*)]
\item Let $ \gamma \in (0,1)$ such that
\begin{equation} \label{p_bar_q_statement_time_averaging_lemma}
	\begin{split}
	 \bar{p}  :=  \frac{1 - \gamma + \rho}{\rho + (1-\rho) (1 - \gamma  ) } < q  := \frac{2(1- \gamma + \rho)}{1- \gamma + \rho+ (1-\gamma - \epsilon(1-\gamma+\rho))(1-2\rho)},
\end{split}
\end{equation}
and let 
 \begin{equation} \label{regularity_results_before_bootstrap_time_integrability}
	p  := \frac{1 - \gamma + \rho}{1 - \gamma + \rho- ( 1 - \gamma - 2 \epsilon(1-\gamma + \rho)) \rho}.
\end{equation}
Suppose that  $ \tilde{u}^{1- \gamma} \mathring{g}_k(\tilde{u}) \in L^{2}_{\omega,t,x}, \tilde{u}^{2 \varkappa(1- \gamma)}  \mathring{g}_k^2(\tilde{u})  \in L^1_{\omega,t,x}$ and  $h, \tilde{h}$ satisfy the condition \eqref{condition_singular_moments_statement} for this value of $ \gamma$.
If $\bar{\chi}  \in L^{ \bar{p}}_{\omega} L^{q}_t L^p_x$, then $ \bar{\chi}   \in   L^{\bar{p}}_{\omega} B^{\kappa_t}_{q, \infty}  L^{p}_x   $ where
\begin{equation} \label{regularity_results_before_bootstrap_time}
\begin{split}
\kappa_t  & := \frac{ \zeta \rho }{1 - \gamma + \rho} - \epsilon,
\end{split}
\end{equation}
with estimate 
			\begin{equation} \label{Bounds_Averaging_lemma_statement_time}
				\begin{split}
\| \bar{\chi}  \|_{ L^{\bar{p}}_{\omega} B^{\kappa_t}_{q, \infty}  L^{p}_x   } & \lesssim     \| \chi  \|_{ L^{\beta}_{\omega,t,x,v}} +  \left( \sum_{k=1}^{\infty} \|      \tilde{u}^{1- \gamma}   \mathring{g}_k   \|_{L^2_{\omega,t,x}}^2 \right)^{\frac{1}{2}}   \\ & + \left( \sum_{k=1}^{\infty}  \esssup_{t \in [0,T]}    \left \|    \mathring{g}_k^2(t,\cdot,\tilde{u})  \right \|_{L^1_{\omega,x}}^{\frac{1}{\varpi}}   \left \|    \tilde{u}^{2 \varkappa(1 -\gamma)}  \mathring{g}_k^2   \right \|_{L^1_{\omega,t,x}}^{\frac{1}{\varkappa}} \right)^{\frac{1}{2}} 	
\\ & + \| \abs{v}^{1- \gamma}  h \|_{L^1_{\omega} \mathcal{M}_{\text{TV}}}   + \| \abs{v}^{- \gamma} \tilde{h}  \|_{L^1_{\omega} \mathcal{M}_{\text{TV}}}   +  \| \bar{\chi} \|_{L^{ \bar{p}}_{\omega} L^{q}_t L^p_x}. 
				\end{split}
			\end{equation}
					   Here and in the proof  $  \lesssim $ denotes a bound that holds up to a multiplicative constant that only depends on $ m, \rho, \epsilon, \zeta, \varpi, \gamma $ and $T$.
	\item Let  $ \nu \in ( \frac{1}{2}, \infty)$ and let $\bar{\epsilon} \in (0,1)$ such that $\tilde{\gamma} \in (1, 1+ \bar{\epsilon}(m-1))$ and $\zeta \in (0, \frac{1}{2} - \bar{\epsilon})$.  Suppose that  $ \tilde{u}^{\nu - \frac{1}{2}} \mathring{g}_k(\tilde{u}) \in L^{2}_{\omega, t,x}$, $ \tilde{u}^{\varkappa(2 \nu- 1)} \mathring{g}_k^2(\tilde{u})  \in L^1_{\omega,t,x}$ and $h, \tilde{h}$ satisfy the condition
	\begin{equation*} 
		\abs{h}(\omega,t,x,v) + | \tilde{h} | (\omega,t,x,v) \abs{v}^{- \tilde{\gamma}} \in L^1_{\omega} \Ma_{\text{TV}}( \mathbb{R}_t \times \mathbb{T}_x^d \times \mathbb{R}_v).
	\end{equation*}
 If $ \chi$ is supported in $v \in \mathbb{R} \setminus (-1,1) $ and  $\bar{\chi}  \in  L^{1}_{\omega} L^2_t L^{1}_x $, then $  \bar{\chi}   \in L^1_{\omega} B^{\zeta - \epsilon}_{2, \infty} L^{1}_x   $ with estimate
		\begin{equation} \label{final_cut_off_time_large_statement}
	\begin{split}
		\left \|  \bar{\chi}  \right \|_{L^1_{\omega} B^{\zeta - \epsilon}_{2, \infty} L^{1}_x  }  & \lesssim       \left( \sum_{k=1}^{\infty} \|    \tilde{u}^{ \nu -\frac{1}{2}}   \mathring{g}_k    \|_{L^2_{\omega,t,x}}^2 \right)^{\frac{1}{2}}    +     \Bigg( \sum_{k=1}^{\infty}   \esssup_{t \in [0,T]}    \left \|    \mathring{g}_k^2(t,\cdot,\tilde{u})  \right \|_{L^1_{\omega,x}}^{\frac{1}{\varpi}}    \\ & \times   \left \|   \tilde{u}^{\varkappa(2 \nu - 1)} \mathring{g}_k^2    \right \|_{L^1_{\omega,t,x}}^{\frac{1}{\varkappa}} \Bigg)^{\frac{1}{2}}    \ + \|    h  \|_{L^1_{\omega} \mathcal{M}_{\text{TV}}}  + \| \abs{v}^{-  \tilde{\gamma}} \tilde{h}  \|_{L^1_{\omega} \mathcal{M}_{\text{TV}}}   \\ & +  \| \bar{\chi} \|_{L^{ 1}_{\omega} L^{2}_t L^1_x}. 
	\end{split}  
\end{equation}
	   Here and in the proof  $ \ \lesssim $ denotes a bound that holds up to a multiplicative constant that only depends on $ m, \rho, \epsilon, \zeta, \varpi, \nu, \tilde{\gamma}, \bar{\epsilon}  $ and $T$.
	\end{enumerate}  
	\end{lem}
	
	\begin{oss}
		Note that if $ \rho,  \gamma$ are chosen close to one, $\zeta$ close to one half and $  \epsilon$ small enough, the orders of the differentiability and integrability exponents in Lemma  \ref{Time_Averaging_Lemma} correspond  to $\sigma_t$ close to $\frac{1}{2}$, the time integrability equal to $2$, the stochastic and space integrability equal to $1$ as in Theorem \ref{Main_theorem_time}.
	\end{oss}
	\begin{proof}
We introduce two main modifications from the proof of Lemma \ref{Isotropic_Averaging_Lemma}: the localization of $\chi$ in Fourier space connected to the time variable $t$ and a different microlocal decomposion of $ \chi$ based on the size of $v$ only. 

Let $\mathcal{L}(\partial_t, \nabla_x, v )$ and $\mathcal{L}(i \tau,  \xi, v )$ be as in the proof of Lemma \ref{Isotropic_Averaging_Lemma}.  Recall that
\begin{equation} \label{definition_av_time_regularity}
	a(v)= m \abs{v}^{m-1},
\end{equation} 
and the periodic heat kernel
\begin{equation*}
	\Phi(t,x) = \begin{cases}\frac{1}{(4 \pi t)^{\frac{d}{2}}} \sum_{n \in \mathbb{Z}^d}  \exp \left(- \frac{\abs{x+n}^2}{4t}\right) & t >0,
		\\ 0  & t < 0. 
	\end{cases}
\end{equation*}
 We prove  Parts \textit{(i)} and \textit{(ii)} of the statement separately.

\paragraph{Part \textit{(i)}} 
As in the proof of Lemma \ref{Isotropic_Averaging_Lemma}, we first assume that $ \chi$ is compactly supported with respect to $v$ and then remove this qualitative assumption at the end of the proof.
 
 Let $\psi_0$ be a smooth function with compact support in the ball $B_2(0)$ such that $\psi_0=1$ in  $B_1(0)$ and $\psi_1:=1- \psi_0$. For $ \delta >0$ to be specified later, we write
\begin{equation*}
	\chi=  \psi_0 \left( \frac{ \abs{v} }{\delta} \right)  \chi + \psi_1 \left(\frac{ \abs{v}}{\delta} \right)  \chi =: \chi^{0}+  \chi^{1}.
\end{equation*}
Then $\chi^1$ solves, in the sense of \eqref{notion_solution_kinetic_torus},
\begin{equation*}
	\begin{split}
		\mathcal{L}(\partial_t, \nabla_x,v) \chi^{1} & =  \psi_1 \left(\frac{ \abs{v} }{\delta} \right)  \left(\sum_{k=1}^{\infty}  \delta_{\tilde{u}=v }   \mathring{g}_k \dot{\beta}_k + h +  \partial_v \tilde{h} \right).
	\end{split}
\end{equation*}
Thus,
\begin{equation} \label{non_degenerate_decomposition_starting}
	\begin{split}
		\chi^{1}(t,x,v) & =  \mathcal{F}_{t,x}^{-1} \psi_1 \left(\frac{ \abs{v} }{\delta} \right) \frac{1}{ \mathcal{L}(i\tau, \xi,v)} \mathcal{F}_{t,x}  \left(\sum_{k=1}^{\infty}  \delta_{\tilde{u}=v }   \mathring{g}_k \dot{\beta}_k \right)
		\\ & + \mathcal{F}_{t,x}^{-1} \psi_1 \left(\frac{ \abs{v} }{\delta} \right) \frac{1}{ \mathcal{L}(i\tau, \xi,v)} \mathcal{F}_{t,x}  h(t,x,v)  \\   & + \mathcal{F}_{t,x}^{-1} \psi_1 \left(\frac{ \abs{v} }{\delta} \right) \frac{1}{ \mathcal{L}(i\tau, \xi,v)} \mathcal{F}_{t,x}  \partial_v \tilde{h}(t,x,v) \\  & =: \chi^{1,1}(t,x,v)  + \chi^{1,2}(t,x,v)  + \chi^{1,3}(t,x,v).
	\end{split}
\end{equation}
Let $ \{ \hat{\eta}_l \}_{l \ge 0}$ be defined as in Section \ref{Notation}. 
For $l \ge 0$, we define the Littlewood--Paley block of $\chi^0$ as follows
\begin{equation*}
	\chi_{l}^0(t,x,v) = \mathcal{F}_{t}^{-1} [ \hat{\eta}_{l} \mathcal{F}_{t} \chi^0(t,x,v) ], 
\end{equation*}
where $ \mathcal{F}_{t} \chi_{l}^0 (\tau,x,v)$ is supported on $ \abs{\tau} \sim 2^l$ for $l > 0$ and on $ \abs{\tau} \lesssim 1$ for $l=0$. In the calculations below, we assume that  $l>0$ unless stated otherwise.

The block for $ \chi^{1} $ is defined as follows
\begin{equation} \label{Littlewood_Paley_block_nondegenerate_stochastic_cutoff_t}
	\begin{split}
\chi^1_l(t,x,v)   = \int_r \eta_{l}(r) \chi^{1}(t-r,x,v) \ dr .  
	\end{split}
\end{equation}
The blocks for $ \chi^{1,1}, \chi^{1,2}$ and $\chi^{1,3}$ are defined using \eqref{non_degenerate_decomposition_starting} as follows
\begin{equation*} 
	 \begin{split}
	\chi^{1,1}_l (t,x,v)& = \int_r \eta_{l}(r) \chi^{1,1}(t-r,x,v)  \ dr,
\\  	\chi^{1,2}_l(t,x,v) & =  \mathcal{F}_{t,x}^{-1} \hat{\eta}_{l}   \psi_1 \left(\frac{ \abs{v} }{\delta} \right) \frac{ 1 }{ \mathcal{L}(i\tau, \xi,v) }  \mathcal{F}_{t,x}  h(t,x,v),   \\   \chi^{1,3}_l(t,x,v)  & = \mathcal{F}_{t,x}^{-1}  \hat{\eta}_{l}  \psi_1 \left(\frac{ \abs{v} }{\delta} \right) \frac{ 1 }{ \mathcal{L}(i\tau, \xi,v) }  \mathcal{F}_{t,x} \partial_v  \tilde{h}(t,x,v),
 \end{split}
\end{equation*}
respectively.

We can decompose \eqref{Littlewood_Paley_block_nondegenerate_stochastic_cutoff_t} as follows
\begin{equation*}
	\begin{split}
\chi^1_l(t,x,v) 
 & = \int_r \eta_{l}(r)  (\chi^{1,1}(t-r,x,v) - \chi^{1,1}(t,x,v))  \ dr \\ & +  \chi^{1,2}_l(t,x,v)  + \chi^{1,3}_l(t,x,v). 
\end{split}
\end{equation*}

We have obtained the following decomposition
\begin{equation} \label{decomposition_time_v_small}
	\begin{split}
		\bar{\chi}_l(t,x)  & = \int_v \chi_l (t,x,v) \ dv
	\\ & 	= \int_v \chi^0_l (t,x,v) \ dv  + 	\int_v \chi^1_l(t,x,v) \ dv   \\ & =  \int_{v} \chi^0_l(t,x,v)  \ dv +  \int_{v,r} \eta_{l}(r) (\chi^{1,1}(t-r,x,v)  - \chi^{1,1}(t,x,v))   \ dr \ dv \\ & +  \int_{v} \chi^{1,2}_l(t,x,v)  \ dv  +  \int_{v} \chi^{1,3}_l(t,x,v)  \ dv.
	\end{split}
\end{equation}

We proceed in dividing this part of the proof into five steps. In the first four we estimate the velocity averages \eqref{decomposition_time_v_small} and in the final step we combine all the estimates.  The main novelty relative to \cite[Proof of Lemma 4.4]{gess2019optimal} is the estimate of the term containing the stochastic integral in Step 2(i) below which is based on the integral representation of the stochastic term as outlined in the Introduction. Steps 1(i), 3(i) and 4(i) below are identical to \cite[Proof of Lemma 4.4 Steps 1, 2 and 3]{gess2019optimal}, respectively with the difference that here the estimates are also $\omega$-dependent and, the estimates of Steps 3(i) and 4(i) are evaluated in $L^2$ in time.  
\subparagraph{Step 1(i)} Let $l > 0 $ be arbitrary and fixed.  We follow the same arguments as in  \cite [Proof of Lemma 4.4 Step 1]{gess2019optimal}. Thus, we note that the estimate 
\begin{equation*}
	 \| \mathcal{F}_t^{-1} \hat{\eta}_{l} \mathcal{F}_t \chi \|_{L^{\beta}_{\omega,t,x}} \lesssim \| \chi   \|_{L^{\beta}_{\omega,t,x}},
\end{equation*}
holds for a constant independent of $l$. Since $ \abs{v} \le \delta $ on the support of $  \psi_0 \left( \frac{ \abs{v} }{\delta} \right)$, we may use Minkowski's and H{\"o}lder's inequality to get
\begin{equation} \label{velocity_degenerate_time}
	\begin{split}
		\left\Vert \int_v \chi^0_{l} \ dv  \right\Vert_{L^{\beta}_{\omega,t,x}}  & = \left\Vert \int_{v}    \psi_0 \left(\frac{ \abs{v} }{\delta} \right)  \chi_{l}  \ dv  \right\Vert_{L^{\beta}_{\omega,t,x}}   \\
		& \le   \int_{v}   \abs{ \psi_0} \left(\frac{\abs{v} }{\delta} \right) \| \chi_{l} \|_{L^{\beta}_{\omega, t,x}} dv  
		\\		&  \lesssim     \| \chi    \|_{L^{\beta}_{\omega,t,x,v}} \left(  \int_{v}   \abs{ \psi_0} \left(\frac{ \abs{v} }{\delta} \right)^{\frac{1}{\rho}}  dv  \right)^{\rho}  
		\\ & 	\lesssim \delta^{\rho}   \| \chi   \|_{ L^{\beta}_{\omega,t,x,v}}. 
	\end{split}
\end{equation}
\subparagraph{Step 2(i)} 
Let $l> 0 $ be arbitrary and fixed. We estimate the second term on the right-hand side of \eqref{decomposition_time_v_small}. Observe that $\chi^{1,1}$ in \eqref{non_degenerate_decomposition_starting} is given by
\begin{equation} \label{non_degenerate_decomposition_starting_step2_time}
	\begin{split}
 \chi^{1,1}(t,x,v)= \sum_{k=1}^{\infty} \int_0^t \int_{\mathbb{T}^d} \Phi(a(v)(t-s),x-y) \psi_1 \left( \frac{\abs{v}}{\delta} \right)  \delta_{\tilde{u}(s,y)=v }   \mathring{g}_k(s,y,v) \ dy \  d\beta_k(s).
	\end{split}
\end{equation}
As in the proof of Lemma \ref{Isotropic_Averaging_Lemma}, we denote the first derivative with respect to the first argument of $ \Phi$ by $ \Phi_1$ and let
\begin{equation*}
	\begin{split}
		\tilde{\Phi} (t,x)  &  := \Phi_1 (t,x) t.
	\end{split}
\end{equation*}
Note that 
\begin{equation}  \label{bound_Phi_tilde_norm_1}
	\begin{split}
		\|  \tilde{\Phi} ( t,\cdot) \|_{L^1_x}  \lesssim 1,
	\end{split}
\end{equation} 
and
\begin{equation}  \label{bound_Phi_tilde_derivative_norm_1}
	\begin{split}
		\|  \partial_t \tilde{\Phi} ( t,\cdot) \|_{L^1_x}  \lesssim t^{-1}.
	\end{split}
\end{equation} 
For fixed $t,r$, let
\begin{equation*}
	\begin{split}
	\Phi_{\text{inc}}(a(v)s,x) & := \Phi(a(v)(t-r-s),x) - \Phi(a(v)(t-s),x),
\\  \tilde{\Phi}_{\text{inc}}(a(v)s,x) & := \tilde{\Phi}(a(v)(t-r-s),x) - \tilde{\Phi}(a(v)(t-s),x).
	\end{split}
\end{equation*} 
Using \eqref{non_degenerate_decomposition_starting_step2_time}, we have
\begin{equation} \label{First_decomposition_time_regularity_estimate}
	\begin{split}
		& \int_v \chi^{1,1}(t-r,x,v) - \chi^{1,1}(t,x,v) \ dv
		\\ & = \int_v  \sum_{k=1}^{\infty} \int_0^{t-r} \int_y 	\Phi_{\text{inc}}(a(v)s,x-y)  \delta_{\tilde{u}(s,y)=v} \psi_1 \left(\frac{ \abs{v}}{\delta} \right) \mathring{g}_k(s,y,v)    \ dy \ d\beta_k(s) \ dv
\\	& - \int_v \sum_{k=1}^{\infty} \int_{t-r}^t \int_y  \Phi(a(v)(t-s),x-y)  \delta_{\tilde{u}(s,y)=v}  \psi_1 \left(\frac{ \abs{v}}{\delta} \right) \mathring{g}_k(s,y,v)    \ dy \ d\beta_k(s) \ dv
\\ & =  \sum_{k=1}^{\infty} \int_0^{t-r} \int_{y,v}   \mathbbm{1}_{[0,\tilde{u}(s,y)]}(v) \  \partial_v \left( 	\Phi_{\text{inc}}(a(v)s,x-y)  \psi_1 \left(\frac{ \abs{v}}{\delta} \right) \right)       \ dv  \\ & \times  \mathring{g}_k(s,y,\tilde{u}(s,y))    \ dy \ d\beta_k(s) 
		\\ & -  \sum_{k=1}^{\infty} \int_{t-r}^t \int_{y,v}   \mathbbm{1}_{[0,\tilde{u}(s,y)]}(v) \  \partial_v \left( \Phi(a(v)(t-s),x-y)   \psi_1 \left(\frac{ \abs{v}}{\delta} \right) \right)   \ dv  \\ & \times  \mathring{g}_k(s,y,\tilde{u}(s,y))    \ dy \ d\beta_k(s)
			\\ &  =I + II - III - IV,
	\end{split}
\end{equation}
where
\begin{equation*}
		\begin{split}
I  := &	 \sum_{k=1}^{\infty} \int_0^{t-r} \int_{v,y}  \abs{v}^{-\nu} \mathbbm{1}_{[0,\tilde{u}(s,y)]}(v)  \	\tilde{\Phi}_{\text{inc}}(a(v)s,x-y)  \frac{a^{'}(v)}{a(v)} \abs{v}^{ \nu}   \psi_1 \left(\frac{ \abs{v}}{\delta} \right)    \\ & \times  \mathring{g}_k(s,y,\tilde{u}(s,y))    \ dy  \ dv \  d\beta_k(s),
\\  II  :=  & \sum_{k=1}^{\infty}  \int_0^{t-r} \int_{v,y} \abs{v}^{- \nu} \mathbbm{1}_{[0,\tilde{u}(s,y)]}(v) \	\Phi_{\text{inc}}(a(v)s,x-y) \abs{v}^{-1+ \nu} \psi_1^{'} \left(\frac{ \abs{v}}{\delta} \right)    \frac{\abs{v} \sgn(v)}{\delta}   \\ & \times   \mathring{g}_k(s,y,\tilde{u}(s,y))    \ dy  \ dv  \ d\beta_k(s),
\\ III  := & \sum_{k=1}^{\infty}  \int_{t-r}^t \int_{v,y} \abs{v}^{- \nu} \mathbbm{1}_{[0,\tilde{u}(s,y)]}(v)  \ \tilde{\Phi}(a(v)(t-s),x-y)  \frac{a^{'}(v)}{a(v)} \abs{v}^{\nu} \psi_1 \left(\frac{ \abs{v}}{\delta} \right)      \\ & \times  \mathring{g}_k(s,y,\tilde{u}(s,y))    \ dy  \ dv  \ d\beta_k(s),
	\end{split}
\end{equation*} 
\begin{equation*}
\begin{split}
 IV  := & \sum_{k=1}^{\infty} \int_{t-r}^t \int_{v,y} \abs{v}^{- \nu} \mathbbm{1}_{[0,\tilde{u}(s,y)]}(v) \ \Phi(a(v)(t-s),x-y) \abs{v}^{-1+ \nu} \psi_1^{'} \left(\frac{ \abs{v}}{\delta} \right) \frac{\abs{v} \sgn(v)}{\delta}   \\ & \times  \mathring{g}_k(s,y,\tilde{u}(s,y))    \ dy  \ dv   \ d\beta_k(s) .
	\end{split}
\end{equation*} 
Using \eqref{definition_av_time_regularity}, we have
\begin{equation} \label{bound_a_prime_v_time}
		\begin{split}
 \frac{a^{'}(v)}{a(v)}  \lesssim \abs{v}^{-1}.
	\end{split}
\end{equation}
By the definition of $ \tilde{\Phi}_{\text{inc}}$ and \eqref{bound_Phi_tilde_norm_1}, we have
\begin{equation} \label{bound_2_phi_inc}
\|  \tilde{\Phi}_{\text{inc}}(a(v)s,\cdot)    \|_{L^1_x}^2  \le 2( \|  \tilde{\Phi}(a(v)(t-r-s),\cdot)  \|_{L^1_x}^2  + \|   \tilde{\Phi}(a(v)(t-s),\cdot)  \|_{L^1_x}^2)  \lesssim 1.
\end{equation}
Using \eqref{bound_Phi_tilde_derivative_norm_1}, we have for $s \le t-r$
\begin{equation} \label{L1_estimate_inc_time}
	\begin{split}
\|  \tilde{\Phi}_{\text{inc}}(a(v)s,\cdot)    \|_{L^1_x}^2 & \lesssim   \left( \int_{t-r}^{t} \left \| \partial_{\tau}  \tilde{\Phi}(a(v)(\tau-s),\cdot) \right \|_{L^1_x}   d\tau \right)^2 \\ &  \lesssim \left( \int_{t-r}^{t} (\tau -s)^{-1} \ d\tau \right)^2 \le r^2 (t-r-s)^{-2}.
\end{split}
\end{equation}
Recall that $\zeta \in (0,\frac{1}{2})$. We deduce using \eqref{bound_2_phi_inc} and \eqref{L1_estimate_inc_time} that
\begin{equation} \label{L1_estimate_inc_time_half}
	\begin{split}
		\|  \tilde{\Phi}_{\text{inc}}(a(v)s,\cdot)    \|_{L^1_x}^2 & \lesssim   r^{ 2 \zeta} (t-r-s)^{- 2 \zeta}.
	\end{split}
\end{equation}
Let $ \nu \in ( \frac{1}{2}, \infty)$. For fixed $t,r$, we estimate the term $I$ in \eqref{First_decomposition_time_regularity_estimate} using Cauchy--Schwarz inequality, It{\^o}'s isometry, estimate \eqref{bound_a_prime_v_time}, Young's convolution inequality and estimate \eqref{L1_estimate_inc_time_half}  
\begin{equation*}  
		\begin{split}
& \mathbb{E}	\left \|  	I     \right \|_{L^{1}_{x}}^2  =  \mathbb{E} \Bigg( \int_{x} \Bigg | \sum_{k=1}^{\infty} \int_0^{t-r} \int_{v,y} \abs{v}^{- \nu} \mathbbm{1}_{[0,\tilde{u}(s,y)]}(v)  \  \tilde{\Phi}_{\text{inc}}(a(v)s,x-y) \frac{a^{'}(v)}{a(v)} \abs{v}^{\nu} \psi_1 \left(\frac{ \abs{v}}{\delta} \right) 
\\     & \times    \  \mathring{g}_k(s,y,\tilde{u}(s,y))    \ dy \ dv  \ d\beta_k(s)  \Bigg |  dx \Bigg)^2
	\\ & \le   \int_{x}   \mathbb{E} \Bigg | \sum_{k=1}^{\infty} \int_0^{t-r}  \int_{v,y} \abs{v}^{- \nu} \mathbbm{1}_{[0,\tilde{u}(s,y)]}(v)  \  \tilde{\Phi}_{\text{inc}}(a(v)s,x-y) \frac{a^{'}(v)}{a(v)}  \abs{v}^{\nu} \psi_1 \left(\frac{ \abs{v}}{\delta} \right)   \\ & \times   \mathring{g}_k(s,y,\tilde{u}(s,y))    \ dy \ dv \  d\beta_k(s)  \Bigg |^2   dx 
	\\		& =   \int_{x}  \mathbb{E} \sum_{k=1}^{\infty}  \int_0^{t-r} \Bigg( \int_v \abs{v}^{- \nu} \psi_1 \left(\frac{ \abs{v}}{\delta} \right)  \int_{y}  \mathbbm{1}_{[0,\tilde{u}(s,y)]}(v)  \  \tilde{\Phi}_{\text{inc}}(a(v)s,x-y) \frac{a^{'}(v)}{a(v)} \abs{v}^{\nu}  \\ & \times     \mathring{g}_k(s,y,\tilde{u}(s,y))    \ dy \ dv    \Bigg)^2 ds \   dx
				\end{split}
\end{equation*}	
\begin{equation} \label{Estimate_I_time_regularity_decomposition}
	\begin{split} 
	&  \lesssim  \mathbb{E}   \sum_{k=1}^{\infty}  \int_0^{t-r}  \int_{\abs{v} \ge \delta}  \abs{v}^{- 2 \nu}  dv   \int_v \abs{v}^{2 \nu -2} \\ & \times \int_{x}  \left |  \int_{y}  \tilde{\Phi}_{\text{inc}}(a(v)s,x-y)      \mathring{g}_k(s,y,\tilde{u}(s,y))  \mathbbm{1}_{[0,\tilde{u}(s,y)]}(v)   \ dy   \right |^2 dx \ dv \  ds
		 \\ &	\lesssim  \delta^{1 - 2 \nu }      \mathbb{E} \sum_{k=1}^{\infty}    \int_0^{t-r}   \int_v  \abs{v}^{2 \nu -2}      \left \|   \tilde{\Phi}_{\text{inc}}(a(v)s,\cdot)  \right \|^{2 }_{L^1_x}  \int_x |\mathring{g}_k(s,x,\tilde{u}(s,x))|^2  \mathbbm{1}_{[0,\tilde{u}(s,x)]}(v) \ dx  \     dv \  ds    
\\ &	\lesssim  \delta^{1 - 2 \nu } r^{2 \zeta}     \mathbb{E} \sum_{k=1}^{\infty} \int_x     \int_0^{t-r}  \left(   t-r -s \right)^{- 2 \zeta}     \abs{\tilde{u}(s,x)}^{2 \nu -1}     \abs{\mathring{g}_k(s,x,\tilde{u}(s,x))}^2     \  ds  \ dx.
				\end{split}
\end{equation}
The term $II$ in \eqref{First_decomposition_time_regularity_estimate} is estimated similarly since $\abs{v} \sim \delta$ on the support of $\psi_1^{'} \left( \frac{\abs{v}}{\delta} \right)$. For fixed $t,r$, we have
\begin{equation}  \label{Estimate_II_time_regularity_decomposition}
	\begin{split}
		& \mathbb{E}	\left \|  	II     \right \|_{L^{1}_{x}}^2 \\ &	\lesssim   	  \delta^{1 - 2 \nu } r^{ 2 \zeta}  \mathbb{E}  \sum_{k=1}^{\infty} \int_x     \int_0^{t-r}   \left(   t-r -s \right)^{- 2 \zeta}      \abs{\tilde{u}(s,x)}^{2 \nu -1}     \abs{\mathring{g}_k(s,x,\tilde{u}(s,x))}^2     \  ds  \ dx .
\end{split}
\end{equation}
Recall that $\varpi^{-1}+ \varkappa^{-1} =1 $. For fixed $t,r$, we estimate the term $III$ in \eqref{First_decomposition_time_regularity_estimate} using Cauchy--Schwarz inequality, It{\^o}'s isometry, estimate \eqref{bound_a_prime_v_time}, Young's convolution inequality, estimate \eqref{bound_Phi_tilde_norm_1} and H{\"o}lder's inequality  
\begin{equation*}   
	\begin{split}
& \mathbb{E}	\left \|  	III     \right \|_{L^{1}_{x}}^2
 = \mathbb{E} \Bigg( \int_{x}	\Bigg | 	\sum_{k=1}^{\infty}	 \int_{t-r}^t \int_{v,y}  \abs{v}^{- \nu} \mathbbm{1}_{[0,\tilde{u}(s,y)]}(v) \ \tilde{\Phi}(a(v)(t-s),x-y)  \frac{a^{'}(v)}{a(v)}   \abs{v}^{\nu}
\\  & \times   \psi_1 \left(\frac{ \abs{v}}{\delta} \right)    \mathring{g}_k(s,y,\tilde{u}(s,y))    \ dy \ dv  \ d\beta_k(s)  	\Bigg | \ dx \Bigg)^2 .
\\ & \le  \int_{x}   \mathbb{E} \Bigg | \sum_{k=1}^{\infty} \int_{t-r}^t \int_{v,y}  \abs{v}^{- \nu} \mathbbm{1}_{[0,\tilde{u}(s,y)]}(v) \ \tilde{\Phi}(a(v)(t-s),x-y)  \frac{a^{'}(v)}{a(v)} \abs{v}^{ \nu} \psi_1 \left(\frac{ \abs{v}}{\delta} \right)    \\ & \times  \mathring{g}_k(s,y,\tilde{u}(s,y))    \ dy  \ dv \ d\beta_k(s)     \Bigg |^2  dx
\\	& = \int_{x}  \mathbb{E} \sum_{k=1}^{\infty} \int_{t-r}^t \Bigg( \int_v \abs{v}^{- \nu} \psi_1 \left(\frac{ \abs{v}}{\delta} \right) \int_{y}   \mathbbm{1}_{[0,\tilde{u}(s,y)]}(v) \ \tilde{\Phi}(a(v)(t-s),x-y) \frac{a^{'}(v)}{a(v)}     \abs{v}^{\nu}
\\	& \times  \mathring{g}_k(s,y,\tilde{u}(s,y))    \ dy  \ dv   \Bigg )^2 ds 
\ dx 
\\ & \lesssim  \delta^{1 - 2 \nu }     \mathbb{E} \sum_{k=1}^{\infty}   \int_{t-r}^t \int_v \abs{v}^{ 2 \nu -2 } \int_x \Bigg |   \int_y \mathbbm{1}_{[0,\tilde{u}(s,y)]}(v) \ \tilde{\Phi}(a(v)(t-s),x-y)   
		\end{split}
\end{equation*}  
\begin{equation} \label{Estimate_III_time_regularity_decomposition}
	\begin{split} 
  & \times  \mathring{g}_k(s,y,\tilde{u}(s,y))    \ dy  \Bigg |^2  dx  \ dv  \  ds  
\\  & \lesssim  \delta^{1 - 2 \nu }     \sum_{k=1}^{\infty} \mathbb{E}  \int_{t-r}^t \int_v \abs{v}^{ 2 \nu -2 } \int_x  \mathbbm{1}_{[0,\tilde{u}(s,x)]}(v)   \abs{  \mathring{g}_k(s,x,\tilde{u}(s,x))     }^2 \ dx  \ dv  \  ds  
\\ & =  \delta^{1 - 2 \nu }   \sum_{k=1}^{\infty}   \int_{t-r}^t  \mathbb{E}  \int_x | \tilde{u}(s,x) |^{2 \nu - 1}     | \mathring{g}_k(s,x,\tilde{u}(s,x))     |^{2 \left(\frac{1}{\varpi}+\frac{1}{\varkappa}\right)} \  dx    \  ds
 \\ & \le \delta^{1 - 2 \nu }    \sum_{k=1}^{\infty}   \int_{t-r}^t  \mathbb{E} \left( \int_x | \mathring{g}_k(s,x,\tilde{u}(s,x))     |^{2 } dx \right)^{\frac{1}{\varpi}} 
 \\ & \times  \left( \int_x | \tilde{u}(s,x) |^{\varkappa(2 \nu - 1)}     | \mathring{g}_k(s,x,\tilde{u}(s,x))     |^{2} \  dx \right)^{\frac{1}{\varkappa}}    \  ds  
  \\   & \le \delta^{1 - 2 \nu }  r^{\frac{1}{ \varpi}}  \sum_{k=1}^{\infty}   \esssup_{s \in [t-r,t]}   \left \|  \mathring{g}^2 _k(s,\cdot,\tilde{u})  \right \|_{L^1_{\omega,x}}^{\frac{1}{\varpi}}  \left( \int_{t-r}^t  \left \|   \tilde{u}(s,\cdot)^{\varkappa(2 \nu - 1)}   \mathring{g}_k^2 (s,\cdot,\tilde{u})  \right \|_{L^1_{\omega, x} }      \   ds \right)^{\frac{1}{\varkappa}}  
    \\   & \le \delta^{1 - 2 \nu }  r^{\frac{1}{ \varpi}}  \sum_{k=1}^{\infty}   \esssup_{s \in [0,T]}   \left \|  \mathring{g}_k^2(s,\cdot,\tilde{u})   \right \|_{L^1_{\omega,x}}^{\frac{1}{\varpi}}   \left \|    \tilde{u}^{\varkappa(2 \nu - 1)}  \mathring{g}_k^2   \right \|_{L^1_{\omega,s,x}}^{\frac{1}{\varkappa}}  .
\end{split}
\end{equation}
The term $IV$  in \eqref{First_decomposition_time_regularity_estimate} is estimated similarly since $\abs{v} \sim \delta$ on the support of $\psi_1^{'} \left( \frac{\abs{v}}{\delta} \right)$. Thus, 
\begin{equation} \label{Estimate_IV_time_regularity_decomposition}
	\begin{split}
	\mathbb{E}	\left \|  	IV     \right \|_{L^{1}_{x}}^2 & 	   \lesssim \delta^{1 - 2 \nu }  r^{\frac{1}{ \varpi}}    \sum_{k=1}^{\infty}   \esssup_{s \in [0,T]}   \left \|  \mathring{g}_k^2(s,\cdot,\tilde{u})   \right \|_{L^1_{\omega,x}}^{\frac{1}{\varpi}}   \left \|    \tilde{u}^{\varkappa(2 \nu - 1)}  \mathring{g}_k^2   \right \|_{L^1_{\omega,s,x}}^{\frac{1}{\varkappa}}  .
	\end{split}
\end{equation}
Recall that $\varpi^{-1}> 2 \zeta$. The estimate for the second term on the right-hand side of \eqref{decomposition_time_v_small} follows using Minkowski's integral inequality and combining \eqref{First_decomposition_time_regularity_estimate} with bounds \eqref{Estimate_I_time_regularity_decomposition}, \eqref{Estimate_II_time_regularity_decomposition}, \eqref{Estimate_III_time_regularity_decomposition}, \eqref{Estimate_IV_time_regularity_decomposition} 
\begin{equation*} 
	\begin{split}
			& 	\left\Vert    \int_{v,r} \eta_{l}(r) (\chi^{1,1}(\cdot-r,\cdot,v)  - \chi^{1,1}(\cdot,\cdot,v))   \ dr  \  dv \right\Vert_{L^1(\Omega;  L^2(([-1,T+1]) ;  L^1(\mathbb{T}^d)))} 
	\\	& \lesssim	\left\Vert    \int_{v,r} \eta_{l}(r) (\chi^{1,1}(\cdot-r,\cdot,v)  - \chi^{1,1}(\cdot,\cdot,v))   \ dr  \  dv \right\Vert_{L^2(\Omega;  L^2(([-1,T+1]) ;  L^1(\mathbb{T}^d)))} 
	\\ 		& \le \left \{   	\int_{-1}^{T+1}  \left[   \int_r \abs{\eta_{l}(r)} \left( \mathbb{E} \Bigg \|    \int_v  (\chi^{1,1}(t-r,\cdot,v)  - \chi^{1,1}(t,\cdot,v))    \  dv \Bigg \|_{L^1_x}^2 \right)^{\frac{1}{2}} dr \right]^2      dt   	  \right \}^{\frac{1}{2}} 
	\\ & 	\lesssim \left \{   	\int_{-1}^{T+1}   \left[    \int_r \abs{\eta_{l}(r) } \left( \mathbb{E}	\left \|  	I     \right \|_{L^{1}_{x}}^2 + \mathbb{E}	\left \|  	II     \right \|_{L^{1}_{x}}^2 + \mathbb{E}	\left \|  	III     \right \|_{L^{1}_{x}}^2 + \mathbb{E}	\left \|  	IV     \right \|_{L^{1}_{x}}^2   \right)^{\frac{1}{2}} dr \right]^2      dt   	  \right \}^{\frac{1}{2}} 
\end{split}
\end{equation*}
\begin{equation} \label{Final_step1_time_reg_small_v}
	\begin{split}
	& \lesssim \delta^{\frac{1}{2} - \nu } \Bigg \{   	\int_{-1}^{T+1}   \Bigg [    \int_r \abs{\eta_{l}(r)  }  r^{  \zeta} \\ & \times \left( \mathbb{E}  \sum_{k=1}^{\infty} \int_x     \int_0^{t-r}   \left(   t-r -s \right)^{- 2 \zeta}      \abs{\tilde{u}(s,x)}^{2 \nu -1}     \abs{\mathring{g}_k(s,x,\tilde{u}(s,x))}^2     \  ds  \ dx \right)^{\frac{1}{2}} dr \Bigg]^2      dt   	  \Bigg \}^{\frac{1}{2}} 
	\\			& + \delta^{\frac{1}{2} - \nu } \left \{  	\int_{-1}^{T+1}   \left[    \int_r \abs{\eta_{l}(r) }  r^{\frac{1}{ 2 \varpi}} 
	\left( \sum_{k=1}^{\infty}   \esssup_{s \in [0,T]}   \left \|  \mathring{g}_k^2(s,\cdot,\tilde{u})   \right \|_{L^1_{\omega,x}}^{\frac{1}{\varpi}}   \left \|    \tilde{u}^{\varkappa(2 \nu - 1)}  \mathring{g}_k^2   \right \|_{L^1_{\omega,s,x}}^{\frac{1}{\varkappa}}  \right)^{\frac{1}{2}} dr \right]^2      dt   	  \right \}^{\frac{1}{2}}
	\\	 & \lesssim \delta^{\frac{1}{2} - \nu } \Bigg \{   	\int_{-1}^{T+1}  \Bigg[    \int_{\tau} \abs{\eta_{l}(t- \tau)}  (t- \tau)^{  \zeta} \\ & \times \left( \mathbb{E}  \sum_{k=1}^{\infty} \int_x     \int_0^{\tau}   \left(   \tau -s \right)^{- 2 \zeta}      \abs{\tilde{u}(s,x)}^{2 \nu -1}     \abs{\mathring{g}_k(s,x,\tilde{u}(s,x))}^2     \  ds  \ dx \right)^{\frac{1}{2}} d\tau \Bigg]^2      dt   	  \Bigg \}^{\frac{1}{2}} 	
	\\ & + \delta^{\frac{1}{2} - \nu }   	       \int_r \abs{\eta_{l}(r)  }  r^{\frac{1}{ 2 \varpi}} \ dr \left( \sum_{k=1}^{\infty}   \esssup_{s \in [0,T]}   \left \|  \mathring{g}_k^2(s,\cdot,\tilde{u})   \right \|_{L^1_{\omega,x}}^{\frac{1}{\varpi}}   \left \|    \tilde{u}^{\varkappa(2 \nu - 1)}  \mathring{g}_k^2   \right \|_{L^1_{\omega,s,x}}^{\frac{1}{\varkappa}}  \right)^{\frac{1}{2}} 	
	\\	& \lesssim \delta^{\frac{1}{2} - \nu } 2^{- l \zeta }    \left( \mathbb{E}  \sum_{k=1}^{\infty} \int_x     \int_0^{T}         \abs{\tilde{u}(s,x)}^{2 \nu -1}     \abs{\mathring{g}_k(s,x,\tilde{u}(s,x))}^2     \  ds  \ dx \right)^{\frac{1}{2}}  
 \\ & + \delta^{\frac{1}{2} - \nu }   	     2^{- \frac{l}{ 2 \varpi}}  \left( \sum_{k=1}^{\infty}   \esssup_{s \in [0,T]}   \left \|  \mathring{g}_k^2(s,\cdot,\tilde{u})   \right \|_{L^1_{\omega,x}}^{\frac{1}{\varpi}}   \left \|    \tilde{u}^{\varkappa(2 \nu - 1)}  \mathring{g}_k^2   \right \|_{L^1_{\omega,s,x}}^{\frac{1}{\varkappa}}  \right)^{\frac{1}{2}}
    \\  & \le \delta^{\frac{1}{2} - \nu }  2^{-l \zeta  } \Bigg[ \left( \sum_{k=1}^{\infty} \|      \tilde{u}^{ \nu -\frac{1}{2}}   \mathring{g}_k    \|_{L^2_{\omega,t,x}}^2 \right)^{\frac{1}{2}}    +     \Bigg(  \sum_{k=1}^{\infty}     \esssup_{t \in [0,T]}    \left \|    \mathring{g}_k^2(t,\cdot,\tilde{u})  \right \|_{L^1_{\omega,x}}^{\frac{1}{\varpi}}  \\ & \times    \left \|  \tilde{u}^{\varkappa(2 \nu - 1)}  \mathring{g}_k^2    \right \|_{L^1_{\omega,t,x}}^{\frac{1}{\varkappa}} \Bigg)^{\frac{1}{2}}  \Bigg].
	\end{split}
\end{equation}		
\subparagraph{Step 3(i)}  Let $l > 0 $ be arbitrary and fixed. Following the same arguments of \cite [Proof of Lemma 4.4 Step 2]{gess2019optimal}, we obtain using Bernstein's lemma
\begin{equation*} 
	\begin{split}
	&	\left\Vert  \int_{v} \chi^{1,2}_l \ dv  \right\Vert_{L^1(\Omega; L^2(\mathbb{R};  L^1(\mathbb{T}^d)))} 
		\\	&  \lesssim 2^{ \frac{l}{2}}	\left\Vert  \int_{v} \chi^{1,2}_l \ dv  \right\Vert_{L^1(\Omega \times  \mathbb{R} \times  \mathbb{T}^d)}   
		 \\ & = 2^{ \frac{l}{2}} \left\Vert   \int_{v}  \mathcal{F}_{t,x}^{-1}   \hat{\eta}_{l}   \psi_1 \left(\frac{ \abs{v} }{\delta} \right) \frac{ \abs{v}^{\gamma-1} }{ \mathcal{L}(i\tau, \xi,v) }  \mathcal{F}_{t,x} \abs{v}^{1- \gamma}  h  \ dv \right\Vert_{L^1(\Omega \times  \mathbb{R} \times  \mathbb{T}^d)}  
	\end{split}
\end{equation*}
\begin{equation} \label{Estimate_cut_time_reg_small_v}
	\begin{split}	
 &	\lesssim \delta^{\gamma -1} 2^{ - \frac{l}{2}}   \| \abs{v}^{1- \gamma}   h  \|_{L^1_{\omega} \mathcal{M}_{\text{TV}}}.
	\end{split}
\end{equation}

\subparagraph{Step 4(i)} Let $l > 0 $ be arbitrary and fixed. We have 
\begin{equation} \label{Expression_kinetic_time_small_v}
	\begin{split}
		\int_{v} \chi^{1,3}_l \  dv & = -  \int_{v}  \mathcal{F}_{t,x}^{-1} \hat{\eta}_{l}  \psi_1^{'} \left(\frac{ \abs{v} }{\delta} \right) \frac{ \abs{v} }{\delta}  \frac{ \sgn(v) \abs{v}^{ \gamma-1}}{ \mathcal{L}(i\tau, \xi,v)} \mathcal{F}_{t,x} \abs{v}^{- \gamma} \tilde{h}  \ dv  \\  & + m (m -1)  \int_{v}  \mathcal{F}_{t,x}^{-1} \hat{\eta}_{l} \psi_1 \left(\frac{ \abs{v} }{\delta} \right)  \frac{ 4 \pi^2 \abs{v}^{m-2 + \gamma} \sgn(v) \abs{\xi}^2}{ \mathcal{L}(i\tau, \xi,v)^2} \mathcal{F}_{t,x}    \abs{v}^{- \gamma} \tilde{h}  \ dv. 
	\end{split}
\end{equation}
Using Bernstein's lemma, \eqref{Expression_kinetic_time_small_v} and the same arguments of \cite [Proof of Lemma 4.4 Step 3]{gess2019optimal}, we obtain  
\begin{equation} \label{Kinetic_First_Averaging_small_v}
	\begin{split}
		\left\Vert \int_{v} \chi^{1,3}_l \ dv \right\Vert_{L^1(\Omega; L^2(\mathbb{R};  L^1(\mathbb{T}^d)))}  	&  \lesssim  2^{ \frac{l}{2}} \left\Vert \int_{v} \chi^{1,3}_l \ dv \right\Vert_{L^1(\Omega \times  \mathbb{R} \times  \mathbb{T}^d)} \\ &   \lesssim \delta^{\gamma -1} 2^{- \frac{l}{2}}    \| \abs{v}^{- \gamma} \tilde{h}  \|_{L^1_{\omega} \mathcal{M}_{\text{TV}}}.
	\end{split} 
\end{equation}

\subparagraph{Step 5(i)} We combine the estimates obtained in the previous steps by real interpolation. This argument relies on embeddings of $L^p$ spaces in bounded domains; for this reason, we restrict the norms to the time interval $[-1,T+1]$. The norms in the time interval $ \mathbb{R} \setminus [-1,T+1]$ are bounded after the interpolation argument. 

For $z>0$, let
\begin{equation*}
	\begin{split}
	K(z, \bar{\chi}_{l})  := \inf \left \{ \left \| \int_v \chi^1_{l} \ dv \right \|_{ L^{1}(\Omega;  L^2([-1,T+1]; L^1(\mathbb{T}^d)))}  + z \left \| \int_v \chi^0_{l} \ dv  \right \|_{ L^{\beta}(\Omega \times [-1,T+1] \times \mathbb{T}^d) }  \right \}.
	\end{split}
\end{equation*}
By the estimates \eqref{velocity_degenerate_time},  \eqref{Final_step1_time_reg_small_v}, \eqref{Estimate_cut_time_reg_small_v} and \eqref{Kinetic_First_Averaging_small_v}, we have
\begin{equation*}
	\begin{split}
	K(z, \bar{\chi}_{l}) & \lesssim  	\delta^{\frac{1}{2}- \nu} 2^{-l \zeta}  \Bigg[ \left( \sum_{k=1}^{\infty} \|      \tilde{u}^{ \nu -\frac{1}{2}}  \mathring{g}_k   \|_{L^2_{\omega,t,x}}^2 \right)^{\frac{1}{2}} 
\\   &  +     \left( \sum_{k=1}^{\infty}   \esssup_{t \in [0,T]}    \left \|    \mathring{g}_k^2(t,\cdot,\tilde{u})  \right \|_{L^1_{\omega,x}}^{\frac{1}{\varpi}}  \left \|    \tilde{u}^{\varkappa(2 \nu - 1)}  \mathring{g}_k^2   \right \|_{L^1_{\omega,t,x}}^{\frac{1}{\varkappa}} \right)^{\frac{1}{2}}  \Bigg] \\ & + \delta^{\gamma -1} 2^{-   \frac{l}{2} } ( \| \abs{v}^{1- \gamma}   h  \|_{L^1_{\omega} \mathcal{M}_{\text{TV}}} + \| \abs{v}^{- \gamma} \tilde{h}  \|_{L^1_{\omega} \mathcal{M}_{\text{TV}}} ) + z \delta^{\rho } \| \chi  \|_{ L^{\beta}_{\omega,t,x,v}}.
	\end{split}
\end{equation*}
We choose $\nu = \frac{3}{2} - \gamma$ so that the above estimate becomes
\begin{equation} \label{K_functional_bootstrap_time_reg_small_v}
	K(z, \bar{\chi}_{l})  \lesssim      \delta^{\gamma -1} 2^{- l \zeta } \mathcal{K}_{(i)} + z \delta^{\rho } \| \chi  \|_{ L^{\beta}_{\omega,t,x,v}},
\end{equation}
where 
\begin{equation*}
	\begin{split}
	\mathcal{K}_{(i)} & = \left( \sum_{k=1}^{\infty} \|     \tilde{u}^{ 1-\gamma}  \mathring{g}_k    \|_{L^2_{\omega,t,x}}^2 \right)^{\frac{1}{2}}     +     \left( \sum_{k=1}^{\infty}     \esssup_{t \in [0,T]}    \left \|    \mathring{g}_k^2(t,\cdot,\tilde{u})  \right \|_{L^1_{\omega,x}}^{\frac{1}{\varpi}}  \left \|  \tilde{u}^{2 \varkappa(1-\gamma)}   \mathring{g}_k^2  \right \|_{L^1_{\omega,t,x}}^{\frac{1}{\varkappa}} \right)^{\frac{1}{2}}  \\ & + \| \abs{v}^{1- \gamma}   h  \|_{L^1_{\omega} \mathcal{M}_{\text{TV}}}  + \| \abs{v}^{- \gamma} \tilde{h}  \|_{L^1_{\omega} \mathcal{M}_{\text{TV}}} .
	\end{split}
\end{equation*}
We equilibrate the first and the second term on the right-hand side in \eqref{K_functional_bootstrap_time_reg_small_v} as follows
\begin{equation*} \label{calibration_bounds_time_reg_small_v}
	\delta^{\gamma -1} 2^{- l \zeta } = z \delta^{\rho }.
\end{equation*}
This allows us to derive the value of $ \delta$ from the  expression above, namely $ \delta = z^{- \frac{1}{1 - \gamma+\rho}}   2^{-l  \frac{ \zeta }{1 - \gamma+\rho}}$ and plug into the right-hand side of \eqref{K_functional_bootstrap_time_reg_small_v} that becomes 
\begin{equation} \label{z_small_time_reg_small_v}
	z^{- \theta} K(z, \bar{\chi}_{l}) \lesssim 2^{-l \tilde{\kappa}_t }  \left(  \mathcal{K}_{(i)} + \| \chi   \|_{ L^{\beta}_{\omega,t,x,v}}  \right),  
\end{equation} 
where $\theta :=  \frac{ 1 - \gamma }{1 - \gamma + \rho} $ and $ \tilde{\kappa}_t  :=  \frac{\zeta \rho }{1 - \gamma + \rho}$.
We take the supremum over $z>0$ on both sides in \eqref{z_small_time_reg_small_v} to have
\begin{equation}  \label{z_small_time_reg_small_v_after}
	\begin{split}
	\| \bar{\chi}_{l} \|_{\left(L^{1}(\Omega;  L^2([-1,T+1]; L^1(\mathbb{T}^d))), L^{\beta}(\Omega \times [-1,T+1] \times \mathbb{T}^d) \right)_{\theta, \infty}} \lesssim 2^{-l \tilde{\kappa}_t }  \left(  \mathcal{K}_{(i)} + \| \chi   \|_{ L^{\beta}_{\omega,t,x,v}}  \right).
	\end{split}
\end{equation}
Let
\begin{equation*}
	\begin{split}
	\bar{p} & = \frac{1}{1- \theta \rho}=  \frac{1 - \gamma + \rho}{\rho + (1-\rho) (1 - \gamma  ) }, 
	\\	q & = \frac{2}{1+ (\theta- \epsilon) (1- 2 \rho)}= \frac{2(1- \gamma + \rho)}{1- \gamma + \rho+ (1-\gamma - \epsilon(1-\gamma+\rho))(1-2\rho)}, \\
		  p  &= \frac{1}{1- (\theta-2\epsilon) \rho}  = \frac{1 - \gamma + \rho}{1 - \gamma + \rho- ( 1 - \gamma - 2 \epsilon(1-\gamma + \rho)) \rho}.
\end{split}
\end{equation*}
Recall that $p < \bar{p}< q$. We apply a series of real interpolation arguments below. We use \cite[Corollary 3.8.2]{bergh2012interpolation} in the first and last embedding, 
\cite[Theorem 1.18.4]{Triebel78Interpolation} in all the three equalities, \cite[Theorem 3.4.1 (b)]{bergh2012interpolation} in the second embedding to obtain
\begin{equation} \label{Embeddings_time_regularity_Averaging_lemmata}
	\begin{split}
&	\left(L^{1}(\Omega ; L^2([-1,T+1]); L^1(\mathbb{T}^d)), L^{\beta}(\Omega; L^{\beta}([-1,T+1]); L^{\beta}(\mathbb{T}^d) ) \right)_{\theta, \infty} \\ & \hookrightarrow 	\left(L^{1}(\Omega ; L^2([-1,T+1]); L^1(\mathbb{T}^d)), L^{\beta}(\Omega; L^{\beta}([-1,T+1]); L^{\beta}(\mathbb{T}^d) )\right)_{\theta - \epsilon, \bar{p}} \\ &  = L^{ \bar{p}} \left( \Omega;  ( L^2([-1,T+1]; L^1(\mathbb{T}^d)  , L^{\beta}([-1,T+1]; L^{\beta}(\mathbb{T}^d) ) )_{\theta - \epsilon, \bar{p}} \right) \\ & \hookrightarrow  L^{ \bar{p}}  \left( \Omega; (L^2([-1,T+1]; L^1(\mathbb{T}^d) , L^{\beta}([-1,T+1]; L^{\beta}(\mathbb{T}^d) ) )_{\theta - \epsilon, q} \right) \\ & =  L^{ \bar{p}} \left( \Omega;  L^{q} ( [-1,T+1] ;( L^1(\mathbb{T}^d)  , L^{\beta}(\mathbb{T}^d)  )_{\theta - \epsilon, q} ) \right) \\ & \hookrightarrow  L^{ \bar{p}} \left( \Omega; L^{q} ( [-1,T+1] ; ( L^1(\mathbb{T}^d)  , L^{\beta}(\mathbb{T}^d)  )_{\theta - 2 \epsilon, p} ) \right)
	 \\ & =  L^{ \bar{p}} \left( \Omega; L^{q} ( [-1,T+1] ;  L^p(\mathbb{T}^d) ) \right).
	\end{split}
\end{equation}
Using the embeddings \eqref{Embeddings_time_regularity_Averaging_lemmata} in \eqref{z_small_time_reg_small_v_after}, we have
\begin{equation} \label{switch_inequality_time_reg_small_v} 
	\| \bar{\chi}_{l} \|_{  L^{ \bar{p}} ( \Omega; L^{q} ( [-1,T+1] ;  L^p(\mathbb{T}^d) ))}   \lesssim 2^{-l \tilde{\kappa}_t}  \left(  \mathcal{K}_{(i)} + \| \chi   \|_{ L^{\beta}_{\omega,t,x,v}}  \right).
\end{equation}
Using the fact that $\bar{\chi}$ is compactly supported in $(0,T)$, Young's convolution inequality and the fact that $\eta$ is a Schwartz function, we have

\begin{equation}  \label{switch_inequality_time_reg_small_v_large_time_Part_i}
	\begin{split}
	 \| \bar{\chi}_l \|_{L^{ \bar{p}}(\Omega;  L^{q}( \mathbb{R} \setminus [-1,T+1];    L^p(\mathbb{T}^d) ))}
	 & =  \left \| \int_r \eta_{l}(r) \bar{\chi}(\cdot-r,\cdot) \ dr \right \|_{L^{ \bar{p}}(\Omega;  L^{q}( \mathbb{R} \setminus [-1,T+1];    L^p(\mathbb{T}^d) ))}
	 \\ & = \left \| \int_r \eta_{l}(r) \mathbbm{1}_{(B_1(0))^c}(r) \ \bar{\chi}(\cdot-r,\cdot) \ dr \right \|_{L^{ \bar{p}}(\Omega;  L^{q}( \mathbb{R} \setminus [-1,T+1];    L^p(\mathbb{T}^d) ))}
	 	 \\ & \le \left \| \int_r \eta_{l}(r) \mathbbm{1}_{(B_1(0))^c}(r) \ \bar{\chi}(\cdot-r,\cdot) \ dr \right \|_{L^{ \bar{p}}(\Omega;  L^{q}( \mathbb{R} ;    L^p(\mathbb{T}^d) ))}
	 \\ & \le  \| \eta_l \mathbbm{1}_{(B_1(0))^c} \|_{L^{1}( \mathbb{R})}	 \| \bar{\chi} \|_{L^{ \bar{p}}(\Omega;  L^{q}( \mathbb{R} ;    L^p(\mathbb{T}^d) ))}
	 \\ & = \int_{\mathbb{R}} |\eta(t)| \mathbbm{1}_{|t| \ge 2^l} \ dt \	 \| \bar{\chi} \|_{L^{ \bar{p}}(\Omega;  L^{q}( \mathbb{R} ;    L^p(\mathbb{T}^d) ))}
	 \\ & \lesssim 2^{-l} \| \bar{\chi} \|_{L^{ \bar{p}}(\Omega;  L^{q}( \mathbb{R};    L^p(\mathbb{T}^d) ))}.
	 \end{split}
\end{equation}
Using  \eqref{switch_inequality_time_reg_small_v} and \eqref{switch_inequality_time_reg_small_v_large_time_Part_i}, we have for $l>0$
\begin{equation} \label{switch_inequality_time_reg_small_v_all_time_part_i}
	\begin{split}
		\| \bar{\chi}_l \|_{L^{ \bar{p}}(\Omega;  L^{q}(\mathbb{R};    L^p(\mathbb{T}^d) ))} 
		&  \le  \| \bar{\chi}_l \|_{L^{ \bar{p}}(\Omega;  L^{q}([-1,T+1];    L^p(\mathbb{T}^d) ))}  +  \| \bar{\chi}_l \|_{L^{ \bar{p}}(\Omega;  L^{q}( \mathbb{R} \setminus [-1,T+1];    L^p(\mathbb{T}^d) ))}
		\\ & \lesssim   2^{-l \tilde{\kappa}_t}  \left(  \mathcal{K}_{(i)} + \| \chi   \|_{ L^{\beta}_{\omega,t,x,v}}  \right) 
		 +  2^{-l} \| \bar{\chi} \|_{L^{ \bar{p}}_{\omega} L^{q}_t L^p_x}
		 \\ & \le  2^{-l \tilde{\kappa}_t} \left(  \mathcal{K}_{(i)} + \| \chi   \|_{ L^{\beta}_{\omega,t,x,v}} +  \| \bar{\chi} \|_{L^{ \bar{p}}_{\omega} L^{q}_t L^p_x}  \right).
	\end{split}
\end{equation}
We then continue along the same lines as in the proof of Lemma \ref{Isotropic_Averaging_Lemma}. Namely, we  multiply the left-and right-hand side of \eqref{switch_inequality_time_reg_small_v_all_time_part_i} by $  2^{l \tilde{\kappa}_t}   $, take the supremum over $l >  0$ to arrive at
	\begin{gather*} 
	\sup_{l > 0}     2^{l \tilde{\kappa}_t}    \| \bar{\chi}_{l}  \|_{ L^{ \bar{p}}(\Omega;  L^{q}(\mathbb{R};    L^p(\mathbb{T}^d) ))}  \lesssim     \mathcal{K}_{(i)} + \| \chi   \|_{ L^{\beta}_{\omega,t,x,v}} +  \| \bar{\chi} \|_{L^{ \bar{p}}_{\omega} L^{q}_t L^p_x}. 
\end{gather*}
Let $\kappa_t := \tilde{\kappa}_t - \epsilon$.  Similarly as done in \eqref{Losing_epsilon_space_regularity_interpolation_step}, we have
\begin{equation} \label{losing_epsilon_time_explanation}
	\begin{split}
\left\Vert \sup_{l > 0} 2^{l \kappa_t}  \| \bar{\chi}_{l}  \|_{ L^{q}(\mathbb{R};    L^p(\mathbb{T}^d) )} \right\Vert_{L^{\bar{p}}(\Omega)}   & \lesssim 	\sup_{l > 0}     2^{l \tilde{\kappa}_t}    \| \bar{\chi}_{l}  \|_{L^{ \bar{p}}(\Omega;  L^{q}(\mathbb{R};    L^p(\mathbb{T}^d) ))} \\ &  \lesssim     \mathcal{K}_{(i)} + \| \chi   \|_{ L^{\beta}_{\omega,t,x,v}} +  \| \bar{\chi} \|_{L^{ \bar{p}}_{\omega} L^{q}_t L^p_x}.
\end{split}
\end{equation}	
For $l=0$, we use 
\begin{equation} \label{case_j0_l0_time}
   \|   \bar{\chi}_{l}   \|_{L^{ \bar{p}}(\Omega;  L^{q}(\mathbb{R};    L^p(\mathbb{T}^d) ))} \lesssim     \|   \bar{\chi}   \|_{L^{ \bar{p}}(\Omega;  L^{q}(\mathbb{R};    L^p(\mathbb{T}^d) ))}. 
\end{equation}
Using \eqref{losing_epsilon_time_explanation} and \eqref{case_j0_l0_time}, we have 
		\begin{equation} \label{final_cut_off_time}
	\begin{split}
	\| \bar{\chi}   \|_{L^{\bar{p}}(\Omega; B^{\kappa_t}_{q, \infty}(\mathbb{R}; L^p(\mathbb{T}^d) ))}  & = \left\Vert \sup_{l \ge 0}  2^{l \kappa_t}  \| \bar{\chi}_{l}  \|_{L^{q}(\mathbb{R};    L^p(\mathbb{T}^d) )} \right\Vert_{L^{\bar{p}}(\Omega)}  \\ &  \lesssim    \mathcal{K}_{(i)} + \| \chi   \|_{ L^{\beta}_{\omega,t,x,v}}   +  \| \bar{\chi} \|_{L^{ \bar{p}}_{\omega} L^{q}_t L^p_x}. 
	\end{split}  
\end{equation} 

We remove the assumption that $ \chi$ is localized in $v$ similarly as done in \cite[Proof of Lemma 4.2]{gess2019optimal}. Let $ \Upsilon$, $\chi^{\Upsilon}$, $\mathring{g}^{\Upsilon}_k$, $h^{\Upsilon}$, $ \tilde{h}^{\Upsilon^{'}}$ and $ \tilde{h}^{\Upsilon}$ be as in the proof of Lemma \ref{Isotropic_Averaging_Lemma}. Then $ \chi^{\Upsilon}$ is a solution to \eqref{localization_v_space} and  the estimate \eqref{final_cut_off_time} reads as follows:
\begin{equation} \label{Smooth_cutoff_time_reg_large_v}
\begin{split}
 \| \bar{\chi}^{\Upsilon} \|_{L^{\bar{p}}(\Omega; B^{\kappa_t}_{q, \infty}(\mathbb{R}; L^p(\mathbb{T}^d) ))}  & \lesssim        \left( \sum_{k=1}^{\infty} \|     \tilde{u}^{ 1- \gamma}   \mathring{g}_k^{\Upsilon}    \|_{L^2_{\omega,t,x}}^2 \right)^{\frac{1}{2}}    \\ &  +    \left( \sum_{k=1}^{\infty}    \esssup_{t \in [0,T]}   \left \|   (\mathring{g}^{\Upsilon}_k)^2(t,\cdot,\tilde{u})  \right \|_{L^1_{\omega,x}}^{\frac{1}{\varpi}}       \left \|   \tilde{u}^{2 \varkappa(1- \gamma)}  (\mathring{g}^{\Upsilon}_k)^2   \right \|_{L^1_{\omega,t,x}}^{\frac{1}{\varkappa}}  \right)^{\frac{1}{2}}    \\ &  +  \| \abs{v}^{1- \gamma} ( h^{\Upsilon} - \tilde{h}^{\Upsilon^{'}}) \|_{L^1_{\omega} \mathcal{M}_{\text{TV}}}    + \| \abs{v}^{- \gamma} \tilde{h}^{\Upsilon}\|_{L^1_{\omega} \mathcal{M}_{\text{TV}}} \\ &   + \| \chi^{\Upsilon}  \|_{ L^{\beta}_{\omega,t,x,v}} + \| \bar{\chi}^{\Upsilon} \|_{L^{ \bar{p}}_{\omega} L^{q}_t L^p_x}. 
\end{split}
		\end{equation}
Let $\Upsilon_n$ and  $\Upsilon_n^{'}$ be as in the proof of Lemma \ref{Isotropic_Averaging_Lemma}. For the first term in the right-hand side of \eqref{Smooth_cutoff_time_reg_large_v}, we have  
		\begin{equation*}
		\begin{split}
& \Bigg(  \sum_{k=1}^{\infty} \mathbb{E} \int_{t,x}    \abs{\tilde{u}(t,x)}^{2( 1-\gamma)} \abs{\mathring{g}^{\Upsilon_n}_k(t,x,\tilde{u}(t,x))}^2       \ dx      \ dt \Bigg)^{\frac{1}{2}} \\ &  \lesssim \left( \sum_{k=1}^{\infty}  \mathbb{E} \int_{t,x}      \mathbbm{1}_{  i_n \le  \abs{\tilde{u}} } \abs{\tilde{u}(t,x)}^{2( 1-\gamma)} \abs{\mathring{g}_k(t,x,\tilde{u}(t,x))}^2         \ dt  \ dx   \right)^{\frac{1}{2}}.
			\end{split}
	\end{equation*}
The remaining terms can be treated as in the proof of Lemma \ref{Isotropic_Averaging_Lemma}. 
		We may then take the limit $ n \rightarrow \infty$ in \eqref{Smooth_cutoff_time_reg_large_v} and use Fatou's lemma to obtain \eqref{final_cut_off_time} also for general $ \chi$. 
		
		\paragraph{Part \textit{(ii)}} The proof is similar to the one of Part \textit{(i)}. The main difference is that here we do not introduce a $\delta$-dependent cut-off in $v$ since $ \chi$ is supported in $ \abs{v} \ge 1$. 
		
		As in the proof of Part  \textit{(i)}, we first assume that $ \chi$ is compactly supported in $v$ and then remove this qualitative assumption at the end of the proof. 

		 Then $\chi$ solves, in the sense of  \eqref{notion_solution_kinetic_torus},
		 \begin{equation*}
		 	\begin{split}
		 		\mathcal{L}(\partial_t, \nabla_x,v) \chi & =   \sum_{k=1}^{\infty}  \delta_{\tilde{u}=v }   \mathring{g}_k \dot{\beta}_k + h +  \partial_v \tilde{h}.
		 	\end{split}
		 \end{equation*}
Thus,
	\begin{equation} \label{non_degenerate_decomposition_starting_part_ii}
			\begin{split}
\chi(t,x,v) & = \mathcal{F}_{t,x}^{-1}  \frac{1}{ \mathcal{L}(i\tau, \xi,v)} \mathcal{F}_{t,x} \left(  \sum_{k=1}^{\infty}  \delta_{\tilde{u}=v }   \mathring{g}_k \dot{\beta}_k \right)
\\ & + \mathcal{F}_{t,x}^{-1}  \frac{1}{ \mathcal{L}(i\tau, \xi,v)} \mathcal{F}_{t,x}  h(t,x,v)   + \mathcal{F}_{t,x}^{-1}  \frac{1}{ \mathcal{L}(i\tau, \xi,v)} \mathcal{F}_{t,x}  \partial_v \tilde{h}(t,x,v) \\  & =: \chi^{1}(t,x,v)  + \chi^{2}(t,x,v)  + \chi^{3}(t,x,v).
			\end{split}
		\end{equation}
Let $\eta_{l}$ be as in Part \textit{(i)}. The Littlewood--Paley block for $\chi$ is defined as follows
\begin{equation} \label{Littlewood_Paley_block_part_2_time}
	\begin{split}
		\chi_{l}(t,x,v) & = \int_r \eta_{l}(r) \chi(t-r,x,v)  \ dr.
			\end{split}
	\end{equation}
The blocks for $ \chi^{1} $, $ \chi^{2} $ and $\chi^{3} $ are defined as follows
\begin{equation*}
	\begin{split}
\chi_{l}^1(t,x,v) &=	 \int_r \eta_{l}(r) \chi^{1}(t-r,x,v)  \ dr,
\\ \chi^{2}_l(t,x,v)  &=		 \mathcal{F}_{t,x}^{-1} \hat{\eta}_{l}    \frac{ 1 }{ \mathcal{L}(i\tau, \xi,v) }  \mathcal{F}_{t,x}  h(t,x,v),   \\   \chi^{3}_l(t,x,v) & = \mathcal{F}_{t,x}^{-1}  \hat{\eta}_{l}  \frac{ 1 }{ \mathcal{L}(i\tau, \xi,v) }  \mathcal{F}_{t,x} \partial_v  \tilde{h}(t,x,v),
	\end{split}
\end{equation*}
respectively.

We can decompose \eqref{Littlewood_Paley_block_part_2_time} as follows
\begin{equation*}
		\begin{split}
		\chi_{l}(t,x,v)  & = \int_r \eta_{l}(r) (\chi^{1}(t-r,x,v)  - \chi^{1}(t,x,v))  \ dr \\ &   + \chi^{2}_l(t,x,v) + \chi^{3}_l(t,x,v) . 
	\end{split}
\end{equation*}

We have obtained the following decomposition
\begin{equation} \label{decomposition_time_v_large}
	\begin{split}
		\bar{\chi}_l(t,x)   & =  \int_v \chi_l(t,x,v) \ dv \\    & = \int_{v,r} \eta_{l}(r) (\chi^{1}(t-r,x,v)  - \chi^{1}(t,x,v))   \ dr \ dv \\ &   + \int_v \chi^{2}_l(t,x,v)  \ dv + \int_v \chi^{3}_l(t,x,v) \ dv .  
	\end{split}
\end{equation}
 
 We proceed in dividing this part of the proof into four steps. In the first three we estimate the velocity averages \eqref{decomposition_time_v_large} and in the final step we combine all the estimates.  We introduce two different arguments from Part \textit{(i)}. In Step 3(ii) below, we derive an estimate for the term containing $\tilde{h}$ suitable for $\tilde{\gamma}>1$. In Step 4(ii) below, we combine the estimates obtained in the previous steps without using the real interpolation method. The estimates in Step 1(ii) and 2(ii) are identical to those in Step 2(i) and 3(i) respectively, with the difference that here the estimates are not $\delta$-dependent.
 
\subparagraph{Step 1(ii)} Let $l> 0 $ be arbitrary and fixed. Observe that $\chi^{1}$ in \eqref{non_degenerate_decomposition_starting_part_ii} is given by
\begin{equation*} 
	\begin{split}
		\chi^{1}(t,x,v)= \sum_{k=1}^{\infty} \int_0^t \int_{\mathbb{T}^d} \Phi(a(v)(t-s),x-y)  \delta_{\tilde{u}(s,y)=v }   \mathring{g}_k(s,y,v) \ dy \  d\beta_k(s).
	\end{split}
\end{equation*}
 Using the same arguments as in Step 2(i), we obtain
\begin{equation} \label{Final_step1_time_reg_large_v}
	\begin{split}
		& \left\Vert    \int_{v,r} \eta_{l}(r) (\chi^{1,1}(\cdot-r,\cdot,v)  - \chi^{1,1}(\cdot,\cdot,v))   \ dr  \  dv \right\Vert_{L^1(\Omega;  L^2(([-1,T+1]) ;  L^1(\mathbb{T}^d)))}  \\ &  \lesssim   2^{-l \zeta  } \Bigg[ \left( \sum_{k=1}^{\infty} \|      \tilde{u}^{ \nu -\frac{1}{2}}   \mathring{g}_k    \|_{L^2_{\omega,t,x}}^2 \right)^{\frac{1}{2}}     +      \left( \sum_{k=1}^{\infty}       \esssup_{t \in [0,T]}    \left \|    \mathring{g}_k^2(t,\cdot,\tilde{u})  \right \|_{L^1_{\omega,x}}^{\frac{1}{\varpi}}      \left \|    \tilde{u}^{\varkappa(2 \nu - 1)}   \mathring{g}_k^2  \right \|_{L^1_{\omega,t,x}}^{\frac{1}{\varkappa}} \right)^{\frac{1}{2}}   \Bigg].
	\end{split}
\end{equation}
\subparagraph{Step 2(ii)}  Let $l > 0 $ be arbitrary and fixed. Using the same arguments as in Step 3(i), we obtain
\begin{equation} \label{Estimate_cut_time_reg_large_v}
	\begin{split}
		\left\Vert  \int_{v} \chi^{1,2}_l \ dv  \right\Vert_{L^1(\Omega; L^2(\mathbb{R};  L^1(\mathbb{T}^d)))} & \lesssim 2^{\frac{l}{2}}			\left\Vert  \int_{v} \chi^{2}_l \ dv  \right\Vert_{L^1(\Omega \times  \mathbb{R} \times  \mathbb{T}^d)} \\ &  = 2^{\frac{l}{2}}	 \left\Vert   \int_{v}  \mathcal{F}_{t,x}^{-1}    \hat{\eta}_{l}  \frac{ 1}{ \mathcal{L}(i\tau, \xi,v) }  \mathcal{F}_{t,x}  h    \ dv \right\Vert_{L^1(\Omega \times  \mathbb{R} \times  \mathbb{T}^d)}  \\ 
		&	\lesssim 2^{- \frac{l}{2}}  \|    h  \|_{L^1_{\omega} \mathcal{M}_{\text{TV}}}.
	\end{split}
\end{equation}
\subparagraph{Step 3(ii)} Let $l > 0 $ be arbitrary and fixed. We assume without loss of generality that $ \abs{\xi} \ge 1$ (since $x \in \mathbb{T}^d$). We have
\begin{equation} \label{Expression_kinetic_time_large_v}
	\begin{split}
		\int_{v}  \chi^{3}_l  \ dv & =  m (m -1)  \int_{v}  \mathcal{F}_{t,x}^{-1} \hat{\eta}_{l}   \frac{4 \pi^2  \abs{v}^{m-2 +  \tilde{\gamma}} \sgn(v) \abs{\xi}^2}{ \mathcal{L}(i\tau, \xi,v)^2} \mathcal{F}_{t,x}    \abs{v}^{-  \tilde{\gamma}}  \tilde{h}  \ dv. 
	\end{split}
\end{equation}
Note that $ \abs{\tau} \sim 2^l$  on the support of $ \hat{\eta}_{l}(\tau) $ and $ \abs{v} \ge 1$ by assumption.  Recall that $\bar{\epsilon} \in (0,1)$ such that $   \tilde{\gamma} < 1 + \bar{\epsilon} (m-1) $. Then
\begin{equation*}
	\begin{split}
		\frac{ 4 \pi^2  \abs{v}^{m-2 +  \tilde{\gamma}} \abs{\xi}^2}{ \abs{\mathcal{L}(i\tau, \xi,v)}^2} \lesssim \frac{\abs{v}^{ \tilde{\gamma}-1} }{ \abs{\tau}^{1- \bar{\epsilon}}  \abs{v}^{\bar{\epsilon}(m-1)} \abs{\xi}^{2 \bar{\epsilon}}}  \lesssim \abs{v}^{ \tilde{\gamma} -1- \bar{\epsilon}(m-1) }  \abs{\tau}^{-(1- \bar{\epsilon})} \lesssim  2^{-l(1- \bar{\epsilon})} . 
	\end{split}
\end{equation*}
Using Bernstein's lemma, \eqref{Expression_kinetic_time_large_v} and the same arguments of \cite [Proof of Lemma 4.4 Step 3]{gess2019optimal}, we obtain 
\begin{equation} \label{Kinetic_First_Averaging_large_v}
	\begin{split}
		\left\Vert \int_{v} \chi^{3}_l \ dv \right\Vert_{L^1(\Omega; L^2(\mathbb{R};  L^1(\mathbb{T}^d)))}  & \lesssim 	2^{\frac{l}{2}}	\left\Vert \int_{v} \chi^{3}_l \ dv \right\Vert_{L^1(\Omega \times  \mathbb{R} \times  \mathbb{T}^d)}  \\ & \lesssim  		2^{\frac{l}{2}} 2^{-l(1- \bar{\epsilon})}  \| \abs{v}^{-  \tilde{\gamma}} \tilde{h}  \|_{L^1_{\omega} \mathcal{M}_{\text{TV}}}  \\ & =  2^{- l  (\frac{1}{2} - \bar{\epsilon}) } \| \abs{v}^{-  \tilde{\gamma}} \tilde{h}  \|_{L^1_{\omega} \mathcal{M}_{\text{TV}}}.
	\end{split} 
\end{equation}
\subparagraph{Step 4(ii)} We combine the estimates derived in the previous steps. We denote 
\begin{equation*}
	\begin{split}
		\mathcal{K}_{(ii)} & =    \left( \sum_{k=1}^{\infty} \|      \tilde{u}^{ \nu -\frac{1}{2}}   \mathring{g}_k    \|_{L^2_{\omega,t,x}}^2 \right)^{\frac{1}{2}}    +     \left( \sum_{k=1}^{\infty}       \esssup_{t \in [0,T]}    \left \|    \mathring{g}_k^2(t,\cdot,\tilde{u})  \right \|_{L^1_{\omega,x}}^{\frac{1}{\varpi}}     \left \|  \tilde{u}^{\varkappa(2 \nu - 1)}   \mathring{g}_k^2   \right \|_{L^1_{\omega,t,x}}^{\frac{1}{\varkappa}} \right)^{\frac{1}{2}}  \\ &   + \|    h  \|_{L^1_{\omega} \mathcal{M}_{\text{TV}}}  + \| \abs{v}^{-  \tilde{\gamma}} \tilde{h}  \|_{L^1_{\omega} \mathcal{M}_{\text{TV}}}.
	\end{split}
\end{equation*}
By the estimates  \eqref{Final_step1_time_reg_large_v}, \eqref{Estimate_cut_time_reg_large_v} and  \eqref{Kinetic_First_Averaging_large_v}, we have
\begin{equation} \label{K_functional_bootstrap_time_reg_large_v}
	\left\Vert \bar{\chi}_{l} \right\Vert_{L^{ 1}(\Omega;  L^{2}([-1,T+1];    L^1(\mathbb{T}^d) ))}   \lesssim       2^{-l 	\zeta } \mathcal{K}_{(ii)} . 
\end{equation}
Using the fact that $\bar{\chi}$ is compactly supported in $(0,T)$, Young's convolution inequality and the fact that $\eta$ is a Schwartz function, we have
\begin{equation}  \label{switch_inequality_time_reg_small_v_large_time_Part_ii}
	\begin{split}
		\| \bar{\chi}_l \|_{L^{ 1}(\Omega;  L^{2}( \mathbb{R} \setminus [-1,T+1];    L^1(\mathbb{T}^d) ))}
		& =  \left \| \int_r \eta_{l}(r) \bar{\chi}(\cdot-r,\cdot) \ dr \right \|_{L^{ 1}(\Omega;  L^{2}( \mathbb{R} \setminus [-1,T+1];    L^1(\mathbb{T}^d) ))}
		\\ & = \left \| \int_r \eta_{l}(r) \mathbbm{1}_{(B_1(0))^c}(r) \ \bar{\chi}(\cdot-r,\cdot) \ dr \right \|_{L^{ 1}(\Omega;  L^{2}( \mathbb{R} \setminus [-1,T+1];    L^1(\mathbb{T}^d) ))}
		\\ & \le \left \| \int_r \eta_{l}(r) \mathbbm{1}_{(B_1(0))^c}(r) \ \bar{\chi}(\cdot-r,\cdot) \ dr \right \|_{L^{ 1}(\Omega;  L^{2}( \mathbb{R} ;    L^1(\mathbb{T}^d) ))}
		\\ & \le  \| \eta_l \mathbbm{1}_{(B_1(0))^c} \|_{L^{1}( \mathbb{R})}	 \| \bar{\chi} \|_{L^{ 1}(\Omega;  L^{2}( \mathbb{R} ;    L^1(\mathbb{T}^d) ))}
		 \\ & = \int_{\mathbb{R}} |\eta(t)| \mathbbm{1}_{|t| \ge 2^l} \ dt \	 \| \bar{\chi} \|_{L^{ 1}(\Omega;  L^{2}( \mathbb{R} ;    L^1(\mathbb{T}^d) ))}
		\\ & \lesssim 2^{-l} \| \bar{\chi} \|_{L^{1}(\Omega;  L^{2}( \mathbb{R};    L^1(\mathbb{T}^d) ))}.
	\end{split}
\end{equation}
Using  \eqref{K_functional_bootstrap_time_reg_large_v} and \eqref{switch_inequality_time_reg_small_v_large_time_Part_ii}, we have for $l>0$
\begin{equation} \label{switch_inequality_time_reg_small_v_all_time_part_ii}
	\begin{split}
		\| \bar{\chi}_l \|_{L^{ 1}(\Omega;  L^{2}(\mathbb{R};    L^1(\mathbb{T}^d) ))}
		 & \le  \| \bar{\chi}_l \|_{L^{ 1}(\Omega;  L^{2}([-1,T+1];    L^1(\mathbb{T}^d) ))}   +  \| \bar{\chi}_l \|_{L^{1}(\Omega;  L^{2}( \mathbb{R} \setminus [-1,T+1];    L^1(\mathbb{T}^d) ))}.
		\\ & \lesssim   2^{-l \zeta}   \mathcal{K}_{(ii)}   
		+  2^{-l} \| \bar{\chi} \|_{L^{ 1}_{\omega} L^{2}_t L^1_x}
		\\ & \le  2^{-l \zeta} \left(  \mathcal{K}_{(ii)} +  \| \bar{\chi} \|_{L^{ 1}_{\omega} L^{2}_t L^1_x}  \right).
	\end{split}
\end{equation}
We then proceed as in Part \textit{(i)}. Namely, we  multiply the left-and right-hand side of \eqref{switch_inequality_time_reg_small_v_all_time_part_ii} by $  2^{l \zeta}   $ and take the supremum over $l >  0$
\begin{equation*}
	\begin{split}
		\left\Vert \sup_{l > 0} 2^{l (\zeta - \epsilon)}  \| \bar{\chi}_{l}  \|_{ L^{2}(\mathbb{R};    L^1(\mathbb{T}^d) )} \right\Vert_{L^{1}(\Omega)}    \lesssim 	\sup_{l > 0}     2^{l \zeta}    \| \bar{\chi}_{l}  \|_{L^{ 1}(\Omega;  L^{2}(\mathbb{R};    L^1(\mathbb{T}^d) ))}  \lesssim     \mathcal{K}_{(ii)}  +  \| \bar{\chi} \|_{L^{ 1}_{\omega} L^{2}_t L^1_x}.
	\end{split}
\end{equation*}	
For $l=0$, we use
\begin{equation*} 
  \|   \bar{\chi}_{l}   \|_{L^{ 1}(\Omega;  L^{2}(\mathbb{R};    L^1(\mathbb{T}^d) ))} \lesssim \|   \bar{\chi}  \|_{L^{ 1}(\Omega;  L^{2}(\mathbb{R};    L^1(\mathbb{T}^d) ))}. 
\end{equation*}
 The two estimates above lead to 
		\begin{equation} \label{final_cut_off_time_large}
	\begin{split}
			\| \bar{\chi}   \|_{L^{1}(\Omega; B^{\zeta - \epsilon}_{2, \infty}(\mathbb{R}; L^1(\mathbb{T}^d) ))}  = \left \| \sup_{l \ge 0}  2^{l (\zeta - \epsilon)}  \| \bar{\chi}_{l} \|_{L^{2}(\mathbb{R}; L^1(\mathbb{T}^d))} \right \|_{L^{1}(\Omega)}   \lesssim    \mathcal{K}_{(ii)} +   \| \bar{\chi} \|_{L^{ 1}_{\omega} L^{2}_t L^1_x}.
	\end{split}  
\end{equation} 
The removal of the assumption that $ \chi$ has compact support in $v$ follows along the same lines of Part \textit{(i)}.
\end{proof}
	\section{Application to Stochastic Porous Medium Equations} \label{Application}
	In this section, we provide the proofs of our main results (Theorem \ref{Main_theorem} and Theorem \ref{Main_theorem_time}) by applying the averaging lemmata obtained
	in the previous section to kinetic solutions of \eqref{degenerate_parabolic_hyperbolic_SPDE}. Before that, we establish an a \textit{priori} bound on $u$, an a \textit{priori} bound on the singular moments of the kinetic measure and a version of the Poincar{\'e} inequality.


	\begin{lem} \label{control_apriori_estimate_new}
	Let $p \in [1, \infty)$, $ \alpha \in [\frac{1}{2},1]$ and $u$ be a solution to \eqref{degenerate_parabolic_hyperbolic_SPDE} in the sense of Definition \ref{kinetic_solution} with $u_0 \in L^{2 \alpha}(\mathbb{T}^d)$. Then, there is a constant $C=C(T,\alpha, D,p) > 0$ such that
	\begin{equation} \label{Estimate_L1_energy}
		\mathbb{E} \esssup_{t \in[0, T]} \| u(t) \|_{L^{2 \alpha}}^{2 \alpha p}  \le C (  \|u_0 \|_{L^{2 \alpha}}^{2 \alpha p} +1) .
	\end{equation}
\end{lem}
\begin{proof}
For fixed $\epsilon \in (0,1]$, let 
	\begin{equation*}
		\Theta_{\epsilon}^{''}(u) = 
		\epsilon^{-1}  \mathbbm{1}_{\abs{u} \le \epsilon} - \epsilon \mathbbm{1}_{\frac{1}{\epsilon} \le \abs{u} \le \frac{2}{\epsilon}}, \quad  	\Theta_{\epsilon}(u) = \int_{0}^u \int_{0}^r \Theta_{\epsilon}^{''}(s) \ ds \ dr.
	\end{equation*}
  Let
		\begin{equation*}
			\psi_{\epsilon}(v)= \Theta_{\epsilon}^{'}(v) \abs{v}^{2 \alpha -1} , \quad
			\Psi_{\epsilon}(v)= \int_0^v \psi_{\epsilon}(\tilde{v}) \ d \tilde{v}.
	\end{equation*}
	Fix $\bar{t} < T$ and let $ \zeta_l \in C^{\infty}([0,T))$ with compact support in $[0,\bar{t}]$ such that $\zeta_l=1$ in $[0,\bar{t}- \frac{1}{l}]$.
 An approximation argument shows that $\varphi(t,x,v) = \zeta_l(t) \psi_{\epsilon}(v) $ is an appropriate test function for the Definition \ref{kinetic_solution}, since it has compact support and bounded derivatives.  With this choice of $\varphi$ in \eqref{kinetic_formulation_distribution_PME}, we have
	\begin{equation*} 
		\begin{split}
			& \int_0^T \int_{\mathbb{T}^d}  \Psi_{\epsilon}(u(t,x)) \zeta^{'}_l(t)  \ dx \ dt     + 	 \int_{\mathbb{T}^d}  \Psi_{\epsilon}(u_0(x))  \ dx  \\  &  = - \sum_{k = 1}^{\infty} \int_0^T \int_{\mathbb{T}^d} g_k(x,u(t,x)) \zeta_l(t) \Theta_{\epsilon}^{'}(u(t,x)) \abs{u(t,x)}^{2 \alpha -1}   \ dx \  d\beta_k(t) 	
		\\	 & +  \int_0^T \int_{\mathbb{T}^d} \int_{\mathbb{R}} \zeta_l(t)   (\epsilon^{-1}  \mathbbm{1}_{\abs{v} \le \epsilon} - \epsilon \mathbbm{1}_{\frac{1}{\epsilon} \le \abs{v} \le \frac{2}{\epsilon}} ) \abs{v}^{2 \alpha -1}    \ dn(t,x,v)   
			\\ & + (2 \alpha -1) \int_0^T \int_{\mathbb{T}^d} \int_{\mathbb{R}} \zeta_l(t)   \Theta_{\epsilon}^{'}(v) \abs{v}^{2 \alpha -2} \sgn(v)  \ dn(t,x,v)   
				\end{split}
		\end{equation*} 
		\begin{equation} \label{First_expression_apriori_bound}
		\begin{split}
			  &  - \frac{1}{2 } \int_0^T \int_{\mathbb{T}^d} G^2(x,u(t,x))  \zeta_l(t) (\epsilon^{-1}  \mathbbm{1}_{\abs{u(t,x)} \le \epsilon} - \epsilon \mathbbm{1}_{\frac{1}{\epsilon} \le \abs{u(t,x)} \le \frac{2}{\epsilon}}) \abs{u(t,x)}^{2 \alpha -1}   \ dx \ dt \\  &  - \frac{(2 \alpha -1)}{2} \int_0^T \int_{\mathbb{T}^d} G^2(x,u(t,x))  \zeta_l(t)  \Theta_{\epsilon}^{'}(u(t,x)) \abs{u(t,x)}^{2 \alpha -2} \sgn(u(t,x))  \ dx \ dt. 
		\end{split}
	\end{equation} 
We may take the limit $l \rightarrow \infty$ in \eqref{First_expression_apriori_bound}. We use the Lebesgue differentiation theorem for the first term on the left-hand side of \eqref{First_expression_apriori_bound}, It{\^o} isometry and boundedness of $\Theta_{\epsilon}^{'}$ for the first term on the right-hand side of \eqref{First_expression_apriori_bound} and dominated convergence theorem for the remaining terms. For almost every $\omega \in \Omega$, there exists a set of full Lebesgue measure so that for all $\bar{t}$, the following is satisfied
	\begin{equation} \label{First_expression_apriori_bound_after_limit_in_time}
	\begin{split}
		&  \int_{\mathbb{T}^d}  \Psi_{\epsilon}(u(\bar{t},x))   \ dx      + 	 \int_{\mathbb{T}^d}  \Psi_{\epsilon}(u_0(x))  \ dx  \\  &  = - \sum_{k = 1}^{\infty}  \int_0^{\bar{t}} \int_{\mathbb{T}^d} g_k(x,u(t,x)) \Theta_{\epsilon}^{'}(u(t,x)) \abs{u(t,x)}^{2 \alpha -1}   \ dx \  d\beta_k(t) 				
	\\	& +  \int_0^{\bar{t}} \int_{\mathbb{T}^d} \int_{\mathbb{R}}    (\epsilon^{-1}  \mathbbm{1}_{\abs{v} \le \epsilon} - \epsilon \mathbbm{1}_{\frac{1}{\epsilon} \le \abs{v} \le \frac{2}{\epsilon}}) \abs{v}^{2 \alpha -1}    \ dn(t,x,v)   
	 \\ & + (2 \alpha -1) \int_0^{\bar{t}} \int_{\mathbb{T}^d} \int_{\mathbb{R}}    \Theta_{\epsilon}^{'}(v) \abs{v}^{2 \alpha -2} \sgn(v)  \ dn(t,x,v)    \\  &  - \frac{1}{2 } \int_0^{\bar{t}} \int_{\mathbb{T}^d} G^2(x,u(t,x))  ( \epsilon^{-1}  \mathbbm{1}_{\abs{u(t,x)} \le \epsilon} - \epsilon \mathbbm{1}_{\frac{1}{\epsilon} \le \abs{u(t,x)} \le \frac{2}{\epsilon}}) \abs{u(t,x)}^{2 \alpha -1}   \ dx \ dt \\  &  - \frac{(2 \alpha -1)}{2} \int_0^{\bar{t}} \int_{\mathbb{T}^d} G^2(x,u(t,x))   \Theta_{\epsilon}^{'}(u(t,x)) \abs{u(t,x)}^{2 \alpha -2} \sgn(u(t,x))  \ dx \ dt. 
	\end{split}
\end{equation} 
We may take $p^{\text{th}}$ power, the essential supremum over  $ \bar{t} \in [0,T]$ and the expectation in \eqref{First_expression_apriori_bound_after_limit_in_time} to obtain, for $\tilde{C}=\tilde{C}(p)>0$,
	\begin{equation*} 
		\begin{split}
			& \mathbb{E}  \esssup_{\bar{t} \in [0,T]} \left(\int_{\mathbb{T}^d}     \Psi_{\epsilon}(u(\bar{t},x))   \ dx \right)^p 
				\\ & +   \mathbb{E} \esssup_{\bar{t} \in [0,T]} \left( \int_0^{\bar{t}} \int_{\mathbb{T}^d} \int_{\mathbb{R}}  (	\epsilon^{-1}  \mathbbm{1}_{\abs{v} \le \epsilon} - \epsilon \mathbbm{1}_{\frac{1}{\epsilon} \le \abs{v} \le \frac{2}{\epsilon}} ) \abs{v}^{2 \alpha -1}  \ dn(t,x,v) \right)^p
			\\ & +  (2 \alpha -1)^p \ \mathbb{E} \esssup_{\bar{t} \in [0,T]} \left( \int_0^{\bar{t}} \int_{\mathbb{T}^d} \int_{\mathbb{R}}  \Theta_{\epsilon}^{'}(v) \abs{v}^{2 \alpha -2}  \sgn(v) \ dn(t,x,v)   \right)^p
			\\  &  \le \tilde{C}  \mathbb{E}  \esssup_{\bar{t} \in [0,T]}  \abs{ \sum_{k = 1}^{\infty} \int_0^{\bar{t}} \int_{\mathbb{T}^d} g_k(x,u(t,x))  \Theta_{\epsilon}^{'}(u(t,x)) \abs{u(t,x)}^{2 \alpha -1}    \ dx \  d\beta_k(t)}^p
			\\	& + \frac{1}{2^p } \mathbb{E}  \esssup_{\bar{t} \in [0,T]} \abs{ \int_0^{\bar{t}} \int_{\mathbb{T}^d} G^2(x,u(t,x))  (	\epsilon^{-1}  \mathbbm{1}_{\abs{u} \le \epsilon} - \epsilon \mathbbm{1}_{\frac{1}{\epsilon} \le \abs{u} \le \frac{2}{\epsilon}} )  \abs{u(t,x)}^{2 \alpha -1}      \ dx \ dt}^p 
						 		\end{split}
		\end{equation*}	
		 	\begin{equation} \label{A_priori_bound_2alpha}
		\begin{split}
		&  +  \frac{(2 \alpha -1)^p}{2^p} \mathbb{E}   \esssup_{\bar{t} \in [0,T]} \abs{\int_0^{\bar{t}} \int_{\mathbb{T}^d} G^2(x,u(t,x))   |\Theta_{\epsilon}^{'}(u(t,x))| \abs{u(t,x)}^{2 \alpha -2}  \ dx \ dt}^p 	\\ &	  + 	\left( \int_{\mathbb{T}^d}  \Psi_{\epsilon}(u_0(x))     \ dx \right)^p.
				\end{split}
		\end{equation}	
	Using Assumption \ref{Assumption_diffusion_coefficient}, we have
	\begin{equation*}
		\begin{split}
			&  \frac{1}{2^p } \mathbb{E}  \esssup_{\bar{t} \in [0,T]} \abs{ \int_0^{\bar{t}} \int_{\mathbb{T}^d} G^2(x,u(t,x))  (	\epsilon^{-1}  \mathbbm{1}_{\abs{u(t,x)} \le \epsilon} - \epsilon \mathbbm{1}_{\frac{1}{\epsilon} \le \abs{u(t,x)} \le \frac{2}{\epsilon}} )  \abs{u(t,x)}^{2 \alpha -1}      \ dx \ dt}^p
			 \\ & \le \frac{D^p}{2 } \Bigg( \mathbb{E}  \abs{ \int_0^{T} \int_{\mathbb{T}^d} 	 \abs{u(t,x)}^{2 \alpha } \epsilon^{-1} \mathbbm{1}_{\abs{u(t,x)} \le \epsilon}   \ dx \ dt}^p  \\ & +  \mathbb{E}   \abs{  \int_0^{T} \int_{\mathbb{T}^d}  \abs{u(t,x)}^{4 \alpha -1 } \epsilon  \mathbbm{1}_{\frac{1}{\epsilon} \le \abs{u(t,x)} \le \frac{2}{\epsilon}}        \ dx \ dt}^p \Bigg)
			 \\ & \le \frac{D^p}{2 }  \Bigg( \epsilon^{p(2 \alpha -1)} T^p + 2^{p(2 \alpha -1)} \epsilon^{2p(1-\alpha)} \mathbb{E}   \abs{  \int_0^{T} \int_{\mathbb{T}^d}  \abs{u(t,x)}^{2 \alpha }   \mathbbm{1}_{\frac{1}{\epsilon} \le \abs{u} \le \frac{2}{\epsilon}}      \ dx \ dt}^p  \Bigg) 
		\\	& \le \frac{D^p}{2 } T^p  \left( \epsilon^{p(2 \alpha -1)}  + 2^{p(2 \alpha -1)} \epsilon^{2p(1-\alpha)}  \mathbb{E}  \esssup_{t \in [0,T]} \| u(t) \|^{2 \alpha p}_{L^{2 \alpha}}   \right).
		\end{split}
	\end{equation*}
	Since $ 	\Theta_{\epsilon}^{'}(u(t,x)) \le 1$ and using again Assumption \ref{Assumption_diffusion_coefficient}, we have
	\begin{equation*}
		\begin{split}
			&	\frac{(2 \alpha -1)^p}{2^p} \mathbb{E}   \esssup_{\bar{t} \in [0,T]}  \abs{\int_0^{\bar{t}} \int_{\mathbb{T}^d} G^2(x,u(t,x))   |\Theta_{\epsilon}^{'}(u(t,x))| \abs{u(t,x)}^{2 \alpha -2}  \ dx \ dt}^p \\ & \le  	\frac{D^p(2 \alpha -1)^p}{2} \Bigg( \mathbb{E}  \abs{ \int_0^{T} \int_{\mathbb{T}^d}  \abs{u(t,x)}^{2 \alpha-1 }  \mathbbm{1}_{\abs{u} \le 1} \ dx \ dt }^p  \\ & +  \mathbb{E}   \abs{ \int_0^{T} \int_{\mathbb{T}^d}    \abs{u(t,x)}^{4 \alpha - 2} \mathbbm{1}_{\abs{u} \ge 1}    \ dx \ dt }^p \Bigg)
		\\	& \le  	\frac{D^p(2 \alpha -1)^p}{2} \left( T^p + \mathbb{E}  \abs{ \int_0^{T} \int_{\mathbb{T}^d} \abs{u(t,x)}^{2 \alpha} \mathbbm{1}_{\abs{u} \ge 1} \ dx \ dt }^p \right)
		\\	 & \le 	\frac{D^p(2 \alpha -1)^p}{2} T^p \left( 1 + \mathbb{E}  \esssup_{t \in [0,T]} \| u(t) \|^{2 \alpha p}_{L^{2 \alpha}}  \right).
		\end{split}		
	\end{equation*}
 We may estimate the stochastic integral using the Burkholder--Davis--Gundy inequality and Assumption \ref{Assumption_diffusion_coefficient} as follows  
	\begin{equation*} 
		\begin{split}
			&  \mathbb{E} \esssup_{\bar{t} \in [0,T]}  \abs{\sum_{k = 1}^{\infty} \int_0^{\bar{t}} \int_{\mathbb{T}^d} g_k(x,u(t,x))  \Theta_{\epsilon}^{'}(u(t,x)) \abs{u(t,x)}^{2 \alpha -1}  \ dx \  d\beta_k(t)}^p 
					\end{split}
		\end{equation*}
		\begin{equation*} 
		\begin{split}
			 & \le \tilde{C} \mathbb{E} \left(  \int_0^{T} \sum_{k=1}^{\infty} \abs{ \int_{\mathbb{T}^d} \abs{g_k(x,u(t,x))}  \abs{u(t,x)}^{2 \alpha -1}  \ dx   }^2 \ dt \right)^{\frac{p}{2}}
					\\ & \le \tilde{C} (2D)^{\frac{p}{2}} \mathbb{E} \left(  \int_0^{T} \left( \abs{ \int_{\mathbb{T}^d} \abs{u(t,x)}^{2 \alpha - \frac{1}{2}} \mathbbm{1}_{\abs{u} \le 1}  \ dx   }^2 + \abs{ \int_{\mathbb{T}^d}  \abs{u(t,x)}^{3 \alpha -1}  \mathbbm{1}_{\abs{u} \ge 1}   \ dx   }^2 \right) dt \right)^{\frac{p}{2}}
		\\	 & \le \tilde{C} (2D)^{\frac{p}{2}} \left(  T^{\frac{p}{2}} +  \mathbb{E} \left( \int_0^{T}  \abs{ \int_{\mathbb{T}^d}  \abs{u(t,x)}^{2 \alpha }  \mathbbm{1}_{\abs{u} \ge 1}   \ dx   }^2 \ dt \right)^{\frac{p}{2}} \right)
		\\ & \le \tilde{C} (2D)^{\frac{p}{2}} \left(  T^{\frac{p}{2}} +  \mathbb{E} \left( \int_0^{T}  \| u(t,\cdot) \|_{L^{2 \alpha}}^{4 \alpha} \ dt \right)^{\frac{p}{2}} \right)
			\\ & \le \tilde{C} (2D)^{\frac{p}{2}} T^{\frac{p}{2}} \left(  1 +   \mathbb{E} \esssup_{t \in [0,T]} \| u(t) \|_{L^{2 \alpha}}^{2 \alpha p }  \right).
		\end{split}
	\end{equation*}	
We define 
\begin{equation*}
	\bar{C}= \tilde{C} (2D)^{\frac{p}{2}}  T^{\frac{p}{2}} + \frac{D^p}{2 } \mathbbm{1}_{ \{\alpha =\frac{1}{2} \}} T^p +	\frac{D^p(2 \alpha -1)^p}{2}  T^p.
\end{equation*}
We let $ \epsilon \rightarrow 0$ in \eqref{A_priori_bound_2alpha} using Fatou's lemma and obtain for $  \underline{C} = \underline{C}(\alpha) >0$
	\begin{equation} \label{after_Fatou_lemma_a_priori_bound}
		\begin{split}
				\mathbb{E} \esssup_{t \in[0, T]} \| u(t) \|_{L^{2 \alpha}}^{2 \alpha p} &  \le  	\bar{C}  +  \tilde{C} (2D)^{\frac{p}{2}} T^{\frac{p}{2}}   \mathbb{E} \esssup_{t \in [0,T]} \| u(t) \|_{L^{2 \alpha}}^{2 \alpha p }
\\	 & +  \frac{D^p 2^{p(2 \alpha -1)}}{2 }  \mathbbm{1}_{ \{\alpha =1 \}} T^p \mathbb{E}  \esssup_{t \in [0,T]} \| u(t) \|^{2 \alpha p}_{L^{2 \alpha}}  
\\ &  +  \frac{D^p(2 \alpha -1)^p}{2}  T^p \mathbb{E}  \esssup_{t \in [0,T]} \| u(t) \|^{2 \alpha p}_{L^{2 \alpha}}  + \underline{C} \|u_0 \|_{L^{2 \alpha}}^{2 \alpha p}.
		\end{split}
	\end{equation}
We choose $ T$ small enough such that
\begin{equation} \label{T_small}
	\tilde{C} (2D)^{\frac{p}{2}} T^{\frac{p}{2}} + \frac{D^p}{2 } 2^{p(2 \alpha -1)} \mathbbm{1}_{ \{\alpha =1 \}} T^p +  \frac{D^p(2 \alpha -1)^p}{2}  T^p \le \frac{1}{2}.
\end{equation}
We denote $T^{\star}$ such $T$ satisfying \eqref{T_small}. Then, \eqref{after_Fatou_lemma_a_priori_bound} becomes
\begin{equation*}
	\mathbb{E} \esssup_{t \in[0, T^{\star}]} \| u(t) \|_{L^{2 \alpha}}^{2 \alpha p}  \le C(1+  \|u_0 \|_{L^{2 \alpha}}^{2 \alpha p}).
\end{equation*}
Iterating the argument, we arrive at 
\begin{equation*}
\mathbb{E} \esssup_{t \in[0, T]} \| u(t) \|_{L^{2 \alpha}}^{2 \alpha p}  \le C^{\lceil \frac{T}{T^{\star}} \rceil }(1+  \|u_0 \|_{L^{2 \alpha}}^{2 \alpha p}),
\end{equation*}
where $ \lceil \cdot \rceil $ denotes the ceiling function.
This leads to \eqref{Estimate_L1_energy}.
\end{proof}

	\begin{lem} \label{control_kinetic_measure_v_large_new}
	Let $\gamma \in (0,1)$, $\tilde{\gamma} \in (1, \infty)$ and $ \alpha \in [\frac{1}{2}, 1]$. Assume that $g_k$ satisfies  Assumption \ref{Assumption_diffusion_coefficient} for this value of $\alpha$. Let $u$ be a solution to \eqref{degenerate_parabolic_hyperbolic_SPDE} in the sense of Definition \ref{kinetic_solution} with $u_0 \in L^{2 \alpha}(\mathbb{T}^d)$. Then, there is a constant $C_1=C_1(T,\gamma,D) > 0$ such that
	\begin{equation} \label{lemma_existence_singular_moments_new}
		\begin{split}
			& 				\mathbb{E} \int_0^T \int_{\mathbb{T}^d}	\int_{\abs{v} \le 2}   \abs{v}^{- \gamma}  dn(t,x,v) \le C_1 (\|u_0 \|_{L^1} +1),
		\end{split}
	\end{equation}
	and a constant $C_2=C_2(T, \alpha, \tilde{\gamma},D) > 0$ such that
	\begin{equation} \label{lemma_existence_singular_moments_new_v_large}
		\begin{split}
			& 		\mathbb{E} \int_0^T \int_{\mathbb{T}^d} \int_{\abs{v} \ge 2}    \abs{v}^{2(\alpha- \tilde{\gamma})} dn(t,x,v)  \le C_2  (\|u_0 \|_{L^{2 \alpha}}^{2 \alpha} +1 ).
		\end{split}
	\end{equation}	
\end{lem}
\begin{proof}
	  All the test functions that we construct in this proof have compact support  and bounded derivatives. A suitable approximation argument shows that they are appropriate test functions for the Definition \ref{kinetic_solution}.
	
	\noindent We derive \eqref{lemma_existence_singular_moments_new}  in two steps. 
	
	\noindent	\textbf{Step 1a} \ For fixed $ 0 <\epsilon  \le 3$, let  $\Theta_{\epsilon}$ be defined as follows
	\begin{equation*}
		\Theta_{\epsilon}(u) = 
		\frac{u^2}{2 \epsilon}  \mathbbm{1}_{\abs{u} \le \epsilon} + \left(\abs{u} - \frac{\epsilon}{2} \right) \mathbbm{1}_{\abs{u} > \epsilon} .
	\end{equation*}
	Let $ \zeta_l \in C^{\infty}_c([0,T))$ such that $0 \le \zeta_l(t) \le 1$, $\zeta_l \equiv 1$ if $t \in [0,T- \frac{2}{l}]$, $\zeta_l \equiv 0$ if $t >T- \frac{1}{l}$.  Choosing $\varphi(t,x,v) = \zeta_l(t) \Theta_{\epsilon}^{'}(v)$ and taking the expectation in \eqref{kinetic_formulation_distribution_PME}, we get
	\begin{equation} \label{energy_estimate_v_small_singular}
		\begin{split}
			& \mathbb{E} \int_0^T \int_{\mathbb{T}^d}  \zeta_l^{'}(t) \Theta_{\epsilon}(u(t,x))  \ dx \ dt  +  \int_{\mathbb{T}^d}   \Theta_{\epsilon}(u_0(x))   dx   \\ & =  \frac{1}{\epsilon} \mathbb{E}   \int_{A_{\epsilon}}   \zeta_l(t)  dn(t,x,v)    -  \frac{1}{2} \mathbb{E}  \int_0^T \int_{\mathbb{T}^d} G^2(x,u(t,x))  \zeta_l(t)  \Theta_{\epsilon}^{''}(u(t,x)) \ dx \ dt,
		\end{split}
	\end{equation}	
	where $A_{\epsilon} = [0,T] \times \mathbb{T}^d \times [- \epsilon, \epsilon]$. Due to the Assumption \ref{Assumption_diffusion_coefficient}, 
	\begin{equation*}
			\begin{split}
		\frac{1}{2} G^2(x,u)  \zeta_l(t)  \Theta_{\epsilon}^{''}(u) \le \frac{3^{2 \alpha -1 } D }{2 \epsilon}  \abs{u} \mathbbm{1}_{\abs{u} \le \epsilon} \le \frac{3^{2 \alpha -1 } D }{2 } .
		\end{split}
	\end{equation*}
We have
	\begin{equation} \label{energy_estimate_v_small_singular_no_cutoff}
		\begin{split}
			\frac{1}{\epsilon} \mathbb{E}    |   n ([0,T] \times \mathbb{T}^d \times [- \epsilon,\epsilon ]) |   &\le     \int_{\mathbb{T}^d}   \abs{u_0(x)}  dx + 3^{2\alpha -1} \frac{ DT}{2}.
		\end{split}
	\end{equation}
	\noindent	\textbf{Step 2a} \ Let $  \zeta_l$ be as in Step 1a, $ 0 < \varepsilon < 2$ and  
	\begin{equation*}
		\eta^{\varepsilon} (v) = \frac{\abs{v}}{\varepsilon} \mathbbm{1}_{\abs{v}\le \varepsilon} + \mathbbm{1}_{\varepsilon < \abs{v} \le 2} + (3- \abs{v})\mathbbm{1}_{ 2 < \abs{v} \le 3 }. 
	\end{equation*}
	Let 
	\begin{equation*}
		\psi^{\varepsilon}(v)  =  \eta^{\varepsilon}(v) \abs{v}^{1-\gamma} \sgn(v),
	\end{equation*}
	and $\Psi^{\varepsilon} (v) = \int_0^v 	\psi^{\varepsilon}(\tilde{v}) d\tilde{v}$.
	We may choose $\varphi(t,x,v) = \zeta_l(t) \psi^{\varepsilon}(v)  $  and take the expectation in \eqref{kinetic_formulation_distribution_PME} to have
	\begin{equation} \label{kinetic_form_energy_estimate}
		\begin{split}
			&	 \mathbb{E} \int_0^T \int_{\mathbb{T}^d}  \zeta_l^{'}(t) \Psi^{\varepsilon}(u(t,x))     \ dx \ dt     + 		\int_{\mathbb{T}^d} \int_{\mathbb{R}} \chi(0,x,v) 	\psi^{\varepsilon}(v) \ dv \ dx	
			\\ & =  \mathbb{E} \int_0^T \int_{\mathbb{T}^d} \int_{\mathbb{R}} \zeta_l(t)    (\eta^{\varepsilon})^{'}(v) \abs{v}^{1-\gamma} \sgn(v) \ dn(t,x,v) 	\\ & +(1-\gamma)  \mathbb{E} \int_0^T \int_{\mathbb{T}^d} \int_{\mathbb{R}} \zeta_l(t)    \eta^{\varepsilon}(v) \abs{v}^{-\gamma}  \ dn(t,x,v)   \\  &  -   \frac{1}{2} \mathbb{E} \int_0^T \int_{\mathbb{T}^d} G^2(x,u(t,x)) \zeta_l(t)  (\eta^{\varepsilon})^{'}(u(t,x)) \abs{u(t,x)}^{1-\gamma} \sgn(u(t,x))  \ dx \ dt
			  \\  &  -   \frac{(1-\gamma)}{2} \mathbb{E} \int_0^T \int_{\mathbb{T}^d} G^2(x,u(t,x)) \zeta_l(t)  \eta^{\varepsilon}(u(t,x)) \abs{u(t,x)}^{-\gamma}   \ dx \ dt.
		\end{split}
	\end{equation}	
	Since $ \eta^{\varepsilon}(v) \le \mathbbm{1}_{\abs{v} \le 3}$,  we have 
	\begin{equation*}
		\begin{split}
			\int_{\mathbb{T}^d} \int_{\mathbb{R}} \chi(0,x,v) 	\psi^{\varepsilon}(v) \ dv \ dx	 \le 	\int_{\mathbb{T}^d} \int_{\mathbb{R}}  \abs{\chi}(0,x,v) \mathbbm{1}_{\abs{v} \le 3} \abs{v}^{1-\gamma} \ dv \ dx \le 2 \cdot 3^{2- \gamma}.
		\end{split}
	\end{equation*}	
	Using the estimate \eqref{energy_estimate_v_small_singular_no_cutoff} in Step 1a, we have
	\begin{equation*}
		\begin{split}
&   \mathbb{E} \int_0^T \int_{\mathbb{T}^d} \int_{\mathbb{R}} \zeta_l(t)    (\eta^{\varepsilon})^{'}(v) \abs{v}^{1-\gamma} \sgn(v) dn(t,x,v)  \\ & \le  \mathbb{E} \int_0^T \int_{\mathbb{T}^d} \int_{\mathbb{R}}    \mathbbm{1}_{\abs{v} \le \varepsilon} \frac{\abs{v}^{1-\gamma}}{\varepsilon} \ dn(t,x,v) +  \mathbb{E} \int_0^T \int_{\mathbb{T}^d} \int_{\mathbb{R}}     \mathbbm{1}_{2 < \abs{v} \le 3} \ dn(t,x,v)
			\\ & \le \varepsilon^{1-\gamma} \left( \frac{\mathbb{E}    |   n ([0,T] \times \mathbb{T}^d \times [- \varepsilon,\varepsilon ]) |}{\varepsilon} \right) +  \mathbb{E}    |   n ([0,T] \times \mathbb{T}^d \times [- 3,3 ]) |
			\\ & \le \left( \varepsilon^{1-\gamma} + 3 \right) \left( \| u_0 \|_{L^1} + 3^{2\alpha -1} \frac{DT}{2} \right).
		\end{split}
	\end{equation*}
	By Assumption \ref{Assumption_diffusion_coefficient}, we have
	
	\begin{equation*}
		\begin{split}
& \frac{1}{2} \mathbb{E} \int_0^T \int_{\mathbb{T}^d} G^2(x,u(t,x))  \zeta_l(t)  (\eta^{\varepsilon})^{'}(u(t,x)) \abs{u(t,x)}^{1-\gamma} \sgn(u(t,x)) \ dx \ dt 
\\ & \le \frac{3^{2 \alpha -1} D}{2} \left(	\mathbb{E} \int_0^T \int_{\mathbb{T}^d}     (\mathbbm{1}_{\abs{u(t,x)} \le \varepsilon}  \frac{\abs{u(t,x)}^{2-\gamma}}{\varepsilon} + \abs{u(t,x)}^{2-\gamma} \mathbbm{1}_{2 < \abs{u(t,x)} \le 3} ) \ dx \ dt  \right)
			\\ &  \le \frac{3^{2 \alpha -1} D}{2} T (\varepsilon^{1- \gamma}+ 3^{2 - \gamma}).
		\end{split}
	\end{equation*}
	Since $ \eta^{\varepsilon}(u(t,x)) \le \mathbbm{1}_{\abs{u(t,x)} \le 3}$ and using Assumption \ref{Assumption_diffusion_coefficient}, we have
	\begin{equation*}
		\begin{split}
			& \frac{(1-\gamma)}{2}	\mathbb{E} \int_0^T \int_{\mathbb{T}^d}  G^2(x,u(t,x))  \zeta_l(t)   \eta^{\varepsilon}(u(t,x))
			\abs{u(t,x)}^{-\gamma}  \ dx \ dt  \\	& \le  \frac{(1-\gamma) 3^{2 \alpha - 1} D}{2} 	\mathbb{E} \int_0^T \int_{\mathbb{T}^d}    \abs{u(t,x)}^{1- \gamma} \mathbbm{1}_{  \abs{u(t,x)} \le 3}    \ dx \ dt  \\ &  \le  \frac{(1-\gamma) 3^{2 \alpha - \gamma} D}{2} T.
		\end{split}
	\end{equation*}
We have
	\begin{equation*}
		\begin{split}
					    &  	(1-\gamma)  \mathbb{E} \int_0^T \int_{\mathbb{T}^d} \int_{\mathbb{R}}    \eta^{\varepsilon}(v) \abs{v}^{-\gamma}  \ dn(t,x,v)  \\ & \le  2 \cdot 3^{2- \gamma}  + \left( \varepsilon^{1-\gamma} + 3 \right) \left( \| u_0 \|_{L^1} + 3^{2\alpha -1} \frac{DT}{2} \right)  + \frac{3^{2 \alpha -1} D }{2} T (\varepsilon^{1- \gamma}+ 3^{2 - \gamma}) \\ & + \frac{(1-\gamma) 3^{2 \alpha - \gamma} D}{2}T.
		\end{split}
	\end{equation*}
Letting  $ \varepsilon \rightarrow 0$ and using Fatou's lemma, we arrive at \eqref{lemma_existence_singular_moments_new}.
	
\noindent	We derive \eqref{lemma_existence_singular_moments_new_v_large}  in two steps.
	
	\noindent		\textbf{Step 1b} \ Let $ \zeta_l$ be as in Step 1a. For fixed $   h \ge 1$, let
	\begin{equation*}
		\Theta_{h}^{''}(u) = \frac{1}{h} \mathbbm{1}_{h \le \abs{u} \le 2h}, \quad \quad \Theta_{h}(u) = \int_0^u \int_0^r 	\Theta_{h}^{''}(s) \ ds \ dr.
	\end{equation*}
	Choosing $\varphi(t,x,v) = \zeta_l(t) \Theta_{h}^{'}(v)$ and taking the expectation in \eqref{kinetic_formulation_distribution_PME}, we get
	\begin{equation} \label{energy_estimate_v_large_singular}
		\begin{split}
			& \mathbb{E} \int_0^T \int_{\mathbb{T}^d}  \zeta_l^{'}(t) \Theta_{h}(u(t,x))  \ dx \ dt  + 	\int_{\mathbb{T}^d} \int_{\mathbb{R}} \chi(0,x,v) \Theta_h^{'}(v) \ dv \ dx	  \\ & =  \frac{1}{h} \mathbb{E}   \int_{A_{h}}   \zeta_l(t)  dn(t,x,v)          -  \frac{1}{2} \mathbb{E}  \int_0^T \int_{\mathbb{T}^d} G^2(x,u(t,x))  \zeta_l(t)  \Theta_{h}^{''}(u(t,x)) \ dx \ dt,
		\end{split}
	\end{equation}	
	where $A_{h}= [0,T] \times \mathbb{T}^d \times ([-2h,-h]\cup [h, 2h])$. Since $ \Theta_h^{'}(v) \le \mathbbm{1}_{\abs{v} \ge 1}$, we have
	\begin{equation*}
		\begin{split}
			\int_{\mathbb{T}^d} \int_{\mathbb{R}} \chi(0,x,v) \Theta_h^{'}(v) \ dv \ dx	 \le 	\int_{\mathbb{T}^d} \int_{\mathbb{R}}  \abs{\chi}(0,x,v) \mathbbm{1}_{\abs{v} \ge 1} \ dv \ dx \le \|u_0 \|_{L^1}.
		\end{split}
	\end{equation*}	
	Recall that $ \alpha \in [\frac{1}{2},1]$. Using Assumption \ref{Assumption_diffusion_coefficient} and Lemma \ref{control_apriori_estimate_new}, we have for $C=C(T,\alpha, D) > 0$
	\begin{equation*} 
		\begin{split}
			&	 \frac{1}{2} \mathbb{E}  \int_0^T \int_{\mathbb{T}^d} G^2(x,u(t,x))  \zeta_l(t)  \Theta_{h}^{''}(u(t,x)) \ dx \ dt \\ & \le 	 \frac{D}{2} \mathbb{E}  \int_0^T \int_{\mathbb{T}^d} \frac{\abs{u(t,x)}^{2 \alpha}}{h} \mathbbm{1}_{h \le \abs{u(t,x)} \le 2h}   \ dx \ dt  \\ & \le 2^{2 \alpha -2} D
				 h^{2 \alpha -2} \mathbb{E}  \int_0^T \| u(t,\cdot) \|_{L^1_x}    \ dt
				 \\ & \le  C h^{2 \alpha -2}( \|u_0 \|_{L^1} +1 ).
		\end{split}
	\end{equation*}
We obtain
	\begin{equation} \label{energy_estimate_v_large_singular_no_cutoff}
		\begin{split}
				\frac{1}{h} \mathbb{E}    |   n ([0,T] \times \mathbb{T}^d \times ([-2h,-h]\cup [h, 2h]) ) | 
	 \le \|u_0 \|_{L^1} + C h^{2 \alpha -2}( \|u_0 \|_{L^1} +1 ).
		\end{split}
	\end{equation}
Furthermore, we may multiply \eqref{energy_estimate_v_large_singular} by $h$  and observe that
	\begin{equation*}
		h \Theta_h(u_0) \le h^{2 - 2 \alpha } \abs{u_0}^{2 \alpha }.
	\end{equation*}
In addition, using Assumption \ref{Assumption_diffusion_coefficient} and Lemma \ref{control_apriori_estimate_new}
	\begin{equation*}
		\begin{split}
& \frac{1}{2} \mathbb{E}  \int_0^T \int_{\mathbb{T}^d} h \ G^2(x,u(t,x))  \zeta_l(t)  \Theta_{h}^{''}(u(t,x)) \ dx \ dt \\ &	 \le \frac{D}{2} \mathbb{E}  \int_0^T \int_{\mathbb{T}^d} \abs{u}^{2 \alpha } \mathbbm{1}_{h \le u \le 2h}  \ dx \ dt
\\ & \le \frac{C}{2}(\|u_0 \|_{L^{2 \alpha}}^{2 \alpha} +1).
\end{split}
	\end{equation*}
Similarly as done for the estimate \eqref{energy_estimate_v_large_singular_no_cutoff}, we arrive at
	\begin{equation} \label{energy_estimate_v_large_singular_no_cutoff_h}
		\begin{split}
		 \mathbb{E}    |   n ([0,T] \times \mathbb{T}^d \times([-2h,-h]\cup [h, 2h]) ) |   \le h^{2 - 2 \alpha } \|u_0 \|_{L^{2 \alpha}}^{2 \alpha} + \frac{C}{2} (\|u_0 \|_{L^{2 \alpha}}^{2 \alpha} +1).
		\end{split}
	\end{equation}	
	
	\noindent	\textbf{Step 2b} \ Let $  \zeta_l$ be as in Step 1a, $ \varkappa > 2$ and  
	\begin{equation*}
		\eta^{\varkappa } (v) = (\abs{v} -1) \mathbbm{1}_{1 \le \abs{v}\le 2} + \mathbbm{1}_{2 < \abs{v} \le \varkappa  } + \left( 2 - \frac{\abs{v}}{\varkappa} \right) \mathbbm{1}_{\varkappa < \abs{v} \le 2 \varkappa}. 
	\end{equation*}
	Let 
	\begin{equation*}
		\psi^{\varkappa}(v)  =  \eta^{\varkappa}(v) \abs{v}^{2(\alpha - \tilde{\gamma})+1} \sgn(v),
	\end{equation*}
	and let $\Psi^{\varkappa} (v) = \int_0^{v} 	\psi^{\varkappa}(\tilde{v}) d\tilde{v}$.
	We may choose $\varphi(t,x,v) = \zeta_l(t) \psi^{\varkappa}(v)  $  and take the expectation in \eqref{kinetic_formulation_distribution_PME} to have
	
	\begin{equation} \label{kinetic_form_energy_estimate_v_large}
		\begin{split}
			&	 \mathbb{E} \int_0^T \int_{\mathbb{T}^d}  \zeta_l^{'}(t) \Psi^{\varkappa}(u(t,x))     \ dx \ dt     + 	\int_{\mathbb{T}^d} \int_{\mathbb{R}} \chi(0,x,v)  \eta^{\varkappa}(v) \abs{v}^{2(\alpha - \tilde{\gamma})+1} \sgn(v) \ dv \ dx
\\	& =  \mathbb{E} \int_0^T \int_{\mathbb{T}^d} \int_{\mathbb{R}} \zeta_l(t)    (\eta^{\varkappa})^{'}(v) \abs{v}^{2(\alpha - \tilde{\gamma})+1} \sgn(v) \ dn(t,x,v)  \\ & + (2(\alpha - \tilde{\gamma}) +1) \mathbb{E} \int_0^T \int_{\mathbb{T}^d} \int_{\mathbb{R}} \zeta_l(t)    \eta^{\varkappa}(v) \abs{v}^{2(\alpha - \tilde{\gamma})}  \ dn(t,x,v)   \\  &  -   \frac{1}{2} \mathbb{E} \int_0^T \int_{\mathbb{T}^d} G^2(x,u(t,x)) \zeta_l(t)  (\eta^{\varkappa})^{'}(u(t,x)) \abs{u(t,x)}^{2(\alpha - \tilde{\gamma})+1} \sgn(u(t,x)) \ dx \ dt
			  \\  &  -   \frac{2(\alpha - \tilde{\gamma})+1}{2} \mathbb{E} \int_0^T \int_{\mathbb{T}^d} G^2(x,u(t,x)) \zeta_l(t)  \eta^{\varkappa}(u(t,x)) \abs{u(t,x)}^{2(\alpha - \tilde{\gamma})}  \ dx \ dt.
		\end{split}
	\end{equation}
	Since $  \eta^{\varkappa}(v) \le \mathbbm{1}_{\abs{v} \ge 1}$, we have
	\begin{equation*}
		\begin{split}
			\int_{\mathbb{T}^d} \int_{\mathbb{R}} \chi(0,x,v)  \eta^{\varkappa}(v) \abs{v}^{2(\alpha - \tilde{\gamma})+1} \sgn(v) \ dv \ dx	 \le \|u_0 \|_{L^{2 \alpha}}^{2 \alpha}.
		\end{split}
	\end{equation*}	
	Using estimates \eqref{energy_estimate_v_large_singular_no_cutoff} and \eqref{energy_estimate_v_large_singular_no_cutoff_h} in Step 1b, we have
	\begin{equation*}
		\begin{split}
&	 \mathbb{E} \int_0^T \int_{\mathbb{T}^d} \int_{\mathbb{R}} \zeta_l(t)    (\eta^{\varkappa})^{'}(v) \abs{v}^{2(\alpha - \tilde{\gamma})+1} \sgn(v) \ dn(t,x,v) 
	\\ &	\le	 \mathbb{E} \int_0^T \int_{\mathbb{T}^d} \int_{\mathbb{R}}  \left(  \mathbbm{1}_{1 \le \abs{v}\le 2}   + \frac{\mathbbm{1}_{\varkappa \le \abs{v} \le 2 \varkappa}}{\varkappa}   \right) \abs{v}^{2(\alpha - \tilde{\gamma})+1} dn(t,x,v) \\ & \le 2^{2(\alpha - \tilde{\gamma})+1} (\mathbb{E}    |   n ([0,T] \times \mathbb{T}^d \times ([-2,-1 ] \cup [1,2 ])) |   \\ & +  \varkappa^{2(\alpha - \tilde{\gamma})} \mathbb{E}    |   n ([0,T] \times \mathbb{T}^d \times ([-2 \varkappa, - \varkappa ] \cup [\varkappa,2 \varkappa ])) |)
			\\ & \le 2^{2(\alpha - \tilde{\gamma})+1}   \left[(1+C) \|u_0 \|_{L^1}  +C  \right] + \varkappa^{2(1 - \tilde{\gamma})} \|u_0 \|_{L^{2 \alpha}}^{2 \alpha}  +   \frac{\varkappa^{2(\alpha - \tilde{\gamma})} C}{2} ( \|u_0 \|_{L^{2 \alpha}}^{2 \alpha} +1). 
		\end{split}
	\end{equation*}
By Assumption \ref{Assumption_diffusion_coefficient} and Lemma \ref{control_apriori_estimate_new}, we have
	\begin{equation*}
		\begin{split}
& \frac{1}{2} \mathbb{E} \int_0^T \int_{\mathbb{T}^d} G^2(x,u(t,x)) \zeta_l(t)  (\eta^{\varkappa})^{'}(u(t,x)) \abs{u(t,x)}^{2(\alpha - \tilde{\gamma})+1} \sgn(u(t,x)) \ dx \ dt	\\ & \le  \frac{1}{2}	\mathbb{E} \int_0^T \int_{\mathbb{T}^d} G^2(x,u(t,x))   \left(  \mathbbm{1}_{1 \le \abs{u(t,x)}\le 2}   + \frac{\mathbbm{1}_{\varkappa \le \abs{u(t,x)} \le 2 \varkappa}}{\varkappa}   \right) \abs{u(t,x)}^{2(\alpha - \tilde{\gamma})+1} \ dx \ dt  
\\ & \le \frac{D}{2} \Bigg(	\mathbb{E} \int_0^T \int_{\mathbb{T}^d}   \mathbbm{1}_{1 \le \abs{u(t,x)}\le 2} \abs{u(t,x)}^{2 \alpha + 2(\alpha - \tilde{\gamma})+1} \ dx \ dt  \\ & +	\mathbb{E} \int_0^T \int_{\mathbb{T}^d}      \mathbbm{1}_{\varkappa \le \abs{u(t,x)} \le 2 \varkappa} \frac{\abs{u(t,x)}^{2 \alpha + 2(\alpha - \tilde{\gamma})+1}}{\varkappa} \ dx \ dt \Bigg)
 \\ &  \le D  \left( 2^{2 \alpha + 2(\alpha - \tilde{\gamma})} T +  \mathbb{E} \int_0^T \int_{\mathbb{T}^d}    \mathbbm{1}_{\varkappa \le \abs{u(t,x)} \le 2 \varkappa} \abs{u(t,x)}^{2 \alpha + 2(\alpha - \tilde{\gamma})}  \ dx \ dt  \right)
\\ &  \le D  \left( 2^{2 \alpha + 2(\alpha - \tilde{\gamma})} T +  \mathbb{E} \int_0^T    \| u(t,\cdot) \|^{2 \alpha}_{L^{2 \alpha}_x} \ dt  \right)
\\ & \le   2^{2 \alpha + 2(\alpha - \tilde{\gamma})} TD  + C (  \|u_0 \|_{L^{2 \alpha}}^{2 \alpha } +1)  .
		\end{split}
	\end{equation*}
	Since $ \eta^{\varkappa}(u(t,x)) \le \mathbbm{1}_{\abs{u(t,x)} \ge 1}$ and using Assumption \ref{Assumption_diffusion_coefficient}, Lemma \ref{control_apriori_estimate_new}, we have
	\begin{equation*}
		\begin{split}
		& \frac{2(\alpha - \tilde{\gamma})+1}{2}	\mathbb{E} \int_0^T \int_{\mathbb{T}^d} G^2(x,u(t,x)) \zeta_l(t)  \eta^{\varkappa}(u(t,x)) \abs{u(t,x)}^{2(\alpha - \tilde{\gamma})}  \ dx \ dt
 \\	& \le \frac{(2(\alpha - \tilde{\gamma})+1)D }{2}	\mathbb{E} \int_0^T \int_{\mathbb{T}^d}   \mathbbm{1}_{\abs{u(t,x)} \ge 1}  
\abs{u(t,x)}^{2\alpha + 2(\alpha - \tilde{\gamma})}   \ dx \ dt
 \\	& \le \frac{(2(\alpha - \tilde{\gamma})+1)D }{2}\mathbb{E} \int_0^T    \| u(t,\cdot) \|^{2 \alpha}_{L^{2 \alpha}_x} \ dt
 \\ & \le  \frac{C (2(\alpha - \tilde{\gamma})+1) }{2}  (  \|u_0 \|_{L^{2 \alpha}}^{2 \alpha } +1) .
		\end{split}
	\end{equation*}
We have
		\begin{equation} \label{Step2b_almost_conclusion}
		\begin{split}
		& (2(\alpha - \tilde{\gamma}) +1) \mathbb{E} \int_0^T \int_{\mathbb{T}^d} \int_{\mathbb{R}}   \eta^{\varkappa}(v) \abs{v}^{2(\alpha - \tilde{\gamma})}  \ dn(t,x,v) \\ & \le \|u_0 \|_{L^{2 \alpha}}^{2 \alpha} +  2^{2(\alpha - \tilde{\gamma})+1}   \left[(1+C) \|u_0 \|_{L^1}  +C \right] + \varkappa^{2(1 - \tilde{\gamma})} \|u_0 \|_{L^{2 \alpha}}^{2 \alpha}  \\ & +  \frac{ \varkappa^{2(\alpha - \tilde{\gamma})} C}{2} ( \|u_0 \|_{L^{2 \alpha}}^{2 \alpha} +1)  + 2^{2 \alpha + 2(\alpha - \tilde{\gamma})} TD  + C (  \|u_0 \|_{L^{2 \alpha}}^{2 \alpha } +1)   
			\\ & +   \frac{C(2(\alpha - \tilde{\gamma})+1) }{2}  (  \|u_0 \|_{L^{2 \alpha}}^{2 \alpha } +1).
				\end{split}
		\end{equation}
	Letting  $ \varkappa \rightarrow \infty$ and using Fatou's lemma in \eqref{Step2b_almost_conclusion}, we arrive at
	\begin{equation*}
			\begin{split}
		&    (2(\alpha - \tilde{\gamma}) +1) \mathbb{E} \int_0^T \int_{\mathbb{T}^d} \int_{\mathbb{R}}    \abs{v}^{2(\alpha - \tilde{\gamma})}  \ dn(t,x,v)
	\\	 &  \le 	\|u_0 \|_{L^{2 \alpha}}^{2 \alpha} + 2^{2(\alpha - \tilde{\gamma})+1} [(1+C) \|u_0 \|_{L^1} + C]+  2^{2 \alpha + 2(\alpha - \tilde{\gamma})}  T D \\ & + C (  \|u_0 \|_{L^{2 \alpha}}^{2 \alpha } +1)   +  \frac{C(2(\alpha - \tilde{\gamma})+1) }{2}  (  \|u_0 \|_{L^{2 \alpha}}^{2 \alpha } +1)
	\\ &	\le C_2  (\|u_0 \|_{L^{2 \alpha}}^{2 \alpha} +1 ).
		\end{split}
	\end{equation*}
\end{proof}

\begin{oss}
	It follows from Lemma \ref{control_kinetic_measure_v_large_new} and the definition of $n_1$ in \eqref{definition_n1_parabolic_dissipation_measure} that we can control $ \mathbbm{1}_{\abs{u} \le 2 }\nabla u^{\frac{m+ 1- \gamma}{2}} \in L^2_{\omega,t,x}$ by $u_0 \in L^1_x$ and $ \mathbbm{1}_{\abs{u} \ge 2 }\nabla u^{\frac{m+ 1}{2} + \alpha - \tilde{\gamma}} \in L^2_{\omega,t,x}$ by $u_0 \in L^{2 \alpha}_x$.
\end{oss}

\begin{lem} \label{Control_higher_powers_Lp}
	Let  $m \in (1, \infty)$ and $ \alpha \in [\frac{1}{2},1]$. Let $u: \Omega \times (0,T) \times \mathbb{T}^d \rightarrow \mathbb{R}$ be a measurable function. Then for any $r \in [\frac{1}{2}, \frac{m- 1}{2} + \alpha )$, there is a constant $C=C(m, \alpha, T,d) > 0$ such that
	\begin{equation} \label{Poincare_lemma_induction}
		\begin{split}
		\mathbb{E}	\|  u  \|_{L^{2r}([0,T] \times \mathbb{T}^d)}^{2r}    \le C 	 \left( \esssup_{t \in [0,T]} \mathbb{E}  	\|  u(t)  \|_{L^{1}( \mathbb{T}^d)}^{2r}     + \mathbb{E} \int_0^T \int_{\mathbb{T}^d} \abs{ \nabla (u(t,x))^{r}}^{2} \ dx \ dt    \right).
		\end{split}
	\end{equation}
\end{lem}
\begin{proof}
We prove in Lemma \ref{Control_higher_powers_Lp_appendix} (Appendix \ref{Appendix_Poincare}) the following version of Poincar{\'e} inequality  for fixed $\omega \in \Omega,t \in [0,T]$ and any $\underline{C}=\underline{C}(m,\alpha,d) > 0$,
\begin{equation} \label{Poincare_x_fix}
\int_{\mathbb{T}^d} \abs{u(t,x)}^{2r}  \ dx  \le \underline{C} \left[  \left(\int_{\mathbb{T}^d} \abs{u(t,x)}  dx \right)^{2r}  + \int_{\mathbb{T}^d} (\nabla (u(t,x))^r )^2  \ dx    \right].
\end{equation}
Using \eqref{Poincare_x_fix} and H{\"o}lder's inequality, we have
\begin{equation*}
		\begin{split}
&	\mathbb{E} \int_0^T \int_{\mathbb{T}^d} \abs{u(t,x)}^{2r} \ dx \ dt
	\end{split}	
\end{equation*}
\begin{equation*}
	\begin{split}
  & \le \underline{C} \mathbb{E}  \int_0^T  \left[  \left(\int_{\mathbb{T}^d} \abs{u(t,x)}  dx \right)^{2r}  + \int_{\mathbb{T}^d} (\nabla (u(t,x))^r )^2  \ dx    \right] dt
	\\ & \le  \underline{C} T  \esssup_{t \in [0,T]}  \mathbb{E} \left(\int_{\mathbb{T}^d} \abs{u(t,x)}  dx \right)^{2r}   +   \underline{C} \mathbb{E}  \int_0^T \int_{\mathbb{T}^d} (\nabla (u(t,x))^r )^2  \ dx   \ dt.
	\end{split}	
\end{equation*}
\end{proof}
For the rest of the discussion, we write the kinetic form of \eqref{degenerate_parabolic_hyperbolic_SPDE} as follows:
\begin{equation} \label{kinetic_form_Setting_q}
	\partial_t \chi  - m \abs{v}^{m-1} \Delta_x \chi = \sum_{k=1}^{\infty} \delta_{u=v} g_k \dot{\beta}_k  +  \partial_v q,
\end{equation}
where $q := n - \frac{1}{2} G^2 \delta_{u=v}$. 
		\begin{proof}[Proof of Theorem \ref{Main_theorem}]
				Let $\chi$ be the kinetic function
			corresponding to $u$ and solving \eqref{kinetic_form_Setting_q}.
Let $\Psi_0 \in C^{\infty}_c(\mathbb{R}_v)$ supported in the ball $B_2(0)$ such that $\Psi_0=1$ in  $B_1(0)$ and $\Psi_1:=1- \Psi_0$. We consider the following decompositions for small and large velocities
			\begin{equation} \label{first_decomposition_final_theorem}
				\chi = \chi\Psi_0 + \chi \Psi_1  =: \chi^{<} + \chi^{>}, \quad \text{and} \quad  q = q \Psi_0 + q \Psi_1  =: q^{<} + q^{>}.
			\end{equation}
			Then, we can write 
			\begin{equation}  \label{second_decomposition_final_theorem}
					u=: u^{<} + u^{>} =  \int_v \chi^{<} dv +  \int_v \chi^{>} dv.
			\end{equation}
In order to apply Lemma \ref{Isotropic_Averaging_Lemma}, we introduce a cut-off in time in the kinetic form \eqref{kinetic_form_Setting_q}. Let $N >0$ and $\phi \in C^{\infty}(\mathbb{R}_t)$ such that $ \phi  = 1$ for $t \in \left( \frac{1}{N}, T - \frac{1}{N} \right)$ and $ \phi = 0 $ for $t \notin \left( 0, T \right)$. Multiplying \eqref{kinetic_form_Setting_q} by $ \Psi_0$, $\Psi_1$ and $\phi$ we obtain two equations
			\begin{equation} \label{kinetic_formulation_small_v_final}
				\begin{split}
				\partial_t  (\chi^{<} \phi)   - m \abs{v}^{m-1} \Delta  (\chi^{<} \phi) & =  \sum_{k=1}^{\infty}    \delta_{u=v} g_k \Psi_0  \phi \dot{\beta}_k   \\ &  + \chi^{<} \partial_t \phi -  q \phi \partial_v \Psi_0  + \partial_v  (q^{<} \phi),
				\end{split}
			\end{equation}
		  	\begin{equation}
		  		\label{kinetic_formulation_large_v_final}
		  		\begin{split}
				\partial_t  (\chi^{>}  \phi)  - m \abs{v}^{m-1} \Delta ( \chi^{>}  \phi) & =  \sum_{k=1}^{\infty}     \delta_{u=v}  g_k  \Psi_1  \phi \dot{\beta}_k \\ &  + \chi^{>} \partial_t \phi  +  q \phi \partial_v \Psi_0+ \partial_v (q^{>} \phi). 
					\end{split}
				\end{equation}
			The proof is then divided into three steps, where in the first two we derive the estimates for $u^{<}$ using \eqref{kinetic_formulation_small_v_final} and $u^{>}$ using \eqref{kinetic_formulation_large_v_final}. The estimates for $u$ follow by combining those for  $u^{<}$  and $u^{>}$ in the final step.	
			\paragraph{Step 1} We treat $u^{<}$. Let $ \mathring{g}_k= g_k \Psi_0 \phi$,   $	h  :=   \chi^{<} \partial_t \phi  - q \phi \partial_v \Psi_0$ and $\tilde{h}  :=  q^{<} \phi $ in \eqref{kinetic_formulation_small_v_final}.
We apply Lemma \ref{Isotropic_Averaging_Lemma} to \eqref{kinetic_formulation_small_v_final} choosing $ \gamma, \rho  \in (0,1)$  close to one, $\nu \in (\frac{1}{2}, \frac{m}{2})$, $ \epsilon \in (0,1)$ and $\varpi = \frac{ \varkappa}{\varkappa -1} >1$ small enough such that $\gamma^{\star} <1 $ in \eqref{definition_gamma_star_statement} and $  \kappa_x= \frac{2}{m}$, $p=m$ in \eqref{reg_time_results_before_bootstrap}. Using the embedding contained in Lemma \ref{Embedding_time_spaces}, the estimate \eqref{Bounds_Averaging_lemma_statement} becomes
	\begin{equation} \label{small_velocities_final_first}
	\begin{split}
	&	\| u^{<} \phi  \|_{L^{m}_{\omega,t}     W^{\sigma_x,m}_x}  \\  & \lesssim \| \chi^{<} \phi  \|_{ L^{\beta}_{\omega,t,x,v}} +   \left(  \sum_{k=1}^{\infty}     \| g_k(u) \Psi_0(u)  \phi \nabla  u^{\frac{m+1}{2}- \gamma^{\star} }   \|_{L^2_{\omega,t,x}}^2 \right)^{\frac{1}{2}} 
	\\  & + \left(  \sum_{k=1}^{\infty}     \| u^{\frac{m+1}{2}- \gamma^{\star} }  \nabla (g_k(u)  \Psi_0(u) \phi )  \|_{L^2_{\omega,t,x}}^2 \right)^{\frac{1}{2}}  \\ &  +    \left( \sum_{k=1}^{\infty}  \esssup_{t \in [0,T]}    \left \|   (g_k(u) \Psi_0(u) \phi(t))^2  \right \|_{L^1_{\omega,x}}^{\frac{1}{\varpi}}   \left \|    u^{\varkappa(2 \nu - 1)}  (g_k(u)\Psi_0(u) \phi)^2   \right \|_{L^1_{\omega,t,x}}^{\frac{1}{\varkappa}} \right)^{\frac{1}{2}} 	  \\ &   + \| \abs{v}^{1- \gamma}  h  \|_{L^1_{\omega} \mathcal{M}_{\text{TV}}} + \| \abs{v}^{- \gamma} \tilde{h}  \|_{L^1_{\omega} \mathcal{M}_{\text{TV}}} + \| u^{<} \phi \|_{L^1_{\omega,t,x}}.  
		\end{split}
	\end{equation}
 Since $ \chi^{<} \phi \in L^{\beta}_{\omega,t,x,v}$ has norm bounded by 1 (for $ \beta$ large) and $\abs{u}$ can be estimated by a constant on the support of $\Psi_0$, we estimate the first and the last term on the right-hand side in \eqref{small_velocities_final_first} as follows
	\begin{equation*}
		\| \chi^{<} \phi  \|_{ L^{\beta}_{\omega,t,x,v}}  + \| u^{<} \phi \|_{L^1_{\omega,t,x}}   \lesssim 1.
	\end{equation*}
Using Assumption \ref{Assumption_diffusion_coefficient} and Lemma \ref{control_kinetic_measure_v_large_new} for the second and third term on the right-hand side in \eqref{small_velocities_final_first} lead to

\begin{equation*}
	\begin{split}
& \left(  \sum_{k=1}^{\infty}     \| g_k(u) \Psi_0(u)  \phi  \nabla  u^{\frac{m+1}{2}- \gamma^{\star} }   \|_{L^2_{\omega,t,x}}^2 \right)^{\frac{1}{2}}  +  \left(  \sum_{k=1}^{\infty}     \| u^{\frac{m+1}{2}- \gamma^{\star} }  \nabla (g_k(u)  \Psi_0(u) \phi)   \|_{L^2_{\omega,t,x}}^2 \right)^{\frac{1}{2}}  \\ & \lesssim  	\| \Psi_0(u)  \phi  \nabla u^{\frac{m}{2}+1- \gamma^{\star}}   \|_{L^2_{\omega,t,x}} \lesssim \| u_0 \|_{L^1}^{\frac{1}{2}} +1.
 \end{split}
\end{equation*} 
The fourth term on the right-hand side in \eqref{small_velocities_final_first} is controlled using Assumption \ref{Assumption_diffusion_coefficient} estimating  $ \abs{u}$ and $ \abs{u}^{ \varkappa(2 \nu- 1)+1}$  by a constant on the support of $ \Psi_0$ 
	\begin{equation*}
		\begin{split}
 \left( \sum_{k=1}^{\infty}  \esssup_{t \in [0,T]}    \left \|   (g_k(u) \Psi_0(u) \phi(t))^2  \right \|_{L^1_{\omega,x}}^{\frac{1}{\varpi}}   \left \|    u^{\varkappa(2 \nu - 1)}  (g_k(u) \Psi_0(u) \phi)^2   \right \|_{L^1_{\omega,t,x}}^{\frac{1}{\varkappa}} \right)^{\frac{1}{2}} 	   \lesssim 1.
		\end{split}
	\end{equation*}
	Next, we check that $ \abs{v}^{1- \gamma} h  \in L^1_{\omega} \mathcal{M}_{TV}$. Since $ \abs{v}^{1- \gamma}$ can be estimated by a constant on the support of $ \Psi_0$ and $ \partial_v \Psi_0$, we integrate with respect to the variable $v$ and apply  Lemma  \ref{control_kinetic_measure_v_large_new} to obtain
	\begin{equation*}
		\begin{split}
		\| \abs{v}^{1- \gamma}  h  \|_{L^1_{\omega} \mathcal{M}_{\text{TV}}} & = \|  \abs{v}^{1- \gamma} \left( \chi^{<} \partial_t \phi - \phi q \partial_v \Psi_0 \right) \|_{L^1_{\omega} \mathcal{M}_{\text{TV}}}  \\  & \lesssim \| \chi \partial_t \phi \|_{L^1_{\omega,t,x,v} } +  \| \phi\abs{v}^{1- \gamma}  q   \partial_v \Psi_0 \|_{L^1_{\omega} \mathcal{M}_{\text{TV}}} \\   & \lesssim \| \partial_t \phi \abs{u} \|_{L^1_{\omega,t,x}} +  \|u_0 \|_{L^1}  +1.  
		\end{split}
	\end{equation*} 
	Using Lemma \ref{control_kinetic_measure_v_large_new}, we have
	\begin{equation*}
		\| \abs{v}^{- \gamma} \tilde{h}  \|_{L^1_{\omega} \mathcal{M}_{\text{TV}}} = \| \abs{v}^{- \gamma} q^{<} \phi  \|_{L^1_{\omega} \mathcal{M}_{\text{TV}}} \lesssim \|u_0 \|_{L^1} +1.
	\end{equation*}
	Then in view of the estimates above, \eqref{small_velocities_final_first} becomes
	\begin{equation} \label{small_velocities_final}
			\| u^{<} \phi  \|_{L^{m}_{\omega,t}     W^{\sigma_x,m}_x}  \lesssim   \| \partial_t \phi \abs{u} \|_{L^1_{\omega,t,x}}  +    \| u_0 \|_{L^1} +1.
	\end{equation}
\paragraph{Step 2} We treat $u^{>}$. Let $\mathring{g}_k= g_k \Psi_1 \phi$,  $	h  :=    \chi^{>}  \partial_t \phi + q  \partial_v \Psi_0 $ and $ 
\tilde{h} :=  q^{>} \phi $ in \eqref{kinetic_formulation_large_v_final}. 
 We apply Lemma \ref{Isotropic_Averaging_Lemma} to \eqref{kinetic_formulation_large_v_final} with $\rho \in (0,1)$, $ \gamma >1 $ chosen close to one, $\nu \in (\frac{1}{2}, \frac{m}{2}), \epsilon \in (0,1)$ and $\varpi = \frac{ \varkappa}{\varkappa -1} >1$ small enough such that $\gamma^{\star} >1 $ close to one in \eqref{definition_gamma_star_statement} and $  \kappa_x= \frac{2}{m}$, $p=m$ in \eqref{reg_time_results_before_bootstrap}.  Using the embedding contained in Lemma \ref{Embedding_time_spaces}, the estimate \eqref{Bounds_Averaging_lemma_statement} becomes
\begin{equation} \label{large_velocities_final_first}
	\begin{split}
&	\| u^{>} \phi  \|_{L^{m}_{\omega,t}     W^{\sigma_x,m}_x}  \\ & \lesssim   \| \chi^{>} \phi  \|_{ L^{\beta}_{\omega,t,x,v}} +   \left(  \sum_{k=1}^{\infty}     \| g_k(u) \Psi_1(u) \phi   \nabla  u^{\frac{m+1}{2}- \gamma^{\star} }   \|_{L^2_{\omega,t,x}}^2 \right)^{\frac{1}{2}} 
 \\ & +   \left(  \sum_{k=1}^{\infty}     \| u^{\frac{m+1}{2}- \gamma^{\star} }    \nabla (g_k(u)   \Psi_1(u) \phi )   \|_{L^2_{\omega,t,x}}^2 \right)^{\frac{1}{2}}  \\ & + \left( \sum_{k=1}^{\infty}     \esssup_{t \in [0,T]} \left \|   ( g_k(u) \Psi_1(u) \phi(t))^2  \right \|_{L^1_{\omega,x}}^{\frac{1}{\varpi}}       \left \|     u^{\varkappa(2 \nu - 1)}  (g_k(u) \Psi_1(u) \phi)^2 \right \|_{L^1_{\omega,t,x}}^{\frac{1}{\varkappa}} \right)^{\frac{1}{2}} \\ &  + \|  \abs{v}^{1- \gamma} h  \|_{L^1_{\omega} \mathcal{M}_{\text{TV}}}     + \| \abs{v}^{- \gamma} \tilde{h}  \|_{L^1_{\omega} \mathcal{M}_{\text{TV}}} + \| u^{>} \phi \|_{L^1_{\omega,t,x}} . 
	\end{split}
\end{equation}
 Since $ \chi^{>} \phi \in L^{\beta}_{\omega,t,x,v}$ has norm bounded by 1 (for  $ \beta$ large) and using Lemma \ref{control_apriori_estimate_new}, we estimate the first and the last term on the right-hand side in \eqref{large_velocities_final_first} 
\begin{equation*}
	\| \chi^{>} \phi  \|_{ L^{\beta}_{\omega,t,x,v}}  + \| u^{>} \phi \|_{L^1_{\omega,t,x}}   \lesssim 1+ \| u_0 \|_{L^1}.
\end{equation*}
Recall that $\alpha \in [\frac{1}{2},1]$. The second and third term on the right-hand side in \eqref{large_velocities_final_first} are estimated using Assumption \ref{Assumption_diffusion_coefficient} and Lemma  \ref{control_kinetic_measure_v_large_new} 
\begin{equation} \label{estimate_final_proof_large_space_10}
\begin{split}
& \left(  \sum_{k=1}^{\infty}     \| g_k(u) \Psi_1(u) \phi  \nabla  u^{\frac{m+1}{2}- \gamma^{\star} }   \|_{L^2_{\omega,t,x}}^2 \right)^{\frac{1}{2}} +  \left(  \sum_{k=1}^{\infty}     \| u^{\frac{m+1}{2}- \gamma^{\star} }    \nabla (g_k(u)   \Psi_1(u) \phi )   \|_{L^2_{\omega,t,x}}^2 \right)^{\frac{1}{2}}  \\ & \lesssim	\| \Psi_1(u)  \phi \nabla  u ^{\frac{m+1}{2}  - \gamma^{\star} + \alpha}   \|_{L^2_{\omega,t,x}} \lesssim \| u_0 \|_{L^{2 \alpha }}^{   \alpha } +1.
	\end{split}
\end{equation}
By the assumptions above, $\varkappa $ is large and $\nu $ is close to $\frac{1}{2}$. Let $ r= \alpha + \varkappa(\nu -\frac{1}{2}) < \frac{m-1}{2} + \alpha$.  The fourth term on the right-hand side in \eqref{large_velocities_final_first} is estimated using Assumption \ref{Assumption_diffusion_coefficient}, Lemma \ref{control_apriori_estimate_new}, Lemma \ref{Control_higher_powers_Lp}, Lemma \ref{control_kinetic_measure_v_large_new}, $L^{2\alpha}(\mathbb{T}^d) \hookrightarrow L^1(\mathbb{T}^d)$ and $r \le 2 \alpha$ as follows 
\begin{equation}  \label{Estimate_large_space_reg_used_in_time_reg_final_proof}
	\begin{split}
&		\left( \sum_{k=1}^{\infty}     \esssup_{t \in [0,T]} \left \|   ( g_k(u) \Psi_1(u) \phi(t))^2  \right \|_{L^1_{\omega,x}}^{\frac{1}{\varpi}}       \left \|     u^{\varkappa(2 \nu - 1)}  (g_k(u) \Psi_1(u) \phi)^2 \right \|_{L^1_{\omega,t,x}}^{\frac{1}{\varkappa}} \right)^{\frac{1}{2}}
  \\ & \lesssim  \left( \esssup_{t \in [0,T]} \mathbb{E}   \int_x   \abs{ u(t,x)}^{2 \alpha}  \abs{ \Psi_1(u)  \phi(t)}^2    \ dx       \right)^{\frac{1}{2 \varpi}} \left( \mathbb{E} \int_{t,x}    \abs{ u(t,x)}^{2 \alpha +\varkappa(2 \nu- 1)}  \abs{ \Psi_1(u)  \phi(t)}^2   \ dx      \ dt \right)^{\frac{1}{2 \varkappa}}  \\ & \lesssim (\| u_0 \|_{L^{2 \alpha}}^{\frac{\alpha}{\varpi}}+1) \left( \mathbb{E}	\|    u  \phi   \|^{2r}_{L^{2r}_{t,x}} \right)^{\frac{1}{2\varkappa}} 
   \\ &   \lesssim (\| u_0 \|_{L^{2 \alpha}}^{\frac{\alpha}{\varpi}}+1)  \left[ \left( \esssup_{t \in [0,T]} \mathbb{E}  	\|  u(t)  \|_{L^{1}( \mathbb{T}^d)}^{2r}  \right)^{\frac{1}{2\varkappa}}   +  \left( \mathbb{E} \int_0^T \int_{\mathbb{T}^d} \abs{ \nabla (u(t,x))^{r}}^{2} \ dx \ dt   \right)^{\frac{1}{2\varkappa}} \right]
  \\ &   \lesssim (\| u_0 \|_{L^{2 \alpha}}^{\frac{\alpha}{\varpi}}+1) (\| u_0 \|_{L^1}^{\frac{r}{\varkappa}} + \|u_0 \|_{L^{2 \alpha}}^{\frac{\alpha}{\varkappa}}+1)
\\ &  \lesssim  \| u_0 \|_{L^{2 \alpha}}^{\frac{\alpha}{\varpi}} \| u_0 \|_{L^1}^{\frac{r}{\varkappa}} +  \|u_0 \|_{L^{2 \alpha}}^{\alpha} + \| u_0 \|_{L^1}^{\frac{r}{\varkappa}} +1
\\ & \lesssim \| u_0 \|_{L^{2 \alpha}}^{\alpha(\frac{1}{\varpi}+ \frac{1}{\varkappa}) + \frac{\alpha}{\varkappa}} +  \|u_0 \|_{L^{2 \alpha}}^{\alpha} + \| u_0 \|_{L^1}^{\frac{r}{\varkappa}} +1
   \lesssim  \| u_0 \|_{L^{2 \alpha}}^{2 \alpha}+1.
	\end{split}
\end{equation}
Next, we check that $  \abs{v}^{1- \gamma} h  \in L^1_{\omega} \mathcal{M}_{TV}$. Since $ \abs{v}^{1- \gamma}$ can be estimated by a constant on the support of $ \Psi_1$ and $ \partial_v  \Psi_0$, we apply Lemma  \ref{control_kinetic_measure_v_large_new} to
\begin{equation*}
	\begin{split}
	\| \abs{v}^{1 - \gamma } h  \|_{L^1_{\omega} \mathcal{M}_{\text{TV}}} & = \| \abs{v}^{1 - \gamma }(  \chi^{>} \partial_t \phi + \phi q \partial_v \Psi_0) \|_{L^1_{\omega} \mathcal{M}_{\text{TV}}}  \\ & \lesssim \| \partial_t \phi \abs{u} \|_{L^1_{\omega,t,x}} +  \| u_0 \|_{L^{1}} +1  .
	\end{split}
\end{equation*} 
Using again Lemma \ref{control_kinetic_measure_v_large_new}, we have
\begin{equation*}
	\| \abs{v}^{- \gamma} \tilde{h}  \|_{L^1_{\omega} \mathcal{M}_{\text{TV}}} = \|   \abs{v}^{- \gamma} q^{>} \phi  \|_{L^1_{\omega} \mathcal{M}_{\text{TV}}} \lesssim \| u_0\|_{L^{2 \alpha}}^{  2 \alpha}   +1.
\end{equation*}
	Then in view of the estimates above, \eqref{large_velocities_final_first} becomes
\begin{equation}  \label{large_velocities_final}
	\| u^{>} \phi  \|_{L^{m}_{\omega,t}     W^{\sigma_x,m}_x}  \lesssim \| \partial_t \phi \abs{u} \|_{L^1_{\omega,t,x}} +   \| u_0 \|_{L^{2 \alpha}}^{  2 \alpha}    +  1.
\end{equation}
\paragraph{Conclusion} 
We may set $\phi_N(t)= \psi(Nt) - \psi(Nt-NT+T)$, where $ \psi \in C^{\infty}(\mathbb{R})$ with $ 0 \le \psi \le 1$, $ \text{supp} \ \psi \subset (0, \infty)$, $ \psi(t)=1$ for $t > T$ and $ \| \partial_t \psi \|_{L^1} =1$. For $N \rightarrow \infty$, $\phi_N u \rightarrow u \mathbbm{1}_{[0,T]}$ in the sense of distributions, while $ \partial_t \phi_N $ is a smooth approximation of $ \delta_{\{t=0 \}} - \delta_{ \{t=T \}}$. Using \eqref{small_velocities_final} and \eqref{large_velocities_final}, we have
\begin{equation*}
	\begin{split}
		\sup_{N \in \mathbb{N}} \|  u_N \|_{L^{m}_{\omega,t}     W^{\sigma_x,m}_x} & \lesssim \sup_{N \in \mathbb{N}} \|  u^{<}_N \|_{L^{m}_{\omega,t}     W^{\sigma_x,m}_x} + \sup_{N \in \mathbb{N}} \|  u^{>}_N \|_{L^{m}_{\omega,t}     W^{\sigma_x, m}_x}  \\ & \lesssim 	\sup_{N \in \mathbb{N}}  \|  \abs{u}   \partial_t \phi_N \|_{L^1_{\omega,t,x} }  + \| u_0 \|_{L^{2 \alpha}}^{  2 \alpha}    +  1
		\\ & \lesssim 	\sup_{N \in \mathbb{N}} \esssup_{t \in [0,T]} \mathbb{E} \| u(t) \|_{L^1_x} \int_t \abs{ \partial_t \phi_N } \ dt + \| u_0 \|_{L^{2 \alpha}}^{  2 \alpha}    +  1.
	\end{split}
\end{equation*}
Sending $N \rightarrow \infty$ and using the weak lower semicontinuity of the norm in
 
\noindent $L^{m} \left(\Omega; L^{m}(0,T;  W^{\sigma_x,m}( \mathbb{T}^d) \right)$, we obtain \eqref{Main_theorem_estimate_statement_small_velocity}.
		\end{proof}
	\begin{proof}[Proof of Theorem \ref{Main_theorem_time}] The proof is similar to the one of Theorem \ref{Main_theorem}. Let $\chi$ be the kinetic function
	corresponding to $u$ and solving \eqref{kinetic_form_Setting_q}.  Let $\Psi_0$, $\Psi_1$, $\phi$, $\chi^{<}$, $\chi^{>}$, $q^{<}$, $q^{>}$, $u^{<}$ and $u^{>}$   be as in the proof of Theorem \ref{Main_theorem}. As above, we multiply \eqref{kinetic_form_Setting_q} by $ \Psi_0$, $\Psi_1$ and $\phi$ and we obtain \eqref{kinetic_formulation_small_v_final} and  \eqref{kinetic_formulation_large_v_final}. 
Again, we derive the estimates for $u^{<}$ using \eqref{kinetic_formulation_small_v_final} and $u^{>}$ using \eqref{kinetic_formulation_large_v_final} separately.	
	\paragraph{Step 1} We treat $u^{<}$. Let $ \mathring{g}_k := g_k \Psi_0 \phi$, $h  :=  \chi^{<} \partial_t \phi  - q \phi \partial_v \Psi_0$ and $	\tilde{h}  := \phi  q^{<} $. We apply Lemma \ref{Time_Averaging_Lemma} Part \textit{(i)} to \eqref{kinetic_formulation_small_v_final} choosing $ \gamma , \rho \in (0,1)$ close to one, $ \epsilon \in (0,1)$ and $\varpi = \frac{ \varkappa}{\varkappa -1}>1$ small enough  such that $p=1$ in \eqref{regularity_results_before_bootstrap_time_integrability},  $q=2$, $ \bar{p} \in (p,q) $ in \eqref{p_bar_q_statement_time_averaging_lemma} and $ \kappa_t= \frac{1}{2}-$  in \eqref{regularity_results_before_bootstrap_time}. Using the embedding contained in Lemma \ref{Embedding_time_spaces} in the estimate \eqref{Bounds_Averaging_lemma_statement_time}, we arrive at 
	\begin{equation} \label{small_velocities_final_first_time}
		\begin{split}
 \| u^{<} \phi \|_{L^{1}_{\omega} W^{\sigma_t,2}_t L^{1}_x}  & \lesssim     \| \chi^{<} \phi  \|_{ L^{\beta}_{\omega,t,x,v}} + \left( \sum_{k=1}^{\infty} \|      u^{1 - \gamma}   g_k(u) \Psi_0(u) \phi    \|_{L^2_{\omega,t,x}}^2 \right)^{\frac{1}{2}}  
  \\  &  +  \left( \sum_{k=1}^{\infty}    \esssup_{t \in [0,T]} \left \|   ( g_k(u) \Psi_0(u) \phi(t))^2  \right \|_{L^1_{\omega,x}}^{\frac{1}{\varpi}}      \left \|    u^{2 \varkappa(1- \gamma)}  (g_k(u) \Psi_0(u) \phi)^2  \right \|_{L^1_{\omega,t,x}}^{\frac{1}{\varkappa}} \right)^{\frac{1}{2}}  \\  &  +\| \abs{v}^{1- \gamma}  h \|_{L^1_{\omega} \mathcal{M}_{\text{TV}}}  + \| \abs{v}^{- \gamma} \tilde{h}  \|_{L^1_{\omega} \mathcal{M}_{\text{TV}}}     + \| u^{<} \phi  \|_{L^{\bar{p}}_{\omega} L^q_t L^p_x}. 
\end{split}
	\end{equation}
The first, third, fourth and fifth term on the right-hand side above are estimated as in Step 1 in the proof of Theorem \ref{Main_theorem}. The second term on the right-hand side in \eqref{small_velocities_final_first_time} is treated using Assumption \ref{Assumption_diffusion_coefficient} and estimating $ \abs{u}^{3 - 2 \gamma}$ by a constant on the support of $ \Psi_0$ 
	\begin{equation*}
		\begin{split}
&	\left( \sum_{k=1}^{\infty} \mathbb{E} \int_{t,x}   \abs{u(t,x)}^{2(1  -\gamma)}     \abs{g_k(x,u(t,x)) \Psi_0(u(t,x)) \phi(t)}^2     \ dx      \ dt \right)^{\frac{1}{2}}   \lesssim 1.
		\end{split}
\end{equation*}
The sixth term on the right-hand side in \eqref{small_velocities_final_first_time} is treated estimating $ \abs{u}$ by a constant on the support of $ \Psi_0$  as follows
\begin{equation*}
	\| u^{<} \phi  \|_{L^{\bar{p}}_{\omega} L^q_t L^p_x} \lesssim 	    1.
\end{equation*}
	Then in view of the arguments above, \eqref{small_velocities_final_first_time} becomes
	\begin{equation} \label{small_velocities_final_time}
		\| u^{<} \phi  \|_{L^{1}_{\omega}     W^{\sigma_t,2}_t L^{1}_x} \lesssim  \| \partial_t \phi \abs{u} \|_{L^1_{\omega,t,x}} +   \|u_0 \|_{L^{1}} +1.
	\end{equation}
	
		\paragraph{Step 2} We treat $u^{>}$. Let $ \mathring{g}_k := g_k \Psi_1 \phi$, $h  := \chi^{>}  \partial_t \phi + q  \partial_v \Psi_0$ and $\tilde{h}  :=  q^{>} \phi $.	We apply Lemma \ref{Time_Averaging_Lemma} Part \textit{(ii)} to \eqref{kinetic_formulation_large_v_final} with $\epsilon, \bar{\epsilon} \in (0,1)$ small enough, $\rho \in (0,1) $, $\tilde{\gamma} \in (1,1+ \bar{\epsilon}(m-1))$ close to one, $ \zeta \in (0, \frac{1}{2} -\bar{\epsilon}), \nu \in (\frac{1}{2}, \infty)$ close to one half and $\varpi = \frac{ \varkappa}{\varkappa -1}>1$ small enough such that 
$ \kappa_t := \zeta - \epsilon= \frac{1}{2}-$. Then using the embedding contained in Lemma \ref{Embedding_time_spaces}  in the estimate \eqref{final_cut_off_time_large_statement}, we arrive at 
	\begin{equation} \label{large_velocities_final_first_time}
		\begin{split}
	&	\| u^{>}  \phi \|_{L^{1}_{\omega}     W^{\sigma_t,2}_t L^{1}_x}  \\ &  \lesssim       \left( \sum_{k=1}^{\infty} \|      u^{\nu - \frac{1}{2}}   g_k(u) \Psi_1(u) \phi  \|_{L^2_{\omega,t,x}}^2 \right)^{\frac{1}{2}}    \\ & + \left( \sum_{k=1}^{\infty}     \esssup_{t \in [0,T]} \left \|   ( g_k(u) \Psi_1(u) \phi(t))^2  \right \|_{L^1_{\omega,x}}^{\frac{1}{\varpi}}       \left \|     u^{\varkappa(2 \nu - 1)}  (g_k(u) \Psi_1(u) \phi)^2 \right \|_{L^1_{\omega,t,x}}^{\frac{1}{\varkappa}} \right)^{\frac{1}{2}}  \\ & + \|   h  \|_{L^1_{\omega} \mathcal{M}_{\text{TV}}} + \| \abs{v}^{- \tilde{\gamma}} \tilde{h}  \|_{L^1_{\omega} \mathcal{M}_{\text{TV}}} +  \| u^{>} \phi \|_{L^1_{\omega} L^2_t L^1_x}. 
		\end{split}
	\end{equation}
The second and fourth term on the right-hand side above are estimated as in Step 2 of proof of Theorem \ref{Main_theorem}. Recall that $\alpha \in [\frac{1}{2},1]$. The first term on the right-hand side in \eqref{large_velocities_final_first_time} is treated using Assumption \ref{Assumption_diffusion_coefficient}, H{\"o}lder's inequality, Jensen's inequality and estimate \eqref{Estimate_large_space_reg_used_in_time_reg_final_proof}  in Step 2 of proof of Theorem \ref{Main_theorem} 
	\begin{equation*} 
	\begin{split}
		&	  \left( \sum_{k=1}^{\infty} \mathbb{E} \int_{t,x}  \abs{u(t,x)}^{2 \nu -1}    \abs{g_k(x,u(t,x)) \Psi_1(u(t,x)) \phi(t)}^2      \ dx      \ dt \right)^{\frac{1}{2}}  \\ & \lesssim   \left( \mathbb{E} \int_0^T \int_{x}   \abs{u(t,x)}^{ 2 \nu -1 + 2 \alpha (\frac{1}{\varpi} + \frac{1}{\varkappa})}        \ dx      \ dt \right)^{\frac{1}{2}} 
		\\ & \le \left( \int_0^T \mathbb{E}  \left( \int_{x}   \abs{u(t,x)}^{ 2 \alpha}        \ dx   \right)^{\frac{1}{\varpi}} \left( \int_x \abs{u(t,x)}^{\varkappa (2 \nu -1) + 2 \alpha}  \ dx \right)^{\frac{1}{\varkappa}} \ dt \right)^{\frac{1}{2}} 
	\\ & \le 	 \left( \esssup_{t \in [0,T]} \mathbb{E}   \int_x   \abs{ u(t,x)}^{2 \alpha}    \ dx       \right)^{\frac{1}{2 \varpi}} \left( \mathbb{E} \int_0^T \int_{x}    \abs{ u(t,x)}^{2 \alpha +\varkappa(2 \nu- 1)}   \ dx      \ dt \right)^{\frac{1}{2 \varkappa}}
		\\ & \lesssim   \| u_0 \|_{L^{2 \alpha }}^{ 2  \alpha }  + 1.
	\end{split}
\end{equation*}
Next, we check that $   h  \in L^1_{\omega} \mathcal{M}_{TV}$. We apply Lemma  \ref{control_kinetic_measure_v_large_new} to
\begin{equation*}
	\begin{split}
		\|  h  \|_{L^1_{\omega} \mathcal{M}_{\text{TV}}} & = \| 
		  \chi^{>} \partial_t \phi + \phi   q  \partial_v \Psi_0 \|_{L^1_{\omega} \mathcal{M}_{\text{TV}}} \lesssim \| \partial_t \phi \abs{u} \|_{L^1_{\omega,t,x}} +  \| u_0 \|_{L^{1}}  +1 .
	\end{split}
\end{equation*} 
The fifth term on the right-hand side in \eqref{large_velocities_final_first_time} is treated using Lemma \ref{control_apriori_estimate_new}
\begin{equation*}
	 \| u^{>} \phi \|_{L^1_{\omega} L^2_t L^1_x}  \lesssim  \| u_0 \|_{L^{1}}  +1 .
\end{equation*}
	Then in view of the arguments above, \eqref{large_velocities_final_first_time} becomes
	\begin{equation} \label{large_velocities_final_time}
		\|  u^{>}  \phi  \|_{L^{1}_{\omega}     W^{\sigma_t,2}_t L^{1}_x} \lesssim  \| \partial_t \phi \abs{u} \|_{L^1_{\omega,t,x}} + \| u_0 \|_{L^{2 \alpha}}^{  2 \alpha} +1.
	\end{equation}
\paragraph{Conclusion} Let $\phi_N$ and $u_N$ be as in the proof of Theorem \ref{Main_theorem}. Using \eqref{small_velocities_final_time} and \eqref{large_velocities_final_time}, we have
\begin{equation*}
	\begin{split}
		\sup_{N \in \mathbb{N}} \|  u_N \|_{L^{1}_{\omega} W^{\sigma_t,2}_t L^1_x} & \lesssim \sup_{N \in \mathbb{N}} \|  u^{<}_N \|_{L^{1}_{\omega} W^{\sigma_t,2}_t L^1_x} + \sup_{N \in \mathbb{N}} \|  u^{>}_N \|_{L^{1}_{\omega} W^{\sigma_t,2}_t L^1_x}  \\ & \lesssim  	\sup_{N \in \mathbb{N}} \|  \abs{u}   \partial_t \phi_N \|_{L^1_{\omega,t,x} }  + \| u_0 \|_{L^{2 \alpha}}^{  2 \alpha}    +  1
	\\ & \lesssim	\sup_{N \in \mathbb{N}} \esssup_{t \in [0,T]} \mathbb{E} \| u(t) \|_{L^1_x} \int_t \abs{ \partial_t \phi_N } \ dt  + \| u_0 \|_{L^{2 \alpha}}^{  2 \alpha}    +  1.
	\end{split}
\end{equation*}
Sending $N \rightarrow \infty$ and using the weak lower semicontinuity of the norm in

\noindent $ L^{1} \left(\Omega; W^{\sigma_t,2} (0,T; L^{1} ( \mathbb{T}^d) \right)$, we obtain \eqref{Main_theorem_estimate_statement_time_small_v}.
\end{proof}

\begin{proof}[Proof of Corollary \ref{Space_time_corollary}]
Let $\sigma_x$, $\tilde{\sigma}_x \in [0,\frac{2}{m})$ be such that $ \sigma_x < \tilde{\sigma}_x $ and let $	\sigma_t \in \left[0, \frac{1}{2} \right)$. From Theorem \ref{Main_theorem} and Theorem \ref{Main_theorem_time}, we know that
\begin{equation*}
	u \in L^{ m}(\Omega; L^{ m}(0,T;W^{\tilde{\sigma}_x,  m}(\mathbb{T}^d))) \ \cap L^1(\Omega; W^{ \sigma_t,2}(0,T;L^{1}(\mathbb{T}^d))).
\end{equation*}
Recall that $\theta \in (0,1)$, $	\frac{1}{p}= \frac{1- \theta}{m} + \theta$ and $\frac{1}{q}= \frac{1-\theta}{m} + \frac{\theta}{2}$. We apply a series of real interpolation arguments below. We use  \cite[Theorem 1.3.3 (a)]{Triebel78Interpolation} in the first embedding, \cite[Theorem 1.18.4]{Triebel78Interpolation} in the first equality, \cite[Theorem 1.3.3 (d)]{Triebel78Interpolation} in the second embedding, \cite[Theorem 3.1]{amann2000compact} in the second and last equality, \cite[Theorem 3.4.1 (a), Corollary 3.8.2]{bergh2012interpolation} in the last embedding to obtain for all $\epsilon \in (0,1-\theta)$
\begin{equation*}
	\begin{split}
&	 L^{ m}(\Omega; L^{ m}(0,T;W^{\tilde{\sigma}_x,  m}(\mathbb{T}^d))) \ \cap L^1(\Omega; W^{ \sigma_t,2}(0,T;L^{1}(\mathbb{T}^d)))
	 \\ & \hookrightarrow \left( L^{ m}(\Omega; L^{ m}(0,T;W^{\tilde{\sigma}_x,  m}(\mathbb{T}^d))),  L^1(\Omega; W^{ \sigma_t,2}(0,T;L^{1}(\mathbb{T}^d))) \right)_{\theta,p}
	 \\ & = L^{ p} \left(\Omega; (  L^{ m}(0,T;W^{\tilde{\sigma}_x,  m}(\mathbb{T}^d))), W^{ \sigma_t,2}(0,T;L^{1}(\mathbb{T}^d)) )_{\theta,p} \right) 
	 \\ &  \hookrightarrow L^{ p} \left(\Omega; (  L^{ m}(0,T;W^{\tilde{\sigma}_x,  m}(\mathbb{T}^d))), W^{ \sigma_t,2}(0,T;L^{1}(\mathbb{T}^d)) )_{\theta,q} \right) 
	 \\ & = 	 L^{ p} \left(\Omega;   W^{\theta \sigma_t, q}(0,T;(W^{\tilde{\sigma}_x,  m}(\mathbb{T}^d), L^{1}(\mathbb{T}^d))_{\theta,q})  \right) 
	  	  \\ & \hookrightarrow 	 L^{ p} \left(\Omega;   W^{\theta \sigma_t, q}(0,T;(L^{1}(\mathbb{T}^d), W^{\tilde{\sigma}_x,  m}(\mathbb{T}^d))_{1-\theta - \epsilon,p})  \right) 
	  	   	  \\ & = 	 L^{ p} \left(\Omega;   W^{\theta \sigma_t, q}(0,T; W^{(1-\theta - \epsilon) \tilde{\sigma}_x,  p}(\mathbb{T}^d))  \right). 
	\end{split}
\end{equation*}
We choose $\epsilon$ small enough such that  $(1-\theta - \epsilon)\tilde{\sigma}_x=  (1-\theta)\sigma_x$. From the estimates \eqref{Main_theorem_estimate_statement_small_velocity} and \eqref{Main_theorem_estimate_statement_time_small_v} of Theorem \ref{Main_theorem} and Theorem \ref{Main_theorem_time}  and the embeddings above, we have
\begin{equation*}
	\begin{split}
	\| u \|_{L^{ p} \left(\Omega;   W^{\theta \sigma_t, q}(0,T; W^{(1-\theta ) \sigma_x,  p}(\mathbb{T}^d))  \right)} & \lesssim   ( \| u_0 \|_{ L^{2 \alpha}}^{  2 \alpha}  + 1)^{\theta}( \| u_0 \|_{ L^{2 \alpha}}^{  2 \alpha}  + 1)^{1-\theta} \\ &  \lesssim    \| u_0 \|_{ L^{2 \alpha}}^{  2 \alpha}  + 1.
	\end{split}
\end{equation*}
\end{proof}

	\appendix
	
	\section{$L^1$-based averaging techniques and the stochastic integral} \label{averaging_techniques_and_the_stochastic_integral_appendix}

In this section we show the incompatibility  of the $L^1$-based averaging techniques used in the deterministic setting (i.e. when $g_k=0$ in \eqref{degenerate_parabolic_hyperbolic_SPDE}) with the presence of stochastic integral. In \cite{gess2018well}, the stochastic forcing was treated by means of an extension of the deterministic $L^1_{t,x}$-based 
	techniques using the distributional inequality \eqref{ditribution_equality_noise_coefficients_introduction} in the kinetic form \eqref{kinetic_form_introduction_stochastic} as follows
		\begin{equation*}
			\begin{split}
\mathcal{L}(\partial_t, \nabla_x,v) \chi & :=		\partial_t \chi  - m \abs{v}^{m-1} \Delta_x \chi \\ &  =\sum_{k=1}^{\infty}  \chi \partial_v g_k \dot{\beta}_k  - \sum_{k=1}^{\infty} \partial_v (\chi g_k) \dot{\beta}_k + \partial_{v}(-\frac{1}{2}G^{2}\delta_{u=v}+n),
\end{split}
	\end{equation*}
This implies a representation of the solution $u$ in terms of 
		\begin{equation*}
	\begin{split}
u(t,x)=	\int_v \chi(t,x,v) \ dv &   = \int_v e^{-t \mathcal{L}( \nabla_x,v) } \chi(0,x,v) \ dv  \\ &  + \sum_{k=1}^{\infty} \int_v  \int_0^t e^{-(t-s) \mathcal{L}( \nabla_x,v) }  \chi(s,x,v) \partial_v g_k(x,v) d\beta_k(s) \ dv 	   \\ &  + \sum_{k=1}^{\infty} \int_v  \int_0^t e^{-(t-s) \mathcal{L}( \nabla_x,v) } \partial_v (-\chi(s,x,v) g_k(x,v)) d\beta_k(s) \ dv \\ &  + \int_v \frac{1}{\mathcal{L}(\partial_t, \nabla_x,v)}\partial_{v}(-\frac{1}{2}G^{2}\delta_{u=v}+n) \ dv ,
	\end{split}
\end{equation*}
where $\mathcal{L}( \nabla_x,v)$ is identified with the linear symbol $	\mathcal{L}( \xi, v ) :=  m \abs{v}^{m-1} 4 \pi^2 \abs{\xi}^2$.
The deterministic analysis for the regularity results in \cite{gess2017sobolev,gess2019optimal} rely on multiplier estimates parametrized in velocity, that is on estimates of the type for $ \tilde{\alpha} \in [0,1)$ and all $v \in \mathbb{R}$
\begin{equation} \label{velocity_multiplier_estimates_appendix}
	\|(-\Delta_x)^{\frac{\tilde{\alpha} }{2}}\mathcal{L}_{v}(\nabla_{x},v)(t-s)e^{-(t-s)\mathcal{L}(\nabla_{x},v)}\|_{L^{2}\to L^{2}}\le C(t-s,v,m, \tilde{\alpha}),
\end{equation}
where $\mathcal{L}_{v}(\nabla_{x},v)$ denotes the $v$-derivative of $\mathcal{L}(\nabla_{x},v)$.

We present an informal argument to show that the techniques used in the deterministic setting fail to produce optimal estimates in the case of a stochastic forcing term. Roughly speaking, the aim would be to derive an estimate on
\begin{equation} \label{estimate_appendix_deterministic}
	\mathbb{E} \| (-\Delta_x)^{\tilde{\alpha} } \sum_{k=1}^{\infty} \int_v  \int_{0}^{\cdot}e^{-(\cdot-s)\mathcal{L}(\nabla_{x},v)}\,\partial_{v} (\chi (s,\cdot,v) g_k(\cdot,v)) d\beta_{k}(s)  \ dv \|_{L_{t,x}^{1}},
\end{equation}	
for $\tilde{\alpha} $ as large as possible. As usual in the theory of stochastic PDEs \cite{da2014stochastic}, one may at best expect an improvement of spatial regularity of one order from the linear/nonlinear heat equation due to the non-vanishing quadratic variation of Brownian motion. Hence, there would be no hope to get anywhere near optimal estimates without using a \textit{priori} regularity of $\chi$.

One might then be tempted to try arguing via bootstrapping, as it is often done in the context of averaging lemmata for scalar conservation laws \cite{lions1994kinetic}. However, this technique is again incompatible with the quadratic structure of a stochastic integral. We sketch an informal estimate which highlights this issue with the stochastic integral. In the following, we assume that $\chi$ is compactly supported in $v$. Informally, the term \eqref{estimate_appendix_deterministic} can be treated using an integration by parts in $v$, Burkholder--Davis--Gundy inequality, Jensen's inequality and the multiplier estimate \eqref{velocity_multiplier_estimates_appendix}
\begin{equation} \label{estimate_appendix_deterministic_informal_argument}
	\begin{split}
&	\mathbb{E} \| (-\Delta_x)^{\tilde{\alpha}} \sum_{k=1}^{\infty} \int_v \int_{0}^{\cdot}e^{-(\cdot-s)\mathcal{L}(\nabla_{x},v)}\,\partial_{v} (\chi (s,\cdot,v) g_k(\cdot,v)) d\beta_{k}(s) \ dv  \|_{L_{t,x}^{1}} 
\\		& \lesssim \int_{0}^{T}\int_{x} \Bigg(\sum_{k=1}^{\infty}	\mathbb{E}\int_{0}^{t} \Bigg| \int_{v} (-\Delta_x)^{\tilde{\alpha}}\mathcal{L}_{v}(\nabla_{x},v)(t-s)e^{-(t-s)\mathcal{L}(\nabla_{x},v)} \\ & \times \chi(s,x,v) g_k(x,v) \ dv \Bigg|^{2}ds \Bigg)^{1/2} dx \ dt 
	 \\  & \le \Bigg(  \int_{0}^{T}  \sum_{k=1}^{\infty}	\mathbb{E}\int_{0}^{t} \int_{v}   \| (-\Delta_x)^{\frac{\tilde{\alpha}}{2}} \mathcal{L}_{v}(\nabla_{x},v)(t-s)e^{-(t-s)\mathcal{L}(\nabla_{x},v)} \\ & \times  (-\Delta_x)^{\frac{\tilde{\alpha}}{2}} \chi(s,\cdot,v) g_k(\cdot,v) \|_{L^2_x}^2   dv  \ ds \   dt \Bigg)^{1/2}
					\\   & \le \left( \int_{0}^{T} \sum_{k=1}^{\infty}  	\mathbb{E}\int_{0}^{t} \int_{v}   C(t-s,v,m, \tilde{\alpha})  \|  (-\Delta_x)^{\frac{\tilde{\alpha}}{2}} \chi(s,\cdot,v) g_k(\cdot,v) \|_{L^2_x}^2  \ dv  \ ds   \  dt \right)^{1/2}.
	\end{split}
\end{equation}	
The incompatibility of the usual $L_{t,x}^{1}$-based arguments with the $L_{t,x}^{2}$-based nature of the stochastic integral becomes apparent, since the term $(-\Delta_x)^{\frac{\tilde{\alpha}}{2}} \chi$  on the right-hand side of \eqref{estimate_appendix_deterministic_informal_argument} is not in $L^2_{x,v}$ when $\tilde{\alpha} \ge \frac{1}{2}$. Therefore, one cannot obtain more than one derivative with these bootstrapping arguments. For this reason, this approach cannot be used to prove the optimal regularity result of Theorem \ref{Main_theorem}.

	
	\section{Poincar{\'e} inequality} \label{Appendix_Poincare}
	
	\begin{lem} \label{Control_higher_powers_Lp_appendix}
		Let $x \in \mathbb{T}^d$ for $d \ge 1$. For any $r \ge \frac{1}{2} $, there is a constant $\tilde{C}=\tilde{C}(r,d) > 0$ such that
		\begin{equation} \label{Poincare_x_fix_appendix}
	\int_{\mathbb{T}^d} \abs{u(x)}^{2r}  \ dx  \le \tilde{C} \left[  \left(\int_{\mathbb{T}^d} \abs{u(x)}  dx \right)^{2r}  + \int_{\mathbb{T}^d} (\nabla (u(x))^r )^2  \ dx    \right].
\end{equation}
	\end{lem}
	\begin{proof}
We argue by contradiction. Assume that the estimate \eqref{Poincare_x_fix_appendix} is false. Then there exists an $r$ so that for all $\tilde{C}$, there is a function $u_{\tilde{C}}$ satisfying
	\begin{equation*}
		\int_{\mathbb{T}^d} |u_{\tilde{C}}(x)|^{2r}dx> \tilde{C} \left[\left(\int_{\mathbb{T}^d} |u_{\tilde{C}}(x)| \ dx \right)^{2r}+\int_{\mathbb{T}^d} (\nabla u_{\tilde{C}}^{r}(x))^{2} \ dx \right].
	\end{equation*}
	By homogeneity, we can assume $\int_{\mathbb{T}^d} |u_{\tilde{C}}(x)|^{2r}dx=1$. Then,
	\begin{equation}
		\frac{1}{\tilde{C}}> \left[ \left(\int_{\mathbb{T}^d} |u_{\tilde{C}}(x)| \ dx \right)^{2r}+\int_{\mathbb{T}^d} (\nabla u_{\tilde{C}}^{r}(x))^{2} \ dx \right].\label{eq:1}
	\end{equation}
Letting $\tilde{C}\to\infty$, we have
	\begin{align}
		\int_{\mathbb{T}^d} |u_{\tilde{C}}^{r}(x)|^{2}dx & =1, \label{eq:2}
			\end{align}
		and
		\begin{align*}
			\int_{\mathbb{T}^d} |\nabla u_{\tilde{C}}^{r}(x)|^{2}dx  & \to0. 
	\end{align*}
	Thus, the $u_{\tilde{C}}^{r}$ are uniformly bounded in $H^{1}$ and there is a convergent subsequence $u_{\tilde{C}}^{r}\to U$ in $L^{2}$. 
	Taking limit in (\ref{eq:2}) we get $\int_{\mathbb{T}^d} |U(x)|^{2}dx  =1$ and
	taking limit in (\ref{eq:1}) and using the lower semicontinuity of the second term on right-hand side, we have
	\[
	0\ge\int_{\mathbb{T}^d}(\nabla U(x))^{2} \ dx.
	\]
Thus, $U$ is constant $ U \equiv | \mathbb{T}^d|^{-\frac12}$.
	On the other hand, taking a subsequence, we have $u_{\tilde{C}}^{r}\to U$ almost everywhere and so $|u_{\tilde{C}}|\to |U|^{1/r}$.  Using Fatou's lemma, we have
	\[
	\left(\int_{\mathbb{T}^d} |U(x)|^{1/r} \ dx \right)^{2r}\le \left(\liminf_{\tilde{C} \rightarrow \infty }\int_{\mathbb{T}^d}|u_{\tilde{C}}(x)| \ dx \right)^{2r}\le\liminf_{\tilde{C} \rightarrow \infty}\frac{1}{\tilde{C}}=0.
	\]
	Hence, $U=0$, leading to a contradiction. 
\end{proof}
	\section{Optimal time integrability and scaling} \label{Scaling_appendix}

In this section we present a scaling argument that suggests the optimal  time integrability of solutions of porous medium equations. Consider 
\begin{equation} \label{scaling_section_PME_integrability}
	\begin{split}
		\partial_t u &= \Delta (|u|^{m-1} u) \quad \text{on} \ (0,T) \times \mathbb{T}^d_x, \\
		u(0) & =u_0  \qquad  \qquad  \ \ \ \  \text{on} \ \mathbb{T}^d_x,
	\end{split}
\end{equation}
with $u_0 \in L^{1}(\mathbb{T}^d_x)$ and $m>1$. 
For $T>0$, $ 1 \le p < \infty$, $s \in (0, \infty) \setminus \mathbb{N}$, $f \in W^{\lfloor s \rfloor,1}_{\text{loc}}(\mathbb{R}; X)$ and $\theta = s - \lfloor s \rfloor \in (0,1)$, recall the definition of the homogeneous Slobodeckij seminorm 
\begin{equation} \label{Slobodeckji_seminorm}
	\| f \|_{\dot{W}^{s,p}(0,T;X)} := \left(  \int_{[0,T] \times [0,T]}   \frac{\| D^{\lfloor s \rfloor} f(t) - D^{\lfloor s \rfloor} f(z) \|_{X}}{|t-z|^{\theta p +1}}  \ dz \ dt \  \right)^{\frac{1}{p}} < \infty, 
\end{equation}
with the usual modification in the case of $p=\infty$. 

\begin{lem} \label{Scaling_lemma_time_integrability_appendix}
	Let $T>0$, $m \in (1, \infty)$, $p \ge 1$ and $\sigma_t \in(0, \frac{1}{2})$. 	Assume that there is a constant $C \ge 0$ such that
	\begin{equation} \label{energy_estimates_without_forcing}
		\| u \|_{\dot{W}^{\sigma_t,p}(0,T ; L^1(\mathbb{T}^d_x))} \le C \|u_0\|_{L^{1}(\mathbb{T}^d_x)},
	\end{equation} 
	for all solutions $u$ to \eqref{scaling_section_PME_integrability}. Then necessarily $p \le \frac{1}{\sigma_t}$.
	\begin{proof}
		Given a solution $u$ to \eqref{scaling_section_PME_integrability}, for every $\eta \ge 1$, also $\tilde{u}(t,x):=  \eta u( \eta^{m-1} t,x) $ is a solution to  \eqref{scaling_section_PME_integrability}. Thus, $\tilde{u}$ satisfies \eqref{energy_estimates_without_forcing} i.e.
		\begin{equation} \label{energy_estimates_without_forcing_proof_scaled}
			\| \tilde{u} \|_{\dot{W}^{\sigma_t,p}(0,T ; L^1(\mathbb{T}^d_x))} \le C \| \tilde{u}_0\|_{L^{1}(\mathbb{T}^d_x)}.
		\end{equation}
		We observe  
		\begin{equation*}
			\| \tilde{u} \|_{\dot{W}^{\sigma_t,p}(0,T ; L^1(\mathbb{T}^d_x))} 
			=  \eta^{1+(m-1)(\sigma_t  - \frac{1}{p})}	\| u \|_{\dot{W}^{\sigma_t,p}(0,\eta^{m-1} T ; L^1(\mathbb{T}^d_x))},
		\end{equation*}
		and
		\begin{equation*}
			\|	\tilde{u}_0 \|_{L^{1 }(\mathbb{T}^d_x)} = \eta \|	u_0 \|_{L^{1 }(\mathbb{T}^d_x)}.
		\end{equation*}
		It follows from \eqref{energy_estimates_without_forcing_proof_scaled} that
		\begin{equation*} 
			\| u \|_{\dot{W}^{\sigma_t,p}(0,\eta^{m-1} T ; L^1(\mathbb{T}^d_x))} \le C \eta^{-(m-1)(\sigma_t  - \frac{1}{p})} \| u_0\|_{L^{1}(\mathbb{T}^d_x)}.
		\end{equation*}
		Letting $\eta \rightarrow \infty$, this leads to a contradiction (for non-trivial $u_0$) unless 
		\begin{equation*} 
			p \le \frac{1}{\sigma_t}.
		\end{equation*}
	\end{proof}
\end{lem}

	\paragraph*{Acknowledgments.} SB is supported by a scholarship from the EPSRC Centre for Doctoral Training in Statistical Applied Mathematics at Bath (SAMBa), under the project EP/L015684/1. 
	BG acknowledges support by the Max Planck Society through the Max Planck Research Group \textit{Stochastic partial differential equations}. This work was funded by the Deutsche Forschungsgemeinschaft (DFG, German Research Foundation) - SFB 1283/2 2021 - 317210226. HW is supported by the Royal Society through the University Research Fellowship UF140187  and by the Leverhulme Trust through a Philip Leverhulme Prize. SB, BG and HW thank the Isaac Newton
	Institute for Mathematical Sciences for hospitality during the programme \textit{Scaling limits, rough paths, quantum field theory}, which was supported by EPSRC Grant No. EP/R014604/1.

	\bibliographystyle{plain}
	\bibliography{References_paper}

\end{document}